\documentclass[a4paper,11pt]{article}

\pdfoutput=1

\usepackage{epic}
\usepackage{eepic}
\usepackage{mathtools}
\usepackage{ mathrsfs }
\usepackage{amsmath}
\numberwithin{equation}{section}
\usepackage{tikz-cd}
\usepackage{amsthm}
\usepackage{amssymb}
\usepackage{enumerate}
\usepackage[margin=2.5cm]{geometry}
\usepackage{url}

\let\OLDthebibliography\thebibliography
\renewcommand\thebibliography[1]{
	\OLDthebibliography{#1}
	\setlength{\itemsep}{0pt}
}

\newcommand{\C}{\mathbb C}
\newcommand{\Q}{\mathbb Q}
\newcommand{\R}{\mathbb R}
\newcommand{\Z}{\mathbb Z}

\newcommand{\A}{\mathbb A}
\newcommand{\Imp}{\operatorname{Im}}
\newcommand{\Rep}{\operatorname{Re}}

\newcommand{\Aut}{\operatorname{Aut}}

\newcommand{\Iso}{\operatorname{Iso}}

\newcommand{\sign}{\operatorname{sign}}
\newcommand{\sgn}{\operatorname{sgn}}

\newcommand{\tr}{\operatorname{tr}}

\newcommand{\GL}{\operatorname{GL}}
\newcommand{\q}{\operatorname{q}}
\newcommand{\image}{\operatorname{im}}
\newcommand{\spann}{\operatorname{span}}

\newcommand{\Ort}{\operatorname{O}}

\newcommand{\leg}[2]{\left(\frac{#1}{#2}\right)}
\newcommand{\matr}[4]{\left(\begin{smallmatrix}
		#1 & #2 \\
		#3 & #4
	\end{smallmatrix}\right)}
\newcommand{\sumstack}[1]{\sum_{\substack{#1}}}

\newcommand{\SL}{\mathrm{SL}_2(\Z)}
\newcommand{\SLn}[1]{\mathrm{SL}_2(\Z/#1\Z)}
\newcommand{\Spn}[1]{\mathrm{Sp}_{#1}(\Z)}
\newcommand{\Mp}{\mathrm{Mp}_2(\Z)}

\newcommand{\g}{n}
\newcommand{\n}{l}
\newcommand{\h}{h}
\newcommand{\D}{\mathcal{D}}

\newcommand{\mainTheorem}{Let $\D$ be a discriminant form of even signature $\sign(\mathcal{D})$ and $m$ a positive integer such that $m = \sign(\D)\bmod 8$, $m>\text{$p$-rank}(\D)$ for all primes $p$ and $m>6$. Then there are positive-definite even lattices $L$ such that $L'/L\cong\D$, i.e.\ the genus $I\!I_{m,0}(\D)$ is non-empty. Suppose for any $L$ of genus $I\!I_{m,0}(\D)$ the $\Z_p$-lattice $L_p = L\otimes_\Z\Z_p$ splits a hyperbolic plane over $\Z_p$. Then
	\[ \mathrm{S}_k(\D)\subset\Theta_{m,k}(\D) \]
for all $k\geq m/2$.}

\newcommand{\mainTheoremCorollary}{Let $\mathcal{D}$ be a discriminant form of even signature $\sign(\mathcal{D})$. Let $m\in\Z_{>0}$ with $m = \sign(\mathcal{D})\bmod 8$ and $m\geq10$. Then
	\[ \mathrm{S}_k(\mathcal{D})\subset\Theta_{m,k}^\uparrow(\mathcal{D}) \]
	for all $k\geq m/2$.}

\newtheorem{thm}{Theorem}[section]
\newtheorem{prp}[thm]{Proposition}
\newtheorem{cor}[thm]{Corollary}
\newtheorem{lem}[thm]{Lemma}

\newtheorem{defi}[thm]{Definition}

\title{The basis problem for modular forms for the Weil representation}
\author{Manuel Karl-Heinz M\"uller}

\begin{document}

\begin{center}
{\Large\bf The basis problem for modular forms\\[2mm]
  for the Weil representation}\\[10mm]
Manuel K.-H.\ M\"{u}ller,\\
Technische Universit\"at Darmstadt, Darmstadt, Germany,\\
mmueller@mathematik.tu-darmstadt.de
\end{center}

\vspace*{1.2cm}
\begin{center}
	\begin{minipage}{.8\textwidth}
		\noindent
		{\small The vector valued theta series of a positive-definite even lattice is a modular form for the Weil representation of $\SL$. We show that the space of cusp forms for the Weil representation is generated by such functions. This gives a positive answer to Eichler's basis problem in this case. As applications we derive Waldspurger's result on the basis problem for scalar valued modular forms and give a new proof of the surjectivity of the Borcherds lift based on the analysis of local Picard groups.}
	\end{minipage}
\end{center}

\vspace*{1.2cm}
\begin{tabular}{rl}
1  & Introduction \\
2  & The Weil representation \\
3  & Vector valued Hecke operators \\
4  & Vector valued Eisenstein series \\
5  & The vector valued Siegel--Weil formula \\
6  & The space of theta series \\
7  & Applications
\end{tabular}
\vspace*{10mm}

\section{Introduction}\label{sec:Introduction}

Let $L$ be a positive-definite even lattice of even rank $m$ with bilinear form $(\cdot,\cdot)$. It is well-known that the theta series $\theta(\tau) = \sum_{\lambda\in L}e^{\pi i(\lambda,\lambda)\tau}$ of $L$ is a modular form of weight $m/2$ for the congruence subgroup $\Gamma_0(N)$ and a certain quadratic character, where $N$ is the level of $L$. In \cite{E1} Eichler announced and later in \cite{E2} proved that for $N$ square-free and $m=0\bmod4$, the space of newforms of weight $m/2$ for $\Gamma_0(N)$ has a basis consisting of theta series corresponding to lattices of level $N$. More generally the question of finding an explicit basis of an appropriate space of modular forms consisting of theta series is known as the basis problem. An important contribution to this problem was given by Waldspurger. In \cite{Wal} he proved that for all positive integers $m=0\bmod4$ and $k\geq m/2$ the space of newforms of weight $k$ for $\Gamma_0(N)$ and trivial character is generated by theta series of lattices of rank $m$ and level $N$ weighted with harmonic polynomials of degree $k-m/2$. For non-trivial character he proved the result when $N$ is a square free integer congruent to $1$ modulo $4$. Furthermore, in \cite{BKS} Böcherer, Katsurada and Schulze-Pillot showed that for $k>2\g+1$ the space of Siegel cusp forms of weight $k$ and genus $\g$ for $\Gamma_0(N)$ with $N$ square free is generated by harmonic theta series of appropriate lattices.

Now let $L'$ be the dual lattice of $L$. Then $L'/L$ is called the discriminant group or discriminant form of $L$ and we denote by $\C[L'/L]$ the group algebra of $L'/L$. Let $k\geq m/2$ be an integer and let $P\in\C[x_1,\hdots,x_m]$ be a harmonic polynomial of homogeneous degree $k-m/2$. The vector valued theta series $\theta_{L,P}$ taking values in $\C[L'/L]$ is given by
\begin{align*}
    \theta_{L,P} \coloneqq \sum_{\gamma\in L'/L}\theta_\gamma e^\gamma,
\end{align*}
where $\theta_\gamma(\tau) = \sum_{\lambda\in\gamma + L}P(\lambda)e^{\pi i(\lambda,\lambda)\tau}$. When $P$ is identically $1$ it is usually dropped from the notation. By the Poisson summation formula $\theta_{L,P}$ is a vector valued modular form for the Weil representation, i.e.
\begin{align*}
    \theta_{L,P}(M\tau) = (c\tau+d)^k\rho_{L'/L}(M)\theta_{L,P}(\tau)\quad\text{for $M=\left(\begin{smallmatrix} a & b \\ c & d\end{smallmatrix}\right)\in\SL$},
\end{align*}
where $M\tau$ denotes the M\"obius transformation and $\rho_{L'/L}:\SL\rightarrow\GL(\C[L'/L])$ the Weil representation of $\SL$. The latter is a special case of the representations of symplectic groups constructed by Weil in \cite{W1}. Two positive-definite even lattices $L$ and $M$ of rank $m$ have isomorphic discriminant forms $\mathcal{D}$ and hence isomorphic Weil representations if and only if they are in the same genus, which is denoted by $I\!I_{m,0}(\mathcal{D})$. It is a natural question whether the space $\mathrm{S}_k(\D)$ of cusp forms for $\rho_{\D}$ is generated by the theta series in the genus $I\!I_{m,0}(\D)$. The present paper answers this question in the affirmative if the rank of the lattice is sufficiently large compared to the $p$-ranks of the discriminant form. Because there is no canonical isomorphism between discriminant forms of lattices in the same genus, there is also no canonical way to identify their Weil representations. Therefore, we define for a discriminant form $\mathcal{D}$
\begin{align*}
    \Theta_{m,k}(\mathcal{D}) \coloneqq &\spann\{\sigma^*\theta_{L,P}\mid L\in I\!I_{m,0}(\D), \ P \text{ harmonic of degree }k-m/2, \ \sigma\in\Iso(\mathcal{D},L'/L)\},
\end{align*}
where $\sigma^*\theta_{L,P} = \sum_{\gamma\in\mathcal{D}}\theta_{\sigma(\gamma)}e^\gamma$. The main result of this paper is Theorem \ref{thm:MainTheorem}:

%Therefore, by $\Theta_k(L'/L)$ we denote the space of modular forms generated by all possible embeddings of theta series $\theta_M$ in $\C[L'/L]$, where $M$ is in the genus of $L$. Furthermore set $\Theta_k(L'/L)_0 = \Theta_k(L'/L)\cap\mathrm{S}_k(L'/L)$, then we obtain (see Theorem \ref{thm:MainTheorem})

\medskip
{\em \mainTheorem}
\medskip

We remark that a $p$-adic lattice $L_p$ of rank $m$ splits a hyperbolic plane over $\Z_p$ if and only if $\text{$p$-rank}(\mathcal{D}) < m-2$ or $\text{$p$-rank}(\mathcal{D}) = m-2$ and $\prod_{q}\epsilon_q = \leg{-a}{p}$, where the $p$-adic component of $\D$ is equal to
\[ \bigoplus_qq^{\epsilon_qn_q} \]
in the notation of Conway and Sloane (cf.\ \cite{CS}, chapter 15) and $|\mathcal{D}| = p^\alpha a$ with $(a,p) = 1$.

It is likely possible to extend the result to the case where $m$ is odd. Then we have to consider the metaplectic group $\Mp$ instead of $\SL$. Also the condition $m>6$ can possibly be relaxed. Then, however, we will need to deal with some issues regarding convergence of some of the objects used in the proof.

In order to extend the result to discriminant forms not satisfying the condition of Theorem \ref{thm:MainTheorem}, we consider a larger space of theta series, namely
\begin{align*}
	\Theta_{m,k}^\uparrow(\mathcal{D}) \coloneqq &\spann\{\uparrow_H^\mathcal{D}(\sigma^*\theta_{L,P})\mid L\in I\!I_{m,0}(H^\bot/H) \text{ for some isotropic subgroup } H\subset\mathcal{D}, \\
	&\qquad P \text{ harmonic of degree }k-m/2, \sigma\in\Iso(H^\bot/H,L'/L)\}.
\end{align*}
Here $\uparrow_H^\mathcal{D}:\C[H^\bot/H]\rightarrow\C[\mathcal{D}]$ is a map that commutes with the corresponding Weil representations and thus sends modular forms to modular forms (see section \ref{sec:TheWeilRepresentation}). We obtain Corollary \ref{cor:MainTheorem}:

\medskip
{\em \mainTheoremCorollary}
\medskip

Corollary \ref{cor:MainTheorem} follows immediately from Theorem \ref{thm:MainTheorem} using a result in \cite{M2} (see end of Section \ref{sec:ThetaFunctions}).% The result \cite{Wal} on the scalar valued basis problem by Waldspurger also follows from Theorem \ref{thm:MainTheorem} (see Corollary \ref{cor:Waldspurger}).

The proof of Theorem \ref{thm:MainTheorem} uses the so-called doubling method that was also employed in \cite{BKS}. In what follows we give a short description of the proof idea. \\
First let us consider the case $k=m/2$. The lattices contributing to $\Theta_{m,k}(\mathcal{D})$ are in the genus $I\!I_{m,0}(\mathcal{D})$, which we will denote by $G$. We define a map $\Phi_\mathcal{D}:\mathrm{S}_k(\mathcal{D})\to\Theta_{m,k}(\mathcal{D})$ by
\begin{align*}
    \Phi_\mathcal{D}(f) \coloneqq \mu(G)^{-1}\cdot|\Ort(\mathcal{D})|^{-1}\cdot\sum_{L\in G}\frac{1}{\#\Aut(L)}\sum_{\sigma\in\mathrm{Iso}(\mathcal{D},L'/L)}(f,\sigma^*\theta_{L})\cdot\sigma^*\theta_{L}.
\end{align*}
Here the first sum ranges over the positive-definite even lattices in $G$, $\mu(G)$ denotes the mass of $G$, $(\cdot,\cdot)$ the Petersson scalar product and $\Ort(\mathcal{D})=\Iso(\mathcal{D},\mathcal{D})$ the orthogonal group of $\mathcal{D}$. The map $\Phi_\D$ sends cusp forms to cusp forms with image $\Theta_{m,k}(\mathcal{D})\cap\mathrm{S}_k(\mathcal{D})$. We want to show that it is injective. The definition of vector valued theta series can be extended to Siegel theta series of genus $2$. We find that
\begin{align*}
    \Phi_\D(f)(z') = \int_{\SL\backslash\mathbb{H}}\langle f,\theta_G^{(2)}\left(\matr{z}{0}{0}{-\overline{z'}}\right)\rangle y^k\frac{\mathrm{d}x\mathrm{d}y}{y^2},
\end{align*}
where
\begin{align*}
    \theta_G^{(2)} &= \mu(G)^{-1}\cdot|\Ort(\mathcal{D})|^{-1}\cdot\sum_{L\in G}\frac{1}{\#\Aut(L)}\sum_{\sigma\in\mathrm{Iso}(\mathcal{D},L'/L)}\sigma^*\theta_L^{(2)}
\end{align*}
is the Siegel genus theta series. By the Siegel--Weil formula $\theta_G^{(2)}$ is equal to the Siegel Eisenstein series $E_{k,\D}^{(2)}$. Studying $E_{k,\D}^{(2)}\left(\left(\begin{smallmatrix} z & 0 \\ 0 & z'\end{smallmatrix}\right)\right)$ we show that $\Phi_\D$ can be expressed in terms of Hecke operators for the Weil representation, i.e.
\begin{align*}
    \Phi_\mathcal{D} = C(k)\frac{e(-\sign(\mathcal{D})/8)}{\sqrt{|\mathcal{D}|}}\sum_{\n=1}^\infty\frac{T(\n^2)}{\n^{2k-2}},
\end{align*}
where $C(k)$ is some non-zero constant depending only on $k$. We will see that if $\Phi_{L'/L}(f) = 0$, then for any sublattice $M\subset L$ we also have $\Phi_{M'/M}(\uparrow_{L/M}^{M'/M}(f)) = 0$. On the other hand, by investigating the action of the Hecke operators for primes dividing the level of $\mathcal{D}$, we find that $\Phi_\mathcal{D}(f) = 0$ implies that $f$ has certain symmetries. %Indeed, for $\gamma,\mu\in\mathcal{D}$, whenever $p\gamma=p\mu$ and $\q(\gamma) = \q(\mu)\bmod1$ and both $\gamma$ and $\mu$ are not in the image of $x\mapsto px$, then the $\gamma$- and $\beta$-component of $f$ are equal.
Finally, we show that if the conditions of Theorem \ref{thm:MainTheorem} are satisfied, then we find a sublattice $M\subset L$ such that $\uparrow_{L/M}^{M'/M}(f)$ has the required symmetry only if $f=0$. Thus, $\Phi_\D$ is injective.

For $k=m/2+\h$ with $\h>0$ we apply a certain differential operator $\partial_\h$ to the genus theta series and then proceed analogously to the case $k=m/2$.%, which yields a $\vartheta_{G,k}\in\Theta_{m,k}(\mathcal{D})\otimes\Theta_{m,k}(\mathcal{D})$. Similar to above, we define a map $\Phi_{\mathcal{D},m,k}:\mathrm{S}_k(\mathcal{D})\rightarrow\Theta_{m,k}(\mathcal{D})$ by
%\begin{align*}
%	\Phi_{\mathcal{D},m,k}(f)(z') = \int_{\SL\backslash\mathbb{H}}\langle f,\vartheta_{G,k}(z,-\overline{z'})\rangle y^k\frac{\mathrm{d}x\mathrm{d}y}{y^2}.
%\end{align*}
%Using the same line of reasoning as for $\h = 0$ and the fact that $\partial_\h$ commutes with $|_k[M]$ for certain $M\in\Gamma^{(2)}$, we also find that
%\begin{align*}
%	\Phi_{\mathcal{D},m,k} = C(m,k)\frac{e(-\sign(\mathcal{D})/8)}{\sqrt{|\mathcal{D}|}}\sum_{\n=1}^\infty \frac{T(\n^2)}{\n^{2k-2-\h}}.
%\end{align*}
%From this point the proof is identical to the case $\h=0$.

\medskip

As first application of our main theorem we show how Waldspurger's result can be derived from it. Secondly we prove that the space of (geometrically) local obstructions for constructing Borcherds products generates the space of global obstructions. This gives a condition for whether a given divisor is the divisor of a Borcherds product that depends only on its behaviour on the boundary components.

\medskip

The paper is organized as follows. In Section \ref{sec:TheWeilRepresentation} we recall some results about discriminant forms, the Weil representation and modular forms. In Section \ref{sec:VectorValuedHeckeOperators} we define vector valued Hecke operators and study some of their properties. Then we investigate the relation between Eisenstein series and Petersson's integral kernels of Hecke operators. In Section \ref{sec:SiegelWeil} we show that the genus theta series is equal to an Eisenstein series using the adelic Siegel--Weil formula. In Section \ref{sec:ThetaFunctions} we employ the doubling method to prove the main theorem of this paper by applying the results of the previous sections. Finally, we describe the two aforementioned applications of our main result.

\medskip

I would like to thank my advisor N.\ Scheithauer for suggesting this topic as part of my doctoral thesis and for the support he has given me in pursuing it. Moreover, I would like to thank J.\ H.\ Bruinier, P.\ Kiefer, I.\ Metzler and R.\ Zufetti for stimulating discussions and helpful comments. I also thank C.\ Huang for pointing out an error in Lemma \ref{lem:ZeroForSomeP} in an earlier version of this paper, which has been corrected in the present version.

\medskip

The author acknowledges support by Deutsche Forschungsgemeinschaft (DFG, German Research Foundation) through the Collaborative Research Centre TRR 326 \textit{Geometry and Arithmetic of Uniformized Structures}, project number 444845124.

\section{The Weil representation}\label{sec:TheWeilRepresentation}

In this section we recall some results on discriminant forms, the Weil representation and modular forms.

\subsection*{Discriminant forms}

References for discriminant forms are\ 
%\cite{AGM}, 
\cite {AGM}, \cite{Bo3}, \cite{CS}, \cite{N}, \cite{S2} and \cite{Sk2}.

A discriminant form is a finite abelian group $\mathcal{D}$ with a quadratic form $\q : \mathcal{D} \to \Q/\Z$ such that $(\beta,\gamma) = \q(\beta + \gamma)- \q(\beta) - \q(\gamma) \bmod 1$ is a non-degenerate symmetric bilinear form. The level of $\mathcal{D}$ is the smallest positive integer $N$ such that
$N\q(\gamma) = 0 \bmod 1$ for all $\gamma \in \mathcal{D}$. An element $\gamma\in \mathcal{D}$ is called isotropic if $\q(\gamma) = 0\bmod1$ and anisotropic otherwise.

If $L$ is an even lattice, then $L'/L$ is a discriminant form with the quadratic form given by $\q(\gamma) = (\gamma,\gamma)/2 \bmod 1$. Conversely every discriminant form can be obtained in this way. The corresponding lattice can be chosen to be positive-definite. The signature $\sign(\mathcal{D}) \in \Z/8\Z$ of a discriminant form $\mathcal{D}$ is defined as $\sign(\mathcal{D}) = t_+-t_-\bmod8$, where $(t_+,t_-)$ is the signature of any even lattice realizing $\D$.

%Every discriminant form decomposes into a sum of Jordan components and every Jordan component can be written as a sum of indecomposable Jordan components (usually not uniquely). %For the Jordan components we use the notation introduced by Conway and Sloane (cf.\ \cite{CS}, chapter 15). Recall that for $q$ a power of an odd prime $p$ the symbol $q^{\pm n}$ denotes a discriminant form that, as a group, is isomorphic to $(\Z/q\Z)^n$ and has level $q$. For $q$ a power of $2$ the symbol $q_{I\!I}^{\pm 2n}$ denotes a discriminant form that, as a group, is isomorphic to $(\Z/q\Z)^{2n}$ and has level $q$ and $q_{t}^{\pm n}$ denotes a discriminant form that, as a group, is isomorphic to $(\Z/q\Z)^{n}$ and has level $2q$. The former are called the even $2$-adic components, the latter the odd $2$-adic components.

Let $c$ be an integer. Then $c$ acts by multiplication on $\mathcal{D}$ and we have an exact sequence
$0 \to \mathcal{D}_c \to \mathcal{D} \to \mathcal{D}^c \to 0$,
%\[ 0 \to D_c \to D \to D^c \to 0  \]
where $\mathcal{D}_c$ is the kernel and $\mathcal{D}^c$ the image of this map. Note that $\mathcal{D}^c$ is the orthogonal complement of $\mathcal{D}_c$.

The set $\mathcal{D}^{c*} = \{ \gamma \in \mathcal{D} \, | \, c\q(\alpha) + (\alpha,\gamma) = 0\bmod1 \, \text{ for all } \alpha \in \mathcal{D}_c \}$ is a coset of $\mathcal{D}^c$. After a choice of Jordan decomposition there is a canonical coset representative $x_c \in \mathcal{D}$ with $2x_c=0$. We can write $\gamma \in \mathcal{D}^{c*}$ as $\gamma = x_c + c \mu$. Then $\q_c(\gamma) = c\q(\mu) + (x_c, \mu) \bmod 1$ is well-defined. If $c$ is even, then $\mathcal{D}_{c/2}\subset\mathcal{D}_c$ and for $\alpha\in\mathcal{D}_{c/2}$ we have $c\q(\alpha) = 0\bmod1$, so that $\mathcal{D}^{c*}\subset\{\gamma\in\mathcal{D}\mid (\alpha,\gamma)=0\bmod1 \, \text{for all } \alpha\in\mathcal{D}_{c/2}\} = \mathcal{D}^{c/2}$.

Let $\mathcal{D}$ be a discriminant form of level $N$ and let $N=\prod_{p\mid N}p^{\nu_p}$ be the prime decomposition of $N$. Then $\mathcal{D}$ decomposes into the orthogonal sum of its $p$-subgroups 
\[  \mathcal{D} = \bigoplus_{p|N} \mathcal{D}_{p^{\nu_p}} \, .  \]

For a prime $p$ we define the $\text{$p$-rank}(\mathcal{D})$ of $\mathcal{D}$ as the largest non-negative integer $r$ such that $(\Z/p\Z)^r$ is isomorphic to a subgroup of $\mathcal{D}_{p^{\nu_p}}$. % and the $\text{$p$-sign}$ of $\mathcal{D}$ is defined as
%\begin{align*}
%	\text{$p$-sign}(\mathcal{D}) &\coloneqq \prod_{q}\epsilon_q \quad \text{, where} \\
%	\mathcal{D}_{p^{\nu_p}} &= \bigoplus_qq^{\epsilon_qn_q}.
%\end{align*}
%When $\mathcal{D}$ comes from a positive-definite lattice, then for each $p$, the $\text{$p$-rank}$ of $\mathcal{D}$ is bounded from above by the rank of the lattice. If they are the same, then the $\text{$p$-sign}$ of $\mathcal{D}$ is equal to $\leg{a}{p}$, where $|\mathcal{D}| = p^\alpha a$ and $(a,p) = 1$ (see \cite[chapter 15, section 7.7]{CS}).

The formula
\[ \sum_{\gamma\in\D}e(\q(\gamma)) = e(\sign(\D)/8)\sqrt{|\D|} \]
is known as Milgram's formula. See \cite{S2} for more general Gauss sums.

\subsection*{The symplectic group $\Spn{2\g}$}

Now we want to recall some facts about the symplectic group. A nice reference is \cite{F2}.

Let $\g\in\Z_{>0}$ and let $J  = J_\g = \left( \begin{smallmatrix} 0 & I_\g \\ -I_\g & 0 \end{smallmatrix} \right)$, where $I=I_\g$ is the identity matrix of rank $\g$. The symplectic group $\Spn{2\g}$ is defined as
\begin{align*}
    \Gamma^{(\g)}\coloneqq\Spn{2\g} \coloneqq \{M\in\mathrm{GL}_{2\g}(\Z)\mid M^TJM = J\}.
\end{align*}
A matrix $M=\left(\begin{smallmatrix} A & B \\ C & D\end{smallmatrix}\right)$ with $A,B,C,D\in\mathrm{Mat}_\g(\Z)$ is in $\Spn{2\g}$ if and only if
\begin{align*}
    A^TD-C^TB = D^TA-B^TC = I,\quad A^TC=C^TA,\quad B^TD=D^TB
\end{align*}
or equivalently
\begin{align*}
    AD^T-BC^T = DA^T-CB^T = I,\quad AB^T=BA^T,\quad CD^T=DC^T,
\end{align*}
in particular $\Spn{2} = \SL$. Furthermore, for $M\in\Spn{2\g}$ also $M^T\in\Spn{2\g}$ and $\det(M) = 1$. We have $J^{-1} = J^T = -J$ and in general the inverse of a symplectic matrix $M=\left(\begin{smallmatrix} A & B \\ C & D\end{smallmatrix}\right)$ (notation as above) is given by
\begin{align*}
    M^{-1} = J^{-1}M^TJ = \begin{pmatrix} D^T & -B^T \\ -C^T & A^T \end{pmatrix}.
\end{align*}
We define maps
\begin{align*}
    &n:\mathrm{Sym}_\g(\Z)\rightarrow\Gamma^{(\g)},\quad n(S) = \begin{pmatrix}I & S \\ 0 & I \end{pmatrix} \\
    &a:\GL_\g(\Z)\rightarrow\Gamma^{(\g)},\quad a(U) = \begin{pmatrix}U & 0 \\ 0 & (U^T)^{-1} \end{pmatrix} \\
    &u:\Gamma^{(\g-1)}\rightarrow\Gamma^{(\g)},\quad u\left(\begin{pmatrix} A & B \\ C & D\end{pmatrix}\right) = \begin{pmatrix}A & 0 & B & 0 \\ 0 & 1 & 0 & 0 \\ C & 0 & D & 0 \\ 0 & 0 & 0 & 1 \end{pmatrix} \\
    &d:\Gamma^{(\g-1)}\rightarrow\Gamma^{(\g)},\quad d\left(\begin{pmatrix} A & B \\ C & D\end{pmatrix}\right) = \begin{pmatrix}1 & 0 & 0 & 0 \\ 0 & A & 0 & B \\ 0 & 0 & 1 & 0 \\ 0 & C & 0 & D \end{pmatrix}.
\end{align*}
Then $n,a,u$ and $d$ are group homomorphisms and $u(M)d(M') = d(M')u(M)$. The symplectic group $\Gamma^{(\g)}$ is generated by $J_\g$ and matrices of the form $n(S)$ for $S\in\mathrm{Sym}_\g(\Z)$. The subgroup
\begin{align*}
    \Gamma_\infty^{(\g)} \coloneqq \left\{\begin{pmatrix} A & B \\ 0 & D\end{pmatrix}\in\Gamma^{(\g)}\right\}
\end{align*}
is generated by elements of the form $n(S)$ and $a(U)$. In the case $\g = 2$ we will later need the matrix
\begin{align*}
	\mathcal{A}_{\n} &\coloneqq J_2\cdot n\left(\begin{pmatrix} 0 & -1 \\ -1 & -\n \end{pmatrix}\right)J_2\cdot n\left(\begin{pmatrix} \n^2+\n & -\n-1 \\ -\n-1 & 1 \end{pmatrix}\right)J_2\cdot n\left(\begin{pmatrix} 0 & 0 \\ 0 & 1 \end{pmatrix}\right) \\
	&= \begin{pmatrix}
		\n^2+\n & -\n-1 & -1 & -\n-1 \\
		-\n-1    & 1   & 0 & 0   \\
		-\n      & 1  & 0 & 0   \\
		0      & 0   & -1 & -\n
	\end{pmatrix}.
\end{align*}

\subsection*{The Weil representation of $\Spn{2\g}$}

Let $\mathcal{D}$ be a discriminant form of even signature and level $N$ and let $\C[\mathcal{D}^\g]$ be the group algebra of $\mathcal{D}^\g = \mathcal{D}\times\hdots\times \mathcal{D}$ spanned by a formal basis $(e^{\underline{\gamma}})_{\underline{\gamma}\in \mathcal{D}^\g}$, where $\underline{\gamma} = (\gamma_1,\hdots,\gamma_\g)$. Then the \textit{Weil representation} $\rho_\mathcal{D}^{(\g)}:\Spn{2\g}\to\C[\D^\g]$ of $\Spn{2\g}$ can be defined by (cf.\ \cite{W1}, \cite{Bo1} and \cite{Zh})
\begin{align*} 
\rho_\mathcal{D}^{(\g)}(n(S)) e^{\underline{\gamma}}  & = e(1/2\tr(S(\underline{\gamma},\underline{\gamma})))e^{\underline{\gamma}} \\
\rho_\mathcal{D}^{(\g)}(J) e^{\underline{\gamma}}  & = \frac{e(\g\sign(\mathcal{D})/8)}{\sqrt{|\mathcal{D}|^\g}}
                  \sum_{\underline{\beta}\in \mathcal{D}^\g} e(\tr(\underline{\gamma},\underline{\beta}))\, e^{\underline{\beta}},
\end{align*}
where $e(z) \coloneqq e^{2\pi i z}$ and $(\underline{\gamma},\underline{\beta}) = ((\gamma_i,\beta_j))_{i,j=1}^\g\in\mathrm{Mat}_\g(\Q/\Z)$. This implies
\[ \rho_\mathcal{D}^{(\g)}(a(U)) e^{\underline{\gamma}} = \det(U)^{\sign(\mathcal{D})/2}e^{\underline{\gamma}U^{-1}} \]
and in particular $\rho_\mathcal{D}^{(\g)}(-I) e^{\underline{\gamma}} = e(\g\sign(\mathcal{D})/4)e^{-\underline{\gamma}}$. %The subgroup 
%\begin{align*}
%	\Gamma^{(\g)}(N) \coloneqq \left\{\begin{pmatrix} A & B \\ C & D\end{pmatrix}\in\Gamma^{(\g)}\bigm\vert\begin{pmatrix} A & B \\ C & D\end{pmatrix}=\begin{pmatrix} I & 0 \\ 0 & I\end{pmatrix}\bmod N\right\}
%\end{align*}
%acts trivially in the Weil representation.

We define a scalar product on the group algebra $\C[\mathcal{D}^\g]$ which is linear in the first and antilinear in the second variable by
\begin{align*}
    \langle e^{\underline{\gamma}},e^{\underline{\beta}}\rangle = \begin{cases} 1 & \text{if $\underline{\gamma} = \underline{\beta}$} \\ 0 & \text{otherwise}.\end{cases}
\end{align*}
Then the Weil representation is unitary with respect to this scalar product. There is a natural isomorphism $\C[\D^\g]\cong \C[\D]^{\otimes n}$ that we will make frequent use of. For the following cf.\ \cite[Lemma 3.4]{St}.
\begin{prp}\label{prp:ActionOfAAndD}
For a symplectic matrix $M\in\Gamma^{(\g-1)}$ the symplectic matrices $u(M)$ and $d(M)$ transform in the Weil representation as
\begin{align*}
    \rho_\mathcal{D}^{(\g)}(u(M))e^{(\gamma_1,\hdots,\gamma_\g)} &= \rho_\mathcal{D}^{(\g-1)}(M)e^{(\gamma_1,\hdots,\gamma_{\g-1})}\otimes e^{\gamma_\g} \\
    \rho_\mathcal{D}^{(\g)}(d(M))e^{(\gamma_1,\hdots,\gamma_\g)} &= e^{\gamma_1}\otimes\rho_\mathcal{D}^{(\g-1)}(M)e^{(\gamma_2,\hdots,\gamma_\g)}.
\end{align*}
\end{prp}
\begin{proof}
    It suffices to prove the assertion for the generators $n(S)$ and $J_{\g-1}$ of $\Gamma^{(\g-1)}$. For $S\in\mathrm{Sym}_{\g-1}(\Z)$ we have
    \begin{align*}
        u(n(S)) &= n(S_1) \\
        d(n(S)) &= n(S_2)
    \end{align*}
    with $S_1 = \left(\begin{smallmatrix} S & 0 \\ 0 & 0\end{smallmatrix}\right)$ and $S_2 = \left(\begin{smallmatrix} 0 & 0 \\ 0 & S\end{smallmatrix}\right)$ and the identity is trivial. Furthermore, we find
    \begin{align*}
        u(J_{\g-1}) &= -n\left(\begin{pmatrix} I_{\g-1} & 0 \\ 0 & -1 \end{pmatrix}\right)J_\g n\left(\begin{pmatrix} I_{\g-1} & 0 \\ 0 & 0 \end{pmatrix}\right)J_\g n(I_\g) \\
        d(J_{\g-1}) &= -n\left(\begin{pmatrix} -1 & 0 \\ 0 & I_{\g-1} \end{pmatrix}\right)J_\g n\left(\begin{pmatrix} 0 & 0 \\ 0 & I_{\g-1} \end{pmatrix}\right)J_\g n(I_\g)
    \end{align*}
	and verify that these transform as claimed.
\end{proof}

%For the case $\g=1$ the action of an arbitrary $M\in\SL$ was computed in \cite{S2} and is given by
%\begin{thm}[Theorem 4.7, \cite{S2}]\label{thm:ActionOfArb}
%Let $\mathcal{D}$ be a discriminant form of even signature and $M=\left(\begin{smallmatrix} a & b \\ c & d \end{smallmatrix}\right)\in\SL$. Then $M$ acts in the Weil representation of $\mathcal{D}$ as
%\begin{align*}
%    \rho_\mathcal{D}^{(1)}(M)e^\gamma = \xi\frac{\sqrt{|\mathcal{D}_c|}}{\sqrt{|\mathcal{D}|}}\sum_{\beta\in \mathcal{D}^{c*}}e(a\q_c(\beta))e(b(\beta,\gamma))e(bd\q(\gamma))e^{d\gamma+\beta},
%\end{align*}
%where $\xi$ is some eighth root of unity depending on $M$.
%\end{thm}

Let $H$ be an isotropic subgroup of $\mathcal{D}$. Then $H^\bot/H$ is a discriminant form of the same signature as $\mathcal{D}$ and order $|H^\bot/H| = |\mathcal{D}|/|H|^2.$ There is an \textit{isotropic lift}
\[ \uparrow_H\coloneqq\uparrow_H^{(\g)} : \C[(H^{\perp}/H)^\g] \to \C[\mathcal{D}^\g]  \]
defined by
\[  \uparrow_H(e^{\underline{\gamma}+H^\g}) = \sum_{\underline{\mu} \in H^\g} e^{\underline{\gamma} + \underline{\mu}}   \]
for $\underline{\gamma}\in (H^{\bot})^\g$ and an \textit{isotropic descent}\label{eq:IsotropicDescent} 
\[ \downarrow_H \coloneqq \downarrow_H^{(\g)} : \C[\mathcal{D}^\g] \to \C[(H^{\bot}/H)^\g] \]
defined by
\[ \downarrow_H(e^{\underline{\gamma}}) = 
\begin{cases}
	\, e^{\underline{\gamma}+H^\g} & \text{if $\underline{\gamma}\in (H^{\bot})^\g$}, \\
	\,          0 & \text{otherwise} 
\end{cases} 
\]
(cf.\ for example \cite{Br}, \cite{S3} or \cite{S4}). The following result is easy to prove.

\begin{prp}
	Let $\mathcal{D}$ be a discriminant form of even signature and $H$ an isotropic subgroup of $\mathcal{D}$. Then the maps $\uparrow_H^{(\g)}$ and $\downarrow_H^{(\g)}$ are adjoint with respect to the inner products on $\C[(H^{\perp}/H)^\g]$ and $\C[\mathcal{D}^\g]$ and commute with the Weil representations $\rho_{H^{\perp}/H}^{(\g)}$ and $\rho_\mathcal{D}^{(\g)}$. In particular they map modular forms to modular forms. This implies that $\uparrow_H^{(\g)}$ and $\downarrow_H^{(\g)}$ as maps of modular forms are adjoint with respect to the Petersson scalar product (defined below).
\end{prp}

For two discriminant forms $\mathcal{D}$, $\mathcal{D}'$ with quadratic forms $\q$ and $\q'$ we define
\begin{align*}
	\Iso(\mathcal{D},\mathcal{D}') &\coloneqq \{\sigma:\mathcal{D}\to \mathcal{D}'\mid\sigma\text{ is a group isomorphism with} \\
	&\qquad\qquad\q'(\sigma(\gamma)) = \q(\gamma)\bmod1 \text{ for all } \gamma\in\mathcal{D}\}
\end{align*}
and a $\sigma\in\Iso(\mathcal{D},\mathcal{D}')$ induces a \textit{pullback} $\sigma^* \coloneqq \sigma^{*(\g)}:\C[\mathcal{D}^{\prime \g}]\to\C[\mathcal{D}^\g]$ and a \textit{pushforward} $\sigma_* \coloneqq \sigma_*^{(\g)}:\C[\mathcal{D}^\g]\to\C[\mathcal{D}^{\prime \g}]$ by
\begin{align*}
	\sigma^*e^{(\gamma_1,\hdots,\gamma_\g)} &= e^{(\sigma^{-1}\gamma_1,\hdots,\sigma^{-1}\gamma_\g)} \text{ and} \\
	\sigma_*e^{(\gamma_1,\hdots,\gamma_\g)} &= e^{(\sigma\gamma_1,\hdots,\sigma\gamma_\g)}
\end{align*}
respectively. The pushforward defines a unitary representation of $\Ort(\mathcal{D}) \coloneqq \Iso(\mathcal{D},\mathcal{D})$ on $\C[\D^\g]$ that commutes with the Weil representation.

For an isotropic subgroup $H\subset \mathcal{D}'$ and $\sigma\in\Iso(\mathcal{D},\mathcal{D}')$ we easily check that $\Tilde{\sigma}(\gamma+\sigma^{-1}H) \coloneqq \sigma\gamma + H$ defines an element $\Tilde{\sigma}\in\Iso((\sigma^{-1} H)^\bot/(\sigma^{-1} H),H^\bot/H)$.\label{pg:sigmaTilde} This implies that the diagram
\[ \begin{tikzcd}%[row sep=large]
	\C[\D^\g] \arrow{r}{\sigma_*}\arrow[swap]{d}{\downarrow_{\sigma^{-1}H}} & \C[\D^{\prime\g}]\arrow{d}{\downarrow_H} \\
	\C[((\sigma^{-1} H)^\bot/(\sigma^{-1} H))^\g] \arrow{r}{\Tilde{\sigma}_*} & \C[(H^\bot/H)^\g] 
\end{tikzcd}
\]
commutes.

\subsection*{Modular forms for the Weil representation}\label{subsec:ModularFormsForTheWeilRepresentation}

We will now define modular forms for the Weil representation of $\Spn{2\g}$ generalizing the definition from \cite{F2} to the vector valued case (cf.\ also \cite{Br} and \cite{Zh}).
Let
\begin{align*}
    \mathbb{H}_\g \coloneqq \{Z\in\mathrm{Sym}_\g(\C):\Imp(Z)\in\mathrm{Pos}_\g(\R)\}
\end{align*}
be the Siegel upper half space. Let $k\in\Z$ and $f$ be a function from $\mathbb{H}$ to a complex vector space. Let $M=\left(\begin{smallmatrix} A & B \\ C & D\end{smallmatrix}\right)\in\GL_{2\g}(\R)$ be a symplectic similitude matrix, i.e.\ $M^TJ_\g M = \n J_\g$ for some $\n\in\R_{>0}$. We define the Petersson-slash operator $|_k$ by
\begin{equation*}
    (f|_k[M])(Z) = \det(M)^{k/2}\det(CZ+D)^{-k}f((AZ+B)(CZ+D)^{-1}).
\end{equation*}
\begin{defi}
	Let $\mathcal{D}$ be a discriminant form of even signature and level $N$ and let $k\in\Z$. A function $f:\mathbb{H}_\g \to \C[\mathcal{D}^\g]$ is called a modular form of weight $k$ with respect to $\rho_\mathcal{D}^{(\g)}$ and $\Spn{2\g}$ if
	\begin{enumerate}[(i)]
		\item $f|_k[M] = \rho_\mathcal{D}^{(\g)}(M)f$ for all $M\in\Spn{2\g}$,
		\item $f$ is holomorphic on $\mathbb{H}_\g$,
		\item $f$ is bounded on all domains of type $\Imp(Z)\geq Y_0,Y_0>0$.
	\end{enumerate}
	When $\g>1$ condition (iii) already follows from (i) and (ii) by the Koecher principle. Condition (iii) is equivalent to the fact that $f$ has a Fourier expansion of the form
	\begin{equation*}
		f(Z) = \sum_{\underline{\gamma}\in \mathcal{D}^\g}\sumstack{S=S^T\text{ even} \\ S\geq0} c(\underline{\gamma},S)e\left(\frac{\tr(SZ)}{2N}\right)e^{\underline{\gamma}}.
	\end{equation*}
	Moreover, if $c(\underline{\gamma},S)\not=0$ implies $S>0$, then $f$ is called a cusp form. The
	$\C$-vector space of modular forms of weight $k$ with respect to $\rho_\mathcal{D}^{(\g)}$ and $\Spn{2\g}$ is denoted by $\mathrm{M}_k^{(\g)}(\mathcal{D})$, the subspace of cusp forms by $\mathrm{S}_k^{(\g)}(\mathcal{D})$. 
\end{defi}
In the special case $\g=1$ the Fourier expansion takes the form
\begin{equation*}
    f(z) = \sum_{\gamma\in \mathcal{D}}\sumstack{n\in\Z+\q(\gamma) \\ n\geq0} c(\gamma,n)e(nz)e^\gamma.
\end{equation*}
For $f,g\in\mathrm{M}_k^{(\g)}(\D)$, where at least one of $f$ and $g$ is in $\mathrm{S}_k^{(\g)}(\mathcal{D})$ we define the \textit{Petersson inner product} $(\cdot,\cdot) = (\cdot,\cdot)^{(\g)}$ by
\begin{align*}
    (f,g) = \int_{\Gamma^{(\g)}\backslash\mathbb{H}_\g}\langle f(Z),g(Z)\rangle \det(Y)^k\frac{\mathrm{d}X\mathrm{d}Y}{\det(Y)^{\g+1}},
\end{align*}
where $Z = X+iY$. On $\mathrm{S}_k^{(\g)}(\mathcal{D})$ the Petersson inner product defines a scalar product.

\medskip

We now want to define theta series weighted with harmonic polynomials. Again we generalize the definition in \cite{F2} to the vector valued case in the same way as was done for $\g=1$ for example in \cite{Bo1}.

\begin{defi}
	A \textit{harmonic form} of degree $\h$ in the matrix variable $X = (x_{ij})_{i=1,\hdots,m;j=1,\hdots,\g}$ is a complex polynomial $P(X)$ with the properties
	\begin{enumerate}[(i)]
		\item $P(XA) = (\det A)^\h P(X)$ for $A\in\C^{\g\times\g}$,
		\item $\Delta P = \sum_{i,j}\frac{\partial^2}{(\partial x_{i,j})^2}P = 0$.
	\end{enumerate}
\end{defi}

Let $L$ be a positive-definite even lattice of even rank $m$ with dual lattice $L'$ and bilinear form $(\cdot,\cdot)$. We can choose an embedding $L\subset \R^m$ such that $(\cdot,\cdot)$ extends to the standard scalar product on $\R^m$. Let $P$ be a harmonic form of degree $\h\geq0$. We define the theta series $\theta_{L,P}^{(\g)}$ by
\begin{align*}
	\theta_{L,P}^{(\g)}(Z) \coloneqq \sum_{\underline{\lambda}\in (L')^\g}P(\underline{\lambda})e^{\pi i\tr((\underline{\lambda},\underline{\lambda})Z)}\cdot e^{\underline{\lambda}+L},
\end{align*}
where $(\underline{\lambda},\underline{\lambda}) = ((\lambda_i,\lambda_j))_{i,j=1}^\g\in\mathrm{Mat}_\g(\R)$ for $\underline{\lambda} = (\lambda_1,\hdots,\lambda_\g)$. When $P$ is identically $1$ it is usually dropped from the notation. The Poisson summation formula implies% (cf.\ \cite[Theorem 4.1]{Bo1})
\begin{thm}
	The theta series $\theta_{L,P}^{(\g)}$ is a modular form of weight $m/2+\h$ with respect to $\rho_{L'/L}^{(\g)}$ and $\Spn{2\g}$. If $\h>0$, then $\theta_{L,P}^{(\g)}$ is a cusp form.
\end{thm}

The genus $G$ of an even lattice $L$ is the set of isometry classes of lattices equivalent to $L$ over $\Z_p$ for all primes $p$ and over $\R$. It is uniquely determined by the signature $(t_+,t_-)$ of $L$ and its discriminant form $L'/L$ (cf.\ \cite{N}). Therefore, the positive-definite even lattices $L$ of rank $m$ with $L'/L\cong \mathcal{D}$ form a genus denoted by $I\!I_{m,0}(\mathcal{D})$.

We define the \textit{genus theta series} of the genus $G = I\!I_{m,0}(\mathcal{D})$ as
\begin{align*}
	\theta_G^{(\g)} \coloneqq \mu(G)^{-1}|\Ort(\mathcal{D})|^{-1}\sum_{L\in G}\frac{1}{\#\Aut(L)}\sum_{\sigma\in\mathrm{Iso}(\mathcal{D},L'/L)}\sigma^{*(\g)}\theta_{L}^{(\g)},
\end{align*}
where
\begin{align*}
	\mu(G) = \sum_{L\in G}\frac{1}{\#\Aut(L)}
\end{align*}
is the mass of the genus $G$ and the first sum ranges over the positive-definite even lattices in $G$.

We describe the theta series for $\g=1$ in more detail. Let $\operatorname{H}_m^\h$ denote the space of harmonic polynomials in $m$ variables homogeneous of degree $\h$, i.e.\ harmonic forms of degree $\h$ in $(x_1,\hdots,x_m)^T$. For a discriminant form $\mathcal{D}$ we define
\begin{align*}
    \Theta_{m,k}(\mathcal{D}) \coloneqq &\spann\{\sigma^{*(1)}\theta_{L,P}^{(1)}\mid L\in I\!I_{m,0}(\D), \ P \in\operatorname{H}_m^{k-m/2}, \ \sigma\in\Iso(\mathcal{D},L'/L)\} \\
    = &\spann\{\sigma_*^{(1)}\theta_{L,P}^{(1)}\mid L\in I\!I_{m,0}(\D), \ P \in\operatorname{H}_m^{k-m/2}, \ \sigma\in\Iso(L'/L,\mathcal{D})\}
\end{align*}
and
\begin{align*}
	\Theta_{m,k}(\mathcal{D})_0 \coloneqq \Theta_{m,k}(\mathcal{D})\cap\mathrm{S}_k^{(1)}(\mathcal{D}).
\end{align*}

For lattices $L\subset M\subset M'\subset L'$ the group $M/L$ is an isotropic subgroup of $L'/L$ and 
\[ \downarrow_{M/L}(\theta_{L,P}^{(\g)}) = \theta_{M,P}^{(\g)}, \]
where we identified $(M'/L)/(M/L)$ with $M'/M$. We find that $\downarrow_H(\Theta_{m,k}(\mathcal{D})) \subset \Theta_{m,k}(H^\bot/H)$.\label{pg:SublatticeTheta}

We can generate the space of harmonic polynomials using the so called \textit{Gegenbauer polynomials} $G_m^\h(s,n)$, which are defined by
\begin{align*}
    \frac{1}{(1-2sX+nX^2)^{m/2-1}} = \sum_{\h=0}^\infty G_m^\h(s,n)\cdot X^\h.
\end{align*}
(cf.\ \cite{EZ} and \cite{I}). Then we obtain
\begin{prp}
	The polynomial
	\begin{align*}
		P_m^{\h}(x,y) \coloneqq G_m^\h(x^Ty,\|x\|^2\|y\|^2)
	\end{align*}
	on $\R^m\times\R^m$ is harmonic of degree $\h$ in both $x$ and $y$ when the other variable is fixed and $P_m^\h(Sx,Sy) = P_m^\h(x,y)$ for all $S\in \Ort(m)$.
\end{prp}
\begin{proof}
	We denote 
	\[ f(x) = \frac{1}{(1-2x^TyX+\|x\|^2\|y\|^2X^2)^{m/2-1}} \] and compute
	\begin{align*}
		\frac{\partial^2}{(\partial x_i)^2}f(x) &= \frac{\partial}{\partial x_i}\frac{(-m/2+1)(-2y_iX + 2x_i\|y\|^2X^2)}{(1-2x^TyX+\|x\|^2\|y\|^2X^2)^{m/2}} \\
		&= \frac{(m/2-1)m/2(-2y_iX + 2x_i\|y\|^2X^2)^2}{(1-2x^TyX+\|x\|^2\|y\|^2X^2)^{m/2+1}} - \frac{(m/2-1)2\|y\|^2X^2}{(1-2x^TyX+\|x\|^2\|y\|^2X^2)^{m/2}}.
	\end{align*}
	Now
	\begin{align*}
		\sum_{i=1}^m(-2y_iX + 2x_i\|y\|^2X^2)^2 &= \sum_{i=1}^m(4y_i^2X^2 -8x_iy_i\|y\|^2X^3 + 4x_i^2\|y\|^4X^4) \\
		&= 4\|y\|^2X^2(1-2x^TyX + \|x\|^2\|y\|^2X^2).
	\end{align*}
	Hence
	\begin{align*}
		\Delta f(x) &= \frac{(m-2)m\|y\|^2X^2}{(1-2x^TyX+\|x\|^2\|y\|^2X^2)^{m/2}} - \frac{(m-2)m\|y\|^2X^2)}{(1-2x^TyX+\|x\|^2\|y\|^2X^2)^{m/2}} = 0.
	\end{align*}
	Furthermore, for a $\lambda\in\R$ we have
	\[ P_m^\h(\lambda x,y) = G_m^\h((\lambda x)^Ty,\|\lambda x\|^2\|y\|) = G_m^\h(\lambda x^Ty,\lambda^2\|x\|^2\|y\|) \]
	and
	\begin{align*}
		\sum_{\h=0}^\infty G_m^\h(\lambda s,\lambda^2n)\cdot X^\h &= \frac{1}{(1-2\lambda sX+\lambda^2nX^2)^{m/2-1}} \\
		&= \frac{1}{(1-2s(\lambda X)+n(\lambda X)^2)^{m/2-1}} \\
		&= \sum_{\h=0}^\infty G_m^\h(s,n)\cdot (\lambda X)^\h \\
		&= \sum_{\h=0}^\infty \lambda^\h G_m^\h(s,n)\cdot X^\h
	\end{align*}
	so that $P_m^\h(\lambda x,y) = \lambda^\h P_m^\h(x,y)$. The fact that $P_m^\h(Sx,Sy) = P_m^\h(x,y)$ for all $S\in \Ort(m)$ follows from definition.
\end{proof}
Now let $Z = \left(\begin{smallmatrix} z_1 & z_2 \\ z_2 & z_4 \end{smallmatrix}\right)\in\mathbb{H}_2$  and $\underline{\lambda}=(\lambda,\mu)\in (L')^2$. We compute
\[ e^{\pi i \tr((\underline{\lambda},\underline{\lambda})Z)} = e^{\pi i ((\lambda,\lambda)z_1 + 2(\lambda,\mu)z_2 + (\mu,\mu)z_4)} \]
and therefore
\begin{align*}
	\frac{1}{2}\frac{\partial}{\partial z_2}e^{\pi i \tr((\underline{\lambda},\underline{\lambda})Z)}\big|_{z_2=0} &= \pi i (\lambda,\mu)e^{\pi i ((\lambda,\lambda)z_1 + (\mu,\mu)z_4)} \\
	\frac{\partial^2}{\partial z_1\partial z_4}e^{\pi i \tr((\underline{\lambda},\underline{\lambda})Z)}\big|_{z_2=0} &= (\pi i)^2 (\lambda,\lambda)(\mu,\mu)e^{\pi i ((\lambda,\lambda)z_1 + (\mu,\mu)z_4)}.
\end{align*}
It follows that
\begin{multline*}
	G_m^\h\left(\frac{1}{2}\frac{\partial}{\partial z_2},\frac{\partial^2}{\partial z_1\partial z_4}\right)\theta_L^{(2)}(Z)\big|_{z_2=0} \\
	= (\pi i)^{\h}\sum_{\gamma,\beta\in L'/L}\left[\sumstack{\lambda\in \gamma+L \\ \mu\in \beta+L}P_m^{\h}(\lambda,\mu)e^{\pi i(\lambda,\lambda)z_1}e^{\pi i(\mu,\mu)z_2}\right]e^{\gamma}\otimes e^\beta\in\Theta_{m,m/2+\h}(\mathcal{D})\otimes\Theta_{m,m/2+\h}(\mathcal{D}).
\end{multline*}
In light of this we define on the space of $C^\infty$ functions on $\mathbb{H}_2$ the operator
\begin{align*}
	&\partial_\h:C^\infty(\mathbb{H}_2)\rightarrow C^\infty(\mathbb{H}\times\mathbb{H}) \\
	&\partial_\h f = G_m^\h\left(\frac{1}{2}\frac{\partial}{\partial z_2},\frac{\partial^2}{\partial z_1\partial z_4}\right)f(Z)|_{z_2=0}.
\end{align*}
Also let
\begin{align*}
    \vartheta_{G,m/2+\h} \coloneqq \partial_\h\theta_G^{(2)}\in\Theta_{m,m/2+\h}(\mathcal{D})\otimes\Theta_{m,m/2+\h}(\mathcal{D}).
\end{align*}
Note that $\partial_{\h}$ is essentially the operator $\mathcal{D}_\h$ defined by Eichler and Zagier in \cite{EZ}. It is a special case of the operators studied in \cite{I}.

We define a scalar product $\operatorname{h}(\cdot,\cdot)$ on $\operatorname{H}_m^\h$ by
\[ \operatorname{h}(p,q) = \int_{B_1}\nabla p(x)^T \, \overline{\nabla q(x)}\mathrm{d}x. \]
An element $S\in\operatorname{SO}(m)$ naturally acts on $\operatorname{H}_m^\h$ by $S.p = p(S^T\cdot)$ and it is well-known that this representation is irreducible (see for example \cite[(0.9) and (5.7)]{KV}). Note that we have $\nabla S.p = S\cdot(\nabla p)(S^T\cdot)$ so that for any $p,q\in\operatorname{H}_m^\h$ we obtain
\begin{align*}
	\operatorname{h}(S.p,S.q) &= \int_{B_1}\nabla p(S^Tx)^T \, \overline{\nabla q(S^Tx)}\mathrm{d}x = \int_{B_1}[(\nabla p)(S^Tx)]^TS^TS \, \overline{(\nabla q)(S^Tx)}\mathrm{d}x \\
	&= \int_{S^TB_1}\nabla p(x)^T \, \overline{\nabla q(x)}\mathrm{d}x = \operatorname{h}(p,q),
\end{align*}
where we substituted $Sx$ for $x$ in the last step and used the fact that $S^TS = I$. This implies

\begin{prp}\label{prp:partialthetaislin}
	Let $m$ be even and $\h>0$ and let $(P_1,\hdots,P_r)$ be any orthonormal basis of $\operatorname{H}_m^\h$ with respect to $\operatorname{h}(\cdot,\cdot)$. Then $P_m^\h(x,y)$ is up to a non-zero constant equal to
	\[ \sum_{i=1}^rP_i(x)\overline{P_i}(y). \]
\end{prp}
\begin{proof}
	We define a map $\phi:\operatorname{H}_m^\h\to\operatorname{H}_m^\h$ by
	\[ \phi(p)(x) = \operatorname{h}(p,P_m^\h(\cdot,\overline{x})). \]
	We find that
	\begin{align*}
		(S.\phi(p))(x) &= \operatorname{h}(p,P_m^\h(\cdot,\overline{S^Tx})) = \operatorname{h}(p,P_m^\h(S\cdot,\overline{x})) = \operatorname{h}(S.p,P_m^\h(\cdot,\overline{x})) = \phi(S.p)(x)
	\end{align*}
	and so by Schur's lemma, $\phi$ must be a multiple of the identity. Since $P_m^\h$ is non-zero, this scalar is non-zero. Clearly the identity map is given by
	\[ p \mapsto \sum_{i=1}^r\operatorname{h}(p,P_i)P_i. \]
\end{proof}

\section{Vector valued Hecke operators}\label{sec:VectorValuedHeckeOperators}

We now want to define vector valued Hecke operators. They were introduced by Bruinier and Stein in \cite{BS} and naturally appear when we use the doubling method. We will study some of their properties and describe their kernel functions. In this section we will only consider the case $\g=1$. So from now on, unless stated otherwise, whenever $\g$ is omitted, it is assumed to be equal to $1$.

Let $\mathcal{D}$ be a discriminant form of even signature and level $N$. We define Hecke operators $T(\n^2)$ acting on $\mathrm{M}_k(\mathcal{D})$ (see \cite{BS}). In order to do so we extend the right action given by the inverse of the Weil representation to certain matrices in $\mathrm{Mat}_2(\Z)$. Let $\n$ be a non-negative integer and $\alpha = \left(\begin{smallmatrix} \n^2 & 0 \\ 0 & 1\end{smallmatrix}\right)\in\mathrm{Mat}_2(\Z)$. We define
\begin{align*}
    \rho_\mathcal{D}(\alpha)^{-1}e^\gamma = e^{\n\gamma}.
\end{align*}
For any $\delta=A\alpha B\in\Gamma\alpha\Gamma$ we put
\begin{align*}
    \rho_\mathcal{D}(\delta)^{-1}e^\gamma = \rho_\mathcal{D}(B^{-1})\rho_\mathcal{D}(\alpha)^{-1}\rho_\mathcal{D}(A^{-1})e^\gamma.
\end{align*}
It is shown in \cite{BS} that $\rho_\mathcal{D}(\delta)^{-1}$ is well-defined and that
\begin{align*}
    \rho_\mathcal{D}(A\delta B)^{-1}e^\gamma = \rho_\mathcal{D}(B)^{-1}\rho_\mathcal{D}(\delta)^{-1}\rho_\mathcal{D}(A)^{-1}e^\gamma = \rho_\mathcal{D}(B^{-1})\rho_\mathcal{D}(\delta)^{-1}\rho_\mathcal{D}(A^{-1})e^\gamma
\end{align*}
for all $\delta\in\Gamma\alpha\Gamma$ and $A,B\in\Gamma$. Furthermore, $\beta = \left(\begin{smallmatrix} 1 & 0 \\ 0 & \n^2\end{smallmatrix}\right)\in\Gamma\alpha\Gamma$, in fact $\beta = -J\alpha J$ and
\begin{align*}
    \rho_\mathcal{D}(\beta)^{-1}e^\gamma = \sumstack{\mu\in \mathcal{D} \\ \n\mu = \gamma}e^\mu.
\end{align*}
The element $\beta$ satisfies
\begin{align*}
    \langle \rho_\mathcal{D}(\alpha)^{-1}v,w\rangle &= \langle v,\rho_\mathcal{D}(\beta)^{-1}w\rangle \quad\text{and} \\
    \langle \rho_\mathcal{D}(\beta)^{-1}v,w\rangle &= \langle v,\rho_\mathcal{D}(\alpha)^{-1}w\rangle.
\end{align*}
When $(\n,N) = 1$ we find that for $\delta\in\Gamma\alpha\Gamma$
\begin{align}\label{eq:actionOfAlphaWhenCoprime}
    \rho_\mathcal{D}(\delta)^{-1} = \chi_\mathcal{D}(\n)\rho_\mathcal{D}(\Tilde{\delta})^{-1} = \chi_\mathcal{D}(\n)\rho_\mathcal{D}(\Tilde{\delta}^{-1}),
\end{align}
where $\Tilde{\delta}\in\Gamma$ is any representative of $\n^{-1}\delta\in\SLn{N}$ and $\chi_\mathcal{D}$ is the quadratic character
\[ \chi_\mathcal{D}(\n) = \leg{\n}{|\mathcal{D}|}e((\n-1)\operatorname{oddity}(\mathcal{D})/4). \]
Here $\operatorname{oddity}(\mathcal{D})\bmod8$ depends only on the $2$-adic subgroup of $\mathcal{D}$ and is even if the signature of $\mathcal{D}$ is even. For a precise definition see \cite[Chapter 15]{CS}. The following lemma is well-known (see e.g.\ \cite[Hilfssatz IV.1.12]{F2})
\begin{lem}\label{lem:DoubleCosetSet}
    For $\n,m\in\Z$ with $(\n,m) \not=(0,0)$ we have the equality
    \begin{align*}
        \Gamma\begin{pmatrix} \n & 0 \\ 0 & m \end{pmatrix}\Gamma &= \left\{\begin{pmatrix} a & b \\ c & d \end{pmatrix}\in\mathrm{Mat}_2(\Z) \mid ad-bc=\n m, \gcd(a,b,c,d) = \gcd(\n,m)\right\}.
    \end{align*}
\end{lem}
We will sometimes, when convenient, simply write $(\cdot,\cdot)$ for $\gcd(\cdot,\cdot)$. Denote by
\begin{align*}
    M_\n :\!&= \Gamma\alpha\Gamma =\left\{\begin{pmatrix} a & b \\ c & d \end{pmatrix}\in\mathrm{Mat}_2(\Z) \mid ad-bc=\n^2, \gcd(a,b,c,d) = 1\right\}.
\end{align*}

Now let
\begin{align*}
    M_\n = \Gamma\alpha\Gamma = \bigcup_i\Gamma\cdot\delta_i
\end{align*}
be a disjoint right coset decomposition. We define the Hecke operator $T(\n^2)$ on modular forms $f\in\mathrm{M}_k(\mathcal{D})$ by
\begin{align*}
    T(\n^2)f \coloneqq \n^{k-2}\sum_i\rho_\mathcal{D}(\delta_i)^{-1}f|_k[\delta_i].
\end{align*}
Then
\begin{thm}[Theorem 5.6, \cite{BS}]\label{thm:BruinierStein}
For any positive integer $\n$, the Hecke operator $T(\n^2)$ is a linear operator on $\mathrm{M}_k(\mathcal{D})$ taking cusp forms to cusp forms. It is self-adjoint with respect to the Petersson scalar product. Moreover, if $\n$, $m$ are coprime, then
\begin{align*}
    T(\n^2)T(m^2) = T(\n^2m^2).
\end{align*}
\end{thm}
The operators $T(l^2)$ with $l$ coprime to $N$ behave analogously to the classical Hecke operators for $\SL$.
\begin{prp}\label{prp:EulerProduct}
	The algebra generated by all Hecke operators $T(\n^2)$ with $(\n,N) = 1$ is a commutative algebra of self-adjoint operators. Hence, there exists a basis of $\mathrm{S}_k(\mathcal{D})$ consisting of simultaneous eigenforms for it. Let $f$ be a simultaneous eigenform with eigenvalues $\lambda(\n^2)$. The $L$-series
	\begin{align*}
		L(f,s) \coloneqq \sumstack{\n=1 \\ (\n,N) = 1}^\infty\frac{\lambda(\n^2)}{\n^s}
	\end{align*}
	converges for $\Rep(s) > k$ and has an Euler-product
	\begin{align*}
		L(f,s) = \prod_{p\nmid N}\frac{(1-\chi_\mathcal{D}(p)p^{k-2-s})(1 +\chi_\mathcal{D}(p)p^{k-1-s})}{1-(\lambda(p^2)+\chi_\mathcal{D}(p)(1-p)p^{k-2})p^{-s}+p^{2k-2-2s}}.
	\end{align*}
\end{prp}
\begin{proof}
	Let $(l,N) = 1$ and $\widetilde{M_\n} \coloneqq \{M = \left(\begin{smallmatrix}a&b\\c&d\end{smallmatrix}\right)\in\mathrm{Mat}_2(\Z)\mid\det(M) = \n^2\}$. We define operators $\widetilde{T}(\n^2)$ by
	\[ \widetilde{T}(\n^2)f \coloneqq \n^{k-2}\sum_{\delta\in\Gamma\backslash\widetilde{M_\n}}\rho_\mathcal{D}(\delta)^{-1}f|_k[\delta], \]
	where $\rho_\mathcal{D}(\delta)^{-1}$ acts as $\chi_\mathcal{D}(\n)\rho_\mathcal{D}(\Tilde{\delta}^{-1})$ for any representative $\Tilde{\delta}\in\Gamma$ of $\n^{-1}\delta\in\SLn{N}$. By equation (\ref{eq:actionOfAlphaWhenCoprime}) this extends the previously defined action of $M_\n$. Clearly
	\[ \widetilde{M_\n} = \bigcup_{d\mid\n}\frac{\n}{d}M_d \]
	so that
	\[ \widetilde{T}(\n^2) = \sum_{d\mid \n}\chi_\mathcal{D}\left(\frac{\n}{d}\right)\left(\frac{\n}{d}\right)^{k-2}T(d^2), \]
	which implies $T(p^{2r}) = \widetilde{T}(p^{2r}) - \chi_\mathcal{D}(p)p^{k-2}\widetilde{T}(p^{2(r-1)})$ for a prime $p$. Therefore, the operators $T(\n^2)$ and $\widetilde{T}(\n^2)$ generate the same algebra of operators. Furthermore, as in the classical scalar case, we can show that for $r\geq2$ we have
	\begin{equation}\label{eq:TprFromTp}
		\widetilde{T}(p^{2r}) = \widetilde{T}(p^{2(r-1)})(\widetilde{T}(p^2)-\chi_\mathcal{D}(p)p^{k-1}) - p^{2k-2}\widetilde{T}(p^{2(r-2)})
	\end{equation}
	and so this algebra is a commutative algebra of self-adjoint operators. (More details of this Hecke algebra can be found in the upcoming \cite{Me}.) This implies that $\mathrm{S}_k(\mathcal{D})$ has a basis consisting of simultaneous eigenforms. Let $f$ be an eigenform with eigenvalues $\lambda(\n^2)$ for $T(\n^2)$ and $\widetilde{\lambda}(\n^2)$ for $\widetilde{T}(\n^2)$. We define
	\[ \widetilde{L}(f,s) \coloneqq \sumstack{\n=1 \\ (\n,N) = 1}^\infty\frac{\widetilde{\lambda}(\n^2)}{\n^s}. \]
	Then we have
	\begin{align*}
		L(f,s) &= \prod_{p\nmid N}\sum_{r = 0}^\infty\frac{\lambda(p^{2r})}{p^{rs}} \\
		&= \prod_{p\nmid N}\left(\sum_{r = 0}^\infty\frac{\widetilde{\lambda}(p^{2r})}{p^{rs}} - \chi_\mathcal{D}(p)p^{k-2-s}\sum_{r = 0}^\infty\frac{\widetilde{\lambda}(p^{2r})}{p^{rs}}\right)  \\
		&= \prod_{p\nmid N}(1-\chi_\mathcal{D}(p)p^{k-2-s})\left(\sum_{r = 0}^\infty\frac{\widetilde{\lambda}(p^{2r})}{p^{rs}}\right) \\
		&= \widetilde{L}(f,s)\prod_{p\nmid N}(1-\chi_\mathcal{D}(p)p^{k-2-s}).
	\end{align*}
	The convergence of the latter product is clear for $\Rep(s)>k$, as it the reciprocal of a Dirichlet series. We want to deduce the convergence of $\widetilde{L}(f,s)$ from the scalar situation. Let $\gamma\in\mathcal{D}$ be of order $n$ and $\chi:(\Z/n\Z)^\times\rightarrow\C^\times$ a character. Define
	\begin{align*}
		v_{\gamma,\chi} \coloneqq \sum_{x\in(\Z/n\Z)^\times}\chi(x)^{-1}e^{x\gamma}.
	\end{align*}
	Let $R_a\in\SL$ be any preimage of $\left(\begin{smallmatrix} a^{-1}&0\\0&a \end{smallmatrix}\right)\in\SLn{N}$. Then
	\[ \left\{R_a\left(\begin{smallmatrix}
		a&bN\\
		0&d
	\end{smallmatrix}\right)\mid ad = \n^2,0\leq b<d \right\} \]
	is a system of representatives for $\Gamma\backslash\widetilde{M_\n}$, but also for
	\[ \Gamma(N)\backslash\{M\in\mathrm{Mat}_2(\Z)\mid\det(M) = \n^2, M = \left(\begin{smallmatrix}1&0\\0&\n^2\end{smallmatrix}\right)\bmod N\}. \]
	Using this system of representatives we compute 
	\begin{align*}
		\langle \widetilde{T}(\n^2)f,v_{\gamma,\chi}\rangle = \chi(\n)T^{\Gamma(N)}(\n^2)\langle f,v_{\gamma,\chi}\rangle,
	\end{align*}
	where $T^{\Gamma(N)}(\n^2)$ is the standard Hecke operator on $\mathrm{S}_k(\Gamma(N))$ (cf.\ \cite[Theorem 31 (ii)]{Wr}). Since the elements of the form $v_{\gamma,\chi}$ generate $\C[\D]$, we find a pair $(\gamma,\chi)$ such that
	\begin{align*}
		g \coloneqq \langle f,v_{\gamma,\chi}\rangle \not= 0
	\end{align*}
	and $g\in\mathrm{S}_k(\Gamma(N))$ is a simultaneous eigenform with eigenvalues $\chi(\n)^{-1}\widetilde{\lambda}(\n^2)$. Now the convergence of $\widetilde{L}(f,s)$ follows from the scalar case (see for example \cite{R}). Again we write
	\begin{align*}
		\sumstack{\n=1 \\ (\n,N) = 1}^\infty\frac{\widetilde{\lambda}(\n^2)}{\n^s} &= \prod_{p\nmid N}\sum_{r=0}^\infty\frac{\widetilde{\lambda}(p^{2r})}{p^{rs}}.
	\end{align*}
	From (\ref{eq:TprFromTp}) we obtain
	\begin{align*}
		\sum_{r=0}^\infty\frac{\widetilde{\lambda}(p^{2r})}{p^{rs}} &= 1 + \frac{\widetilde{\lambda}(p^2)}{p^s} + (\widetilde{\lambda}(p^2)-\chi_\mathcal{D}(p)p^{k-1})\cdot \sum_{r=2}^\infty\frac{\widetilde{\lambda}(p^{2(r-1)})}{p^{rs}}-p^{2k-2}\cdot\sum_{r=2}^\infty\frac{\widetilde{\lambda}(p^{2(r-2)})}{p^{rs}} \\
		&= 1 +\chi_\mathcal{D}(p)p^{k-1-s} + \frac{\widetilde{\lambda}(p^2)-\chi_\mathcal{D}(p)p^{k-1}}{p^s}\cdot \sum_{r=0}^\infty\frac{\widetilde{\lambda}(p^{2r})}{p^{rs}}-p^{2k-2-2s}\cdot\sum_{r=0}^\infty\frac{\widetilde{\lambda}(p^{2r})}{p^{rs}},
	\end{align*}
	which implies
	\begin{align*}
		\sum_{r=0}^\infty\frac{\widetilde{\lambda}(p^{2r})}{p^{rs}} &=\frac{1 +\chi_\mathcal{D}(p)p^{k-1-s}}{1-(\widetilde{\lambda}(p^2)-\chi_\mathcal{D}(p)p^{k-1})p^{-s} + p^{2k-2-2s}}.
	\end{align*}
	The theorem now follows from $\widetilde{{\lambda}}(p^2) = \lambda(p^2) + \chi_\D(p)p^{k-2}$.
\end{proof}

We also want to study the behaviour of the Hecke operator $T(p^{2r})$ for $p\mid N$ and $r\in\Z_{\geq0}$. Let $x\in\Q/\Z$ with $Nx=0\bmod1$. Then
\begin{align*}
	\chi_x\left(\matr{a}{b}{c}{d}\right) = e(bx)
\end{align*}
is a character on $\Gamma_1(N)$. For $p\mid N$ and $r\in\Z_{\geq0}$ we define an operator $T^x(p^{2r}):\mathrm{M}_k(\Gamma_1(N),\chi_{p^{2r}x})\rightarrow\mathrm{M}_k(\Gamma_1(N),\chi_{x})$ by
\begin{align*}
	T^x(p^{2r})f \coloneqq p^{rk-2r}\sum_{b=0}^{p^{2r}-1}e(-bx)f|_k\left[\matr{1}{b}{0}{p^{2r}}\right].
\end{align*}
It is not difficult to verify that this is well-defined.
\begin{prp}\label{prp:HeckeP}
Let $\mathcal{D}$ be a discriminant form of even signature, $p$ a prime and $r\geq1$. Let $f\in\mathrm{M}_k(\mathcal{D})$ and $\gamma\in \mathcal{D}$ with $\gamma\not\in\mathcal{D}^{p}$ and if $p=2$, also $\gamma\not\in\mathcal{D}^{2*}$. Then
\begin{align*}
    \langle T(p^{2r})f,e^\gamma\rangle = T^{\q(\gamma)}(p^{2r})\langle f,e^{p^r\gamma}\rangle.
\end{align*}
\end{prp}
\begin{proof}
A system of representatives $\delta_i$ of the right coset decomposition of $M_{p^{2r}}$ is given by
\begin{align*}
    \left\{\delta_{s,b}=\begin{pmatrix} p^s & b \\ 0 & p^{2r-s}\end{pmatrix}\bigm\vert 0\leq s\leq 2r,0\leq b<p^{2r-s}\text{ and } (b,p) = 1\text{ if } 0< s < 2r\right\}
\end{align*}
so that
\begin{align*}
	\langle T(p^{2r})f,e^\gamma\rangle = p^{rk-2r}\langle\rho_\mathcal{D}(\alpha)^{-1}f|_k[\alpha],e^\gamma\rangle &+ p^{rk-2r}\sum_{s=1}^{2r-1}\sumstack{b=0 \\ (b,p) = 1}^{p^{2r-s}-1}\langle\rho_\mathcal{D}(\delta_{s,b})^{-1}f|_k[\delta_{s,b}],e^\gamma\rangle \\
	&+ p^{rk-2r}\sum_{b=0}^{p^{2r}-1}\langle\rho_\mathcal{D}(\delta_{0,b})^{-1}f|_k[\delta_{0,b}],e^\gamma\rangle.
\end{align*}
Recall that
\begin{align*}
	\rho_\mathcal{D}(\beta)^{-1}e^\gamma = \sumstack{\mu\in \mathcal{D} \\ p^r\mu = \gamma}e^\mu
\end{align*}
and so
\begin{align*}
	\langle\rho_\mathcal{D}(\alpha)^{-1}f|_k[\alpha],e^\gamma\rangle &= \langle f|_k[\alpha],\rho_\mathcal{D}(\beta)^{-1}e^\gamma\rangle = 0
\end{align*}
because $\gamma\not\in \mathcal{D}^p$. For a given $s<2r$ and $b$, there exist $x,y\in\Z$ such that $xb-yp^s = 1$. Hence, we can write
\begin{align*}
    \begin{pmatrix} p^s & b \\ 0 & p^{2r-s}\end{pmatrix} = \begin{pmatrix} 0 & 1 \\ -1 & p^{2r-s}x\end{pmatrix}\begin{pmatrix} p^{2r} & 0 \\ 0 & 1\end{pmatrix}\begin{pmatrix} x & y \\ p^s & b\end{pmatrix}.
\end{align*}
Let $0<s<2r$. Similar to $s=2r$, using the explicit formula for the action of an arbitrary element in the Weil representation from \cite[Theorem 4.7]{S2} we obtain
\begin{align*}
    \langle\rho_\mathcal{D}(\delta_{s,b})^{-1}f|_k[\delta_{s,b}],e^\gamma\rangle &= \langle f|_k[\delta_{s,b}],\rho_\mathcal{D}(\left(\begin{smallmatrix} 0 & 1 \\ -1 & p^{2r-s}x\end{smallmatrix}\right))\rho_\mathcal{D}(\beta)^{-1}\rho_\mathcal{D}(\left(\begin{smallmatrix} x & y \\ p^s & b\end{smallmatrix}\right))e^\gamma\rangle \\
    &= \langle f|_k[\delta_{s,b}],\rho_\mathcal{D}(\left(\begin{smallmatrix} 0 & 1 \\ -1 & p^{2r-s}x\end{smallmatrix}\right))\rho_\mathcal{D}(\beta)^{-1} \\
    &\qquad\quad\xi\frac{\sqrt{|\mathcal{D}_{p^s}|}}{\sqrt{|\mathcal{D}|}}\sum_{\mu\in \mathcal{D}^{p^s*}}e(x\q_{p^s}(\mu))e(y(\mu,\gamma))e(by\q(\gamma))e^{b\gamma+\mu}\rangle \\
    &= 0
\end{align*}
because of the following reasoning: Suppose that $p^{r}\gamma' = b\gamma+\mu$ for some $\gamma'\in \mathcal{D}$. If $p$ is odd, then $\mathcal{D}^{p^s*} = \mathcal{D}^{p^s}$ and $\mu=p^s\mu'$. But then also
\[ b\gamma = p(p^{r-1}\gamma'-p^{s-1}\mu'). \]
Since $(b,p) = 1$, this contradicts $\gamma\not\in \mathcal{D}^p$. If $p=2$, first consider $s=1$. Then $b\gamma = 2^{r}\gamma' - \mu\in\mathcal{D}^{2*}$ because $\mathcal{D}^{2*}$ is a coset of $\mathcal{D}^{2}$. It is not difficult to see that then also $\gamma\in\mathcal{D}^{2*}$ because $b$ is odd. If $s>1$, recall that $\mathcal{D}^{2^s*}\subset\mathcal{D}^{2^{s-1}}$. Hence, $\mu = 2^{s-1}\mu'$ and $b\gamma = 2(2^{r-1}\gamma'-2^{s-2}\mu')\in \mathcal{D}^2$ and since $b$ is odd, also $\gamma\in\mathcal{D}^2$. \\
Finally, $\delta_{0,b} = \beta\left(\begin{smallmatrix} 1 & b \\ 0 & 1\end{smallmatrix}\right)$ implies
\begin{align*}
    \langle\rho_\mathcal{D}(\delta_{0,b})^{-1}f|_{\delta_{0,b}},e^\gamma\rangle &= \langle f|_{\delta_{0,b}},\rho_\mathcal{D}(\alpha)^{-1}\rho_\mathcal{D}(\left(\begin{smallmatrix} 1 & b \\ 0 & 1\end{smallmatrix}\right))e^\gamma\rangle \\
    &= e(-b\q(\gamma))\langle f|_{\delta_{0,b}},e^{p^r\gamma}\rangle.
\end{align*}
Therefore, we find 
\[ \langle T(p^{2r})f,e^\gamma\rangle = p^{rk-2r}\sum_{b=0}^{p^{2r}-1}e(-b\q(\gamma))\langle f,e^{p^{r}\gamma}\rangle|_k[\delta_{0,b}]. \]
\end{proof} 

Let $\gamma,\mu\in\D$ and $p$ be a prime. We call the projection of $\gamma$ to the $p$-adic component of $\D$ the $p$-adic component of $\gamma$. For a finite set of primes $P$ denote by $\gamma_P^\mu\in\D$ the element whose $p$-adic components are equal to those of $\mu$ for all $p\in P$ and equal to those of $\gamma$ for all other $p$. For example $\gamma_\emptyset^\mu = \gamma$ and $\gamma_P^\mu = \mu$ if $P$ contains all primes $p\mid N$. For $a\in\Z$ we have $a\gamma_P^\mu = (a\gamma)_P^{a\mu}$. We define
\[ v_{\gamma,\mu,P}\coloneqq\sum_{S\subset P}(-1)^{|S|}e^{\gamma_{S}^\mu}. \]
Then we have

\begin{cor}\label{cor:Symmetry}
	Let $\mathcal{D}$ be a discriminant form of even signature, $P$ a finite set of prime numbers and $\gamma,\mu\in\D$. Assume that $\left(\prod_{p\in P}p\right)(\gamma -\mu) = 0$, $\q(\gamma) = \q(\mu)\bmod1$ and for all $p\in P$ assume $\gamma,\mu\not\in\mathcal{D}^{p}$ and if $2\in P$, also $\gamma,\mu\not\in\mathcal{D}^{2*}$. Let $s\in\C$ and $f\in\mathrm{S}_k(\mathcal{D})$ and assume that
	\[ \prod_{p\in P}\left(\sum_{r=0}^\infty\frac{T(p^{2r})}{p^{rs}}\right)f \]
	converges. Then 
	\[ \langle \prod_{p\in P}\left(\sum_{r=0}^\infty\frac{T(p^{2r})}{p^{rs}}\right)f,v_{\gamma,\mu,P}\rangle = \langle f,v_{\gamma,\mu,P}\rangle. \]
\end{cor}
\begin{proof}
	We prove the statement by induction on $|P|$: Note that for $P=\emptyset$ there is nothing to prove. \\
	So let $|P|\geq1$ and assume that the assertion holds for all sets smaller than $P$. Let $p\in P$. We have
	\[ v_{\gamma,\mu,P} = v_{\gamma,\mu_{\{p\}}^\gamma,P\setminus\{p\}} - v_{\gamma_{\{p\}}^\mu,\mu,P\setminus\{p\}}. \]
	The pairs $(\gamma,\mu_{\{p\}}^\gamma)$ and $(\gamma_{\{p\}}^\mu,\mu)$ satisfy the conditions of the corollary for the set of primes $P\setminus\{p\}$ and so by the induction hypothesis we have
	\[ \langle \prod_{q\in P\setminus\{p\}}\left(\sum_{r=0}^\infty\frac{T(q^{2r})}{q^{rs}}\right)f,v_{\gamma,\mu,P}\rangle = \langle f,v_{\gamma,\mu,P}\rangle. \]
	In fact, all elements of the form $\gamma_{S}^\mu$ for $S\subset P$ satisfy the conditions of Proposition \ref{prp:HeckeP} with $\q(\gamma_{S}^\mu) = \q(\gamma) = \q(\mu)\bmod1$. Hence, we have for $r\geq1$
	\begin{align*}
		\langle T(p^{2r})f,v_{\gamma,\mu_{\{p\}}^\gamma,P\setminus\{p\}}\rangle &= T^{\q(\gamma)}(p^{2r})\langle f,v_{p^r\gamma,p^r\mu_{\{p\}}^{\gamma},P\setminus\{p\}}\rangle \text{ and} \\
		\langle T(p^{2r})f,v_{\gamma_{\{p\}}^\mu,\mu,P\setminus\{p\}}\rangle &= T^{\q(\gamma)}(p^{2r})\langle f,v_{p^r\gamma_{\{p\}}^{\mu},p^r\mu,P\setminus\{p\}}\rangle.
	\end{align*}
	Furthermore, $p\gamma = p\gamma_{\{p\}}^{\mu}$ and $p\mu = p\mu_{\{p\}}^{\gamma}$, so that the right-hand sides are equal, thus
	\begin{align*}
		\langle \left(\sum_{r=0}^\infty\frac{T(p^{2r})}{p^{rs}}\right)f,v_{\gamma,\mu,P}\rangle &= \langle f,v_{\gamma,\mu,P}\rangle + \langle \left(\sum_{r=1}^\infty\frac{T(p^{2r})}{p^{rs}}\right)f,v_{\gamma,\mu_{\{p\}}^\gamma,P\setminus\{p\}} - v_{\gamma_{\{p\}}^\mu,\mu,P\setminus\{p\}}\rangle \\
		&= \langle f,v_{\gamma,\mu,P}\rangle + \sum_{r=1}^\infty\frac{T^{\q(\gamma)}(p^{2r})}{p^{rs}}\langle f,v_{p^r\gamma,p^r\mu_{\{p\}}^{\gamma},P\setminus\{p\}} - v_{p^r\gamma_{\{p\}}^{\mu},p^r\mu,P\setminus\{p\}}\rangle \\
		&= \langle f,v_{\gamma,\mu,P}\rangle.
	\end{align*}
\end{proof}

Finally, we want to find a kernel function for the Hecke operators. We will need
\begin{lem}\label{lem:cotSum}
	Let $k\geq2$ and $z\in\mathbb{H}$. Then for $x\in\Q$
	\begin{equation*}
		\sum_{n=-\infty}^\infty e(nx)(z+n)^{-k} = \frac{(-2\pi i)^k}{(k-1)!}\sumstack{r\in\Z - x \\ r>0} r^{k-1}e(rz).
	\end{equation*}
\end{lem}
\begin{proof}
	The case $x\in\Z$ is well-known (see e.g.\ \cite[(7.1.9)]{Mi}). For $x=\frac{p}{q}$ with $(p,q) = 1$ write
	\begin{align*}
	    \sum_{n=-\infty}^\infty e(nx)(z+n)^{-k} &= \sum_{m=0}^{q-1}\sum_{n=-\infty}^\infty e((m + qn)x)(z+m + qn)^{-k} \\
	    &= \sum_{m=0}^{q-1}e(mx)\frac{1}{q^k}\sum_{n=-\infty}^\infty\left(\frac{z+m}{q} + n\right)^{-k}.
	\end{align*}
	Applying the equation for $x\in\Z$ with $z$ replaced by $\frac{z+m}{q}$ we get
	\begin{align*}
	    \sum_{m=0}^{q-1}&e(mx)\frac{1}{q^k}\frac{(-2\pi i)^k}{(k-1)!}\sum_{r=1}^\infty r^{k-1}e\left(r\frac{z + m}{q}\right) \\
	    &= \frac{(-2\pi i)^k}{(k-1)!}\sum_{r=1}^\infty \left(\frac{r}{q}\right)^{k-1}e\left(\frac{r}{q}z\right)\frac{1}{q}\sum_{m=0}^{q-1}e\left(m\left(x + \frac{r}{q}\right)\right) \\
	    &= \frac{(-2\pi i)^k}{(k-1)!}\sum_{r=1}^\infty \left(\frac{r}{q}\right)^{k-1}e\left(\frac{r}{q}z\right)\begin{cases} 
	    1 & \text{if $r = -p\bmod q$} \\
	    0 & \text{else}
	    \end{cases} \\
	    &= \frac{(-2\pi i)^k}{(k-1)!}\sumstack{r\in\Z - x \\ r>0} r^{k-1}e(rz).
	\end{align*}
\end{proof}
For $\n\in\Z_{>0}$ define functions $\omega_\n:\mathbb{H}\times\mathbb{H}\rightarrow\C[\mathcal{D}]\otimes\C[\mathcal{D}]$ by
\begin{align*}
	\omega_\n(z,z') \coloneqq \sum_{\gamma\in \mathcal{D}}\sumstack{a,b,c,d\in\Z \\ ad-bc=\n^2 \\ (a,b,c,d) = 1}\frac{1}{(czz'+az+dz'+b)^k}\left(\rho_\mathcal{D}^{(1)}\left(\begin{pmatrix} a & b \\ c & d\end{pmatrix}\right)^{-1}e^\gamma\right)\otimes e^{\gamma}.
\end{align*}
Finally, we also define $\langle\cdot,\cdot\rangle:\C[\D]\times\left(\C[\D]\otimes\C[\D]\right)\to\C[\D]$ by
\[ \langle v,w\otimes u\rangle \coloneqq \langle v,w\rangle\cdot \overline{u} \]
for elements $v,w,u\in\C[\mathcal{D}]$ and extend antilinearly in the second argument. The following proposition adapts \cite[Proposition 1]{Z}, which is originally due to Petersson, to the vector valued case.
\begin{prp}\label{prp:HeckeKernels}
	Let
	\begin{align}\label{eq:Ck}
		C(k) \coloneqq \frac{i^{k}\pi}{2^{k-3}(k-1)}.
	\end{align}
	For $\n\in\Z_{>0}$ the function $\overline{C(k)}^{-1}\n^{2k-2}\omega_\n(z,-\overline{z'})$ is the kernel function for the Hecke operator $T(\n^2)$, i.e.
	\begin{align*}
		C(k)^{-1}\n^{2k-2}\int_{\Gamma\backslash\mathbb{H}}\langle f(z),\omega_\n(z,-\overline{z'})\rangle y^k\frac{\mathrm{d}x\mathrm{d}y}{y^2} = (T(\n^2)f)(z').
	\end{align*}
\end{prp}
\begin{proof}
	First note that we can write
	\begin{align*}
		\omega_{\n}(z,z') &= \sum_{\gamma\in \mathcal{D}}\sumstack{a,b,c,d\in\Z \\ ad-bc = \n^2 \\ (a,b,c,d) = 1}\frac{(cz+d)^{-k}}{(z' + \frac{az+b}{cz+d})^k}\rho_\mathcal{D}^{(1)}\left(\begin{pmatrix}
			a & b \\
			c & d
		\end{pmatrix}^{-1}\right)e^\gamma\otimes e^\gamma \\
		&= \n^{-k}\sum_{\gamma\in \mathcal{D}}\sumstack{M\in M_\n}\frac{1}{(z' + \cdot)^k}\bigm|_k[M](z)\rho_\mathcal{D}^{(1)}(M)^{-1}e^\gamma\otimes e^\gamma.
	\end{align*}
	It is easily seen that $T(\n^2)\omega_1(\cdot,z') = \n^{2k-2}\omega_{\n}(\cdot,z')$. Hence, since the Hecke operators are self-adjoint, it suffices to prove the proposition for $\n=1$. We set $T = \left(\begin{smallmatrix}
		1&1\\0&1
	\end{smallmatrix}\right)$ and $\Gamma_\infty^+ = \langle T\rangle$. Then
	\begin{align*}
		\omega_1(z,z') &= \sum_{\gamma\in \mathcal{D}}\sum_{M\in\Gamma}\frac{1}{(z' + \cdot)^k}\bigm|_k[M](z)\rho_\mathcal{D}(M)^{-1}e^\gamma\otimes e^\gamma \\
		&= \sum_{\gamma\in \mathcal{D}}\sum_{M\in\Gamma_\infty^+\backslash\Gamma}\sum_{n=-\infty}^\infty\frac{1}{(z' + \cdot)^k}\bigm|_k[T^nM](z)\rho_\mathcal{D}(T^nM)^{-1}e^\gamma\otimes e^\gamma \\
		&= \sum_{\gamma\in \mathcal{D}}\sum_{M\in\Gamma_\infty^+\backslash\Gamma}\sum_{n=-\infty}^\infty\frac{e(-n\q(\gamma))}{(z' + \cdot+n)^k}\bigm|_k[M](z)\rho_\mathcal{D}(M)^{-1}e^\gamma\otimes e^\gamma \\
		&= \sum_{\gamma\in \mathcal{D}}\sumstack{M\in\Gamma_\infty^+\backslash\Gamma \\ M = \left(\begin{smallmatrix} a & b \\ c & d \end{smallmatrix}\right)}(cz+d)^{-k}\sum_{n=-\infty}^\infty\frac{e(-n\q(\gamma))}{(z' + Mz+n)^k}\rho_\mathcal{D}(M^{-1})e^\gamma\otimes e^\gamma.
	\end{align*}
	By Lemma \ref{lem:cotSum} this is
	\begin{align*}
		\omega_1(z,z') &= \sum_{\gamma\in \mathcal{D}}\sumstack{M\in\Gamma_\infty^+\backslash\Gamma \\ M = \left(\begin{smallmatrix} a & b \\ c & d \end{smallmatrix}\right)}(cz+d)^{-k}\frac{(-2\pi i)^k}{(k-1)!}\sumstack{r\in\Z + \q(\gamma) \\ r>0} r^{k-1}e(r(z'+Mz))\rho_\mathcal{D}(M^{-1})e^\gamma\otimes e^\gamma \\
		&= \frac{(-2\pi i)^k}{(k-1)!}\sum_{\gamma\in \mathcal{D}}\sumstack{r\in\Z + \q(\gamma) \\ r>0}r^{k-1}\left(\sumstack{M\in\Gamma_\infty^+\backslash\Gamma \\ M = \left(\begin{smallmatrix} a & b \\ c & d \end{smallmatrix}\right)}(cz+d)^{-k}e(rMz)\rho_\mathcal{D}(M^{-1})e^\gamma\right)\otimes e(rz') e^\gamma \\
		&= \frac{(-2\pi i)^k}{(k-1)!}\sum_{\gamma\in \mathcal{D}}\sumstack{r\in\Z + \q(\gamma) \\ r>0}r^{k-1}P_{\gamma,r}(z)\otimes e(rz')e^\gamma,
	\end{align*}
	where $P_{\gamma,r}$ is the Poincaré series of index $(\gamma, r)$ as defined in \cite{Br}. There it is shown that
	\begin{equation*}
		(f,P_{\gamma,r}) = 2\frac{(k-2)!}{(4\pi r)^{k-1}}c(\gamma,r)
	\end{equation*}
	for a cusp form
	\begin{equation*}
		f(\tau) = \sum_{\beta\in \mathcal{D}}\sumstack{r\in\Z+\q(\gamma) \\ r>0}c(\beta,r)e(r\tau)e^\beta.
	\end{equation*}
	We thus have
	\begin{align*}
		C(k)^{-1}(f,\omega_1(z,-\overline{z'})) &=  C(k)^{-1}\frac{(2\pi i)^k}{(k-1)!}\sum_{\gamma\in \mathcal{D}}\sumstack{r\in\Z + \q(\gamma) \\ r>0}r^{k-1}(f,P_{\gamma,r})e(rz') e^\gamma \\
		&= C(k)^{-1}2\frac{(2\pi i)^k}{(k-1)!}\frac{(k-2)!}{(4\pi)^{k-1}}\sum_{\gamma\in \mathcal{D}}\sumstack{r\in\Z+\q(\gamma) \\ r>0}c(\gamma,r)e(r\tau)e^\gamma \\
		&= f(z').
	\end{align*}
\end{proof}

\section{Vector valued Eisenstein series}\label{sec:VectorValuedEisensteinSeries}

In this section we will study a relation between the Eisenstein series for genus $\g=2$ and the Hecke operators for $\g=1$. In the scalar valued case a similar result was shown in \cite{Ga} and \cite{Boe}, called the pullback formula. We generalize the pullback formula to the vector valued case, however using a different approach. In particular, we will prove that $\partial_\h E_{m/2}^{(2)}\left(\left(\begin{smallmatrix}z&0\\0&z'\end{smallmatrix}\right)\right)$ is the sum of the kernel functions for the Hecke operators from the previous section. When $\h=0$ we get, as an additional term, the product of the genus $1$ Eisenstein series in $z$ and in $z'$ (also cf.\ \cite{St}).

\begin{defi}
	Let $\mathcal{D}$ be a discriminant form of even signature, $n\in\Z_{>0}$ and $k\in\Z$ with $k>n+1$. Then the Eisenstein series
	\begin{align*}
		E^{(\g)}_{k}(Z) \coloneqq E^{(\g)}_{k,\mathcal{D}}(Z) \coloneqq \sum_{M = \left(\begin{smallmatrix} A & B \\ C & D \end{smallmatrix}\right)\in\Gamma_\infty^{(\g)}\backslash\Gamma^{(\g)}}\det(CZ+D)^{-k}\rho_\mathcal{D}^{(\g)}(M)^{-1} e^{\underline{0}}
	\end{align*}
	converges normally and thus defines a modular form (cf.\ \cite[Theorem 1]{W2}).
\end{defi}

For a matrix $\left(\begin{smallmatrix}a&b\\c&d\end{smallmatrix}\right)\in\mathrm{Mat}_2(\Z)$, we introduce the notation $\left(\begin{smallmatrix}a&b\\c&d\end{smallmatrix}\right)' = \left(\begin{smallmatrix}d&b\\c&a\end{smallmatrix}\right)$ and $\gcd\left(\left(\begin{smallmatrix}a&b\\c&d\end{smallmatrix}\right)\right) = \gcd(a,b,c,d)$.

\begin{prp}\label{prp:VarPhi}
Let $M = \left(\begin{smallmatrix} A & B \\ C & D\end{smallmatrix}\right)\in\Gamma^{(2)}$ and denote by $C_1$, $C_2$, $D_1$ and $D_2$ the first and second column of $C$ and $D$ respectively. We define the maps
\begin{align*}
    \varphi: \Gamma^{(2)}\rightarrow\mathrm{Mat}_2(\Z),\quad \varphi(M) = \begin{pmatrix}\det(C_1,D_2) & \det(D) \\ \det(C) & \det(D_1,C_2)\end{pmatrix}
\end{align*}
and
\begin{align*}
    \nu : \Gamma^{(2)}\rightarrow\Z,\quad \nu(M) = \det(C_1,D_1).
\end{align*}
Then we have
\begin{align*}
    \varphi(M\cdot u(A)) &= \varphi(M)\cdot A, \\
    \varphi(M\cdot d(A)) &= A'\cdot\varphi(M), \\
    \nu(M\cdot u(A)) &= \nu(M), \\
    \nu(M\cdot d(A)) &= \nu(M)
\end{align*}
for $M\in\Gamma^{(2)}$ and $A \in\Gamma^{(1)}$. The map
\begin{align*}
    \phi:\Gamma_\infty^{(2)}\backslash\Gamma^{(2)}&\rightarrow\{(\delta,\n)\in\mathrm{Mat}_2(\Z)\times\Z \mid\det(\delta) = \n^2,\gcd(\delta) = 1\}/\{(I,1),(-I,-1)\}, \\
    M &\mapsto (\varphi(M),\nu(M))
\end{align*}
is bijective.
\end{prp}
\begin{proof}
    The first relations follow from simple computations. It remains to prove that $\phi$ is a bijection. First we show that it is well-defined, i.e.\ that $\det(\varphi(M)) = \det(C_1,D_1)^2$, that $\gcd(\varphi(M)) = 1$ and that for any $\Tilde{M}\in\Gamma_\infty^{(2)}$ the matrices $\Tilde{M}M$ and $M$ have the same image under $\phi$. The last statement follows immediately from the fact that any $\Tilde{M}\in\Gamma_\infty^{(2)}$ is of the form
    \begin{align*}
        \begin{pmatrix}
        U & B \\ 0 & (U^T)^{-1}
        \end{pmatrix}
    \end{align*}
    for some $U\in\GL_2(\Z)$ and that $\det(U) = \pm1$. Now let $C=\left(\begin{smallmatrix} c_1 & c_2 \\ c_3 & c_4 \end{smallmatrix}\right)$ and $D=\left(\begin{smallmatrix} d_1 & d_2 \\ d_3 & d_4 \end{smallmatrix}\right)$. Then
    \begin{align*}
        \begin{pmatrix} a & b \\ c & d \end{pmatrix} &= \varphi(M) = \begin{pmatrix} c_1d_4-d_2c_3 & d_1d_4-d_2d_3 \\ c_1c_4-c_2c_3 & d_1c_4-c_2d_3 \end{pmatrix}
    \end{align*}
    and so 
    \begin{align*}
        \det(\varphi(M)) &= ad-bc \\
        &= (c_1d_4-d_2c_3)(d_1c_4-c_2d_3)-(d_1d_4-d_2d_3)(c_1c_4-c_2c_3) \\
        &= -c_1c_2d_3d_4 + c_1c_4d_2d_3 + c_2c_3d_1d_4-c_3c_4d_1d_2.
    \end{align*}
    Because $M$ is symplectic, we know that $CD^T = DC^T$, which is equivalent to $c_1d_3 - d_1c_3 = d_2c_4 - c_2d_4$. Therefore, we have
    \begin{align*}
        \det(C_1,D_1)^2 &= (c_1d_3 - d_1c_3)(d_2c_4 - c_2d_4) \\
        &= -c_1c_2d_3d_4 + c_1c_4d_2d_3 + c_2c_3d_1d_4-c_3c_4d_1d_2 \\
        &= \det(\varphi(M)).
    \end{align*}
    By Lemma \ref{lem:DoubleCosetSet} every $\Tilde{M}\in\mathrm{Mat}_2(\Z)$ of determinant $\n^2$ can be written as $A\left(\begin{smallmatrix} \alpha & 0 \\ 0 & \delta\end{smallmatrix}\right)B$ for $A,B\in\Gamma$ and $\alpha\delta = \n^2$ and we have
    \begin{align}\label{eq:varphidoublecoset}
        \varphi(M\cdot d(A')\cdot u(B)) = A\varphi(M)B.
    \end{align}
    We show that there exists an $M\in\Gamma^{(2)}$ such that $\varphi(M) = \left(\begin{smallmatrix} \alpha & 0 \\ 0 & \delta\end{smallmatrix}\right)$ if and only if $(\alpha,\delta) = 1$: \\
    It is well known that a right coset decomposition of the integral $2\times2$ matrices of determinant $c$ is given by
    \begin{align*}
        \bigcup_{\substack{c_1>0 \\ c_1c_4 = c \\ c_2\bmod c_4}}\Gamma\begin{pmatrix} c_1 & c_2 \\ 0 & c_4 \end{pmatrix}.
    \end{align*}
    Hence, if $\varphi(M) = \left(\begin{smallmatrix} \alpha & 0 \\ 0 & \delta\end{smallmatrix}\right)$ for $M\in\Gamma_\infty^{(2)}\backslash\Gamma^{(2)}$, we can choose the representative $M$ such that $C=\left(\begin{smallmatrix} c_1 & c_2 \\ 0 & 0\end{smallmatrix}\right)$. We then find that $c_1d_4 = \alpha$ and $-c_2d_3 = \delta$ and $c_1d_3=-c_2d_4 = \n$. Therefore, $c_1\mid \gcd(\alpha,\n)$. If $x = \gcd(\alpha,\n)/c_1$, then $x\mid d_4$ and $x\mid d_3$. But the last row of $M$ is $(0,0,d_3,d_4)$ so that $x=1$ and $c_1=\gcd(\alpha,\n)$. By an analogous argument $c_2 = \pm\gcd(\delta,\n)$. Using the Laplace expansion along the 3rd row, we find that
    \begin{align*}
        1 = \det(M) &= c_1\cdot\det(\hdots)-c_2\cdot\det(\hdots)+\det(D)\cdot\det(A) \\
        &=c_1\cdot\det(\hdots)-c_2\cdot\det(\hdots),
    \end{align*}
    so that $1 = \gcd(c_1,c_2) = \gcd(\gcd(\alpha,\n),\gcd(\delta,\n)) = \gcd(\alpha,\delta)$. Therefore, $\phi$ is well-defined. \\     
    Recall the definition of
    \[ \mathcal{A}_{\n} = \begin{pmatrix}
    	\n^2+\n & -\n-1 & -1 & -\n-1 \\
    	-\n-1    & 1   & 0 & 0   \\
    	-\n      & 1  & 0 & 0   \\
    	0      & 0   & -1 & -\n
    \end{pmatrix}\in\Gamma^{(2)}. \]
    We have
    \begin{align*}
        \varphi(\mathcal{A}_{\n}) = \begin{pmatrix} \n^2 & 0 \\ 0 & 1 \end{pmatrix}
    \end{align*}
    and $\nu(\mathcal{A}_{\n}) = \n$. Using (\ref{eq:varphidoublecoset}) and Lemma \ref{lem:DoubleCosetSet} this implies that $\phi$ is surjective. For injectivity recall that if $\varphi(M) = \left(\begin{smallmatrix} \n^2 & 0 \\ 0 & 1\end{smallmatrix}\right)$, then we can assume that $C=\left(\begin{smallmatrix} c_1 & c_2 \\ 0 & 0\end{smallmatrix}\right)$ with $c_1 = \gcd(\n^2,\n) = \n$ and $c_2=\pm\gcd(1,\n) = \pm1.$ Then $d_4=\n$, $d_3=\mp1$ and $d_2 = \mp \n d_1$. Now
    \begin{align*}
    	\begin{pmatrix} 1 & b \\ 0 & 1 \end{pmatrix}\cdot\left(C,D\right) &= \begin{pmatrix} 1 & b \\ 0 & 1 \end{pmatrix}\cdot\left(\begin{pmatrix} \n & \pm1 \\ 0 & 0 \end{pmatrix},\begin{pmatrix} d_1 & \mp \n d_1 \\ \mp 1 & \n \end{pmatrix}\right) \\
        &= \left(\begin{pmatrix} \n & \pm1 \\ 0 & 0 \end{pmatrix},\begin{pmatrix} d_1\mp b & \mp \n(d_1\mp b) \\ \mp 1 & \n \end{pmatrix}\right).
    \end{align*}
    Therefore, for every $d_1$ the corresponding matrices $M$ are in the same coset $\bmod \ \Gamma_\infty^{(2)}$. We choose a representative with $d_1 = 0$. Then $\nu(M) = \n$ if and only if $\pm = -$. So there exists exactly one pair $(C,D)\in\GL_2(\Z)\backslash(\mathrm{Mat}_2(\Z)\times\mathrm{Mat}_2(\Z))$ such that for $M=\left(\begin{smallmatrix} A & B \\ C & D \end{smallmatrix}\right)$ we have $\varphi(M) = \left(\begin{smallmatrix} \n^2 & 0 \\ 0 & 1\end{smallmatrix}\right)$ and $\nu(M) = \n$. If $M,\Tilde{M}\in\Gamma^{(2)}$ have identical $C$ and $D$, then $\Tilde{M}M^{-1}\in\Gamma_\infty^{(2)}$ and so $\phi$ is injective.
\end{proof}

For the next proposition we will need 
\begin{lem}\label{lem:APrime}
	Let $\mathcal{D}$ be a discriminant form of even signature, $\n\in\Z$, $A\in M_\n$ and $B\in\SL$. Then $(AB)' = B'A'$ and 
	\begin{align*}
		\sum_{\gamma\in \D}\rho_\D(A)^{-1}e^\gamma\otimes\rho_\D(B)^{-1}e^\gamma = \sum_{\gamma\in \D}\rho_\D(B'A)^{-1}e^\gamma\otimes e^\gamma.
	\end{align*}
\end{lem}
\begin{proof}
	The fact that $(A B)' = B' A'$ follows from a simple calculation. Now let $B_1,B_2\in\SL$ and assume that the identity holds whenever $B$ is equal to $B_1$ or $B_2$. Then for $A\in M_\n$ by Proposition \ref{prp:ActionOfAAndD} also
	\begin{align*}
		\sum_{\gamma\in \D}\rho_\D(A)^{-1}e^\gamma\otimes\rho_\D(B_1B_2)^{-1}e^\gamma &= \rho_\D^{(2)}(d(B_2^{-1}))\sum_{\gamma\in \D}\rho_\D(A)^{-1}e^\gamma\otimes\rho_\D(B_1)^{-1}e^\gamma \\
		&= \rho_\D^{(2)}(d(B_2^{-1}))\sum_{\gamma\in \D}\rho_\D(B_1'A)^{-1}e^\gamma\otimes e^\gamma \\
		&= \sum_{\gamma\in \D}\rho_\D(B_1'A)^{-1}e^\gamma\otimes \rho_\D(B_2)^{-1}e^\gamma \\
		&= \sum_{\gamma\in \D}\rho_\D(B_2' B_1'A)^{-1}e^\gamma\otimes e^\gamma \\
		&= \sum_{\gamma\in \D}\rho_\D((B_1 B_2)'A)^{-1}e^\gamma\otimes e^\gamma,
	\end{align*}
	so it suffices to prove the identity for $B$ equal to generators of $\SL$, i.e.\ $J_1 = \left(\begin{smallmatrix} 0 & 1 \\ -1 & 0 \end{smallmatrix}\right)$ and $n(1) = \left(\begin{smallmatrix} 1 & 1 \\ 0 & 1 \end{smallmatrix}\right)$. Note that both $J_1'=J_1$ and $n(1)' = n(1)$. For $n(1)$ the identity is trivial, while for $J_1$ we have
	\begin{align*}
		\sum_{\gamma\in \D}\rho_\D(A)^{-1}e^\gamma\otimes\rho_\D(J_1)^{-1}e^\gamma &= \sum_{\gamma\in \D}\rho_\D(A)^{-1}e^\gamma\otimes\frac{e(-\sign(\D)/8)}{|\D|}\sum_{\beta\in \D}e(-(\beta,\gamma))e^\beta \\
		&= \sum_{\beta\in \D}\rho_\D(A)^{-1}\frac{e(-\sign(\D)/8)}{|\D|}\sum_{\gamma\in \D}e(-(\beta,\gamma))e^\gamma\otimes e^\beta \\
		&= \sum_{\beta\in \D}\rho_\D(A)^{-1}\rho_\D(J_1)^{-1}e^\beta\otimes e^\beta.
	\end{align*}
\end{proof}

We can now show that in a special case, the action of an element $M\in\Gamma^{(2)}$ in the Weil representation is given in terms of the action of $\varphi(M)$.
\begin{prp}\label{prp:ActionOne0}
	Let $\D$ be a discriminant form of even signature. Let $M\in\Gamma^{(2)}$ and $\epsilon = \sgn(\nu(M))$ with $\sgn(0) = -1$. Then
	\begin{align*}
		\rho_\D^{(2)}(M)^{-1}(e^0\otimes e^0) = \frac{e(\sign(\D)/8)}{\sqrt{|\D|}}\sum_{\gamma\in \D}\rho_\D(\varphi(M))^{-1}e^{-\epsilon\gamma}\otimes e^\gamma.
	\end{align*}
\end{prp}
\begin{proof}
	We first show the statement for $M=\mathcal{A}_\n = Jn_1Jn_2Jn_3$ with
	\begin{align*}
		n_1 &= n\left(\begin{pmatrix} 0 & -1 \\ -1 & -\n \end{pmatrix}\right), \ n_2 = n\left(\begin{pmatrix} \n^2+\n & -\n-1 \\ -\n-1 & 1 \end{pmatrix}\right), \ n_3 = n\left(\begin{pmatrix} 0 & 0 \\ 0 & 1 \end{pmatrix}\right).
	\end{align*}
	Then we find
	\begin{align*}
		\rho_\D^{(2)}(\mathcal{A}_\n)^{-1}(e^0\otimes e^0) &= \rho_\D^{(2)}(n_1Jn_2Jn_3)^{-1}\frac{e(-\sign(\D)/4)}{|\D|}\sum_{(\mu_1,\mu_2)\in \D^2}e^{\mu_1}\otimes e^{\mu_2} \\
		&= \rho_\D^{(2)}(Jn_2Jn_3)^{-1}\frac{e(-\sign(\D)/4)}{|\D|}\sum_{(\mu_1,\mu_2)\in \D^2}e((\mu_1,\mu_2) + \n\q(\mu_2))e^{\mu_1}\otimes e^{\mu_2} \\
		&= \rho_\D^{(2)}(n_2Jn_3)^{-1}\frac{e(-\sign(\D)/2)}{|\D|^2}\sum_{(\beta_1,\beta_2)\in \D^2}\sum_{(\mu_1,\mu_2)\in \D^2}e(-(\mu_1,\beta_1)-(\mu_2,\beta_2)) \\
		&\qquad\qquad e((\mu_1,\mu_2) + \n\q(\mu_2))e^{\beta_1}\otimes e^{\beta_2} \\
		&= \rho_\D^{(2)}(n_2Jn_3)^{-1}\frac{e(-\sign(\D)/2)}{|\D|}\sum_{(\beta_1,\beta_2)\in \D^2}e(-(\beta_1,\beta_2)+\n\q(\beta_1))e^{\beta_1}\otimes e^{\beta_2},
	\end{align*}
	where we used that
	\begin{align*}
		\sum_{\mu_1\in \D}e((\mu_1,\mu_2-\beta_1) = \begin{cases}
			|\D| & \text{if $\mu_2 = \beta_1$} \\
			0 & \text{otherwise}.
		\end{cases}
	\end{align*}
	We proceed% by applying $\rho_\D^{(2)}(n_2)^{-1}$ and
	\begin{align*}
		\rho_\D^{(2)}(Jn_3)^{-1}&\frac{e(-\sign(\D)/2)}{|\D|}\sum_{(\beta_1,\beta_2)\in \D^2}e(-(\n^2+\n)\q(\beta_1)+(\n+1)(\beta_1,\beta_2)-\q(\beta_2)) \\
		&\qquad\qquad e(-(\beta_1,\beta_2)+\n\q(\beta_1))e^{\beta_1}\otimes e^{\beta_2} \\
		&= \rho_\D^{(2)}(n_3)^{-1}\frac{e(-3\sign(\D)/4)}{|\D|^2}\sum_{(\gamma_1,\gamma_2)\in \D^2}\sum_{(\beta_1,\beta_2)\in \D^2}e(-(\beta_1,\gamma_1)-(\beta_2,\gamma_2)) \\
		&\qquad\qquad e(-\q(\n\beta_1)+\n(\beta_1,\beta_2)-\q(\beta_2))e^{\gamma_1}\otimes e^{\gamma_2} \\
		&= \frac{e(-3\sign(\D)/4)}{|\D|^2}\sum_{(\gamma_1,\gamma_2)\in \D^2}\sum_{(\beta_1,\beta_2)\in \D^2}e(-\q(\beta_2-\n\beta_1+\gamma_2)) \\
		&\qquad\qquad e(-(\gamma_1+\n\gamma_2,\beta_1))e^{\gamma_1}\otimes e^{\gamma_2}.
	\end{align*}
	Taking the sum over $\beta_2$ and using Milgram's formula we get
	\begin{align*}
		\frac{e(\sign(\D)/8)}{|\D|^{\frac{3}{2}}}&\sum_{(\gamma_1,\gamma_2)\in \D^2}\sum_{\beta_1\in \D}e(-(\gamma_1+\n\gamma_2,\beta_1))e^{\gamma_1}\otimes e^{\gamma_2} \\
		&= \frac{e(\sign(\D)/8)}{\sqrt{|\D|}}\sum_{\gamma_2\in \D}e^{-\n\gamma_2}\otimes e^{\gamma_2} \\
		&= \frac{e(\sign(\D)/8)}{\sqrt{|\D|}}\sum_{\gamma_2\in \D}\rho_\D\left(\begin{pmatrix} \n^2 & 0 \\ 0 & 1 \end{pmatrix}\right)^{-1}e^{-\epsilon\gamma_2}\otimes e^{\gamma_2}.
	\end{align*}
	(Note that $\rho_\D\left(\left(\begin{smallmatrix} \n^2 & 0 \\ 0 & 1 \end{smallmatrix}\right)\right)^{-1}e^\gamma = e^{|\n|\gamma}$ by definition.) \\
	Now let $M\in\Gamma^{(2)}$ be arbitrary. By Proposition \ref{prp:VarPhi} and Lemma \ref{lem:DoubleCosetSet} we have
	\[ M \in \Gamma_\infty^{(2)}\cdot\mathcal{A}_{\n}u(A)d(B) \]
	for some $\n\in\Z$ and $A,B\in\Gamma^{(1)}$. Since $\Tilde{M} = n(S)a(U)\in\Gamma_\infty^{(2)}$ acts on $e^0\otimes e^0$ as multiplication by $\det(U)^{\sign(\D)/2}$ and $\varphi(\Tilde{M}M) = \det(U)\varphi(M)$ and $\nu(\Tilde{M}M) = \det(U)\nu(M)$, we can assume that $M$ is equal to $\mathcal{A}_{\n}u(A)d(B)$. By Proposition \ref{prp:ActionOfAAndD} and Lemma \ref{lem:APrime} we have
	\begin{align*}
		\rho_\D^{(2)}(\mathcal{A}_\n u(A)d(B))^{-1}&(e^0\otimes e^0) \\
		&= \rho_\D^{(2)}(u(A)\cdot d(B))^{-1}\rho_\D^{(2)}(\mathcal{A}_\n)^{-1}(e^0\otimes e^0) \\
		&= \frac{e(\sign(\D)/8)}{\sqrt{|\D|}}\sum_{\gamma\in \D}\rho_\D(A)^{-1}\rho_\D(\varphi(\mathcal{A}_\n))^{-1}e^{-\epsilon\gamma}\otimes\rho_\D(B)^{-1}e^\gamma \\
		&= \frac{e(\sign(\D)/8)}{\sqrt{|\D|}}\sum_{\gamma\in \D}\rho_\D(B'\varphi(\mathcal{A}_\n)A)^{-1}e^{-\epsilon\gamma}\otimes e^\gamma \\
		&= \frac{e(\sign(\D)/8)}{\sqrt{|\D|}}\sum_{\gamma\in \D}\rho_\D(\varphi(\mathcal{A}_\n u(A)d(B)))^{-1}e^{-\epsilon\gamma}\otimes e^\gamma.
	\end{align*}
\end{proof}
Recall that $M_\n = \left\{\matr{a}{b}{c}{d}\in\mathrm{Mat}_2(\Z) \mid ad-bc=\n^2, \gcd(a,b,c,d) = 1\right\}$. For the case $\n=0$ we will need
\begin{lem}\label{lem:PsiFunction}
	We define a function $\psi:\Gamma^{(1)}\times\Gamma^{(1)}\rightarrow M_0$ by
	\begin{align*}
		\left(\begin{pmatrix} * & * \\ r & s \end{pmatrix},\begin{pmatrix} * & * \\ t & u \end{pmatrix}\right) \mapsto \begin{pmatrix} ru & su \\ rt & st \end{pmatrix}.
	\end{align*}
	Then $\psi$ defines a bijection between  $\Gamma_\infty^{(1)}\backslash\Gamma^{(1)}\times\Gamma_\infty^{(1)}\backslash\Gamma^{(1)}$ and $M_0/\{\pm1\}$. Let $\mathcal{D}$ be a discriminant form of even signature, then we have
	\begin{align*}
		\rho_\D(A)^{-1}e^0\otimes\rho_\D(B)^{-1}e^0 &= \frac{e(\sign(\D)/8)}{\sqrt{|\D|}}\sum_{\gamma\in \D}\rho_\D(\psi(A,B))^{-1}e^\gamma\otimes e^\gamma.
	\end{align*}
\end{lem}
\begin{proof}
	It is easy to see that $\psi$ is well-defined as a mapping from $\Gamma_\infty^{(1)}\backslash\Gamma^{(1)}\times\Gamma_\infty^{(1)}\backslash\Gamma^{(1)}$ to $M_0/\{\pm1\}$. A simple computation shows that for $A,B,C,D\in\Gamma$ we have
	\begin{align*}
		\psi(A\cdot C,B\cdot D) = D'\cdot\psi(A,B)\cdot C.
	\end{align*}
	Noting that $\psi(I,-J_1) = \matr{0}{0}{0}{1}$, surjectivity follows from Lemma \ref{lem:DoubleCosetSet}. For injectivity assume that $\psi(A,B) = \psi(C,D)$ and so
	\[ \psi(AC^{-1},BD^{-1}) = \psi(I,I) = \begin{pmatrix} 0 & 1 \\ 0 & 0 \end{pmatrix}. \]
	We need to show that $AC^{-1},BD^{-1}\in\Gamma_\infty$: \\
	Suppose that
	\begin{align*}
		AC^{-1} = \begin{pmatrix} * & * \\ r & s \end{pmatrix},\quad BD^{-1} = \begin{pmatrix} * & * \\ t & u \end{pmatrix}.
	\end{align*}
	Then $ru = 0$ and $su = 1$, so that we must have $r=0$. Because $st=0$ and $su=1$, we also have $t = 0$ and $s=u=\pm1$, i.e.\ $AC^{-1},BD^{-1}\in\Gamma_\infty$. \\	
	Now we find that 
	\begin{align*}
		\rho_\D(I)^{-1}e^0\otimes\rho_\D(-J_1)^{-1}e^0 &= \frac{e(\sign(\D)/8)}{\sqrt{|\D|}}\sum_{\gamma\in \D}e^0\otimes e^\gamma \\
		&= \frac{e(\sign(\D)/8)}{\sqrt{|\D|}}\sum_{\gamma\in \D}\rho_\D\left(\begin{pmatrix} 0 & 0 \\ 0 & 1 \end{pmatrix}\right)^{-1}e^\gamma\otimes e^\gamma.
	\end{align*}
	For arbitrary $A,B\in\Gamma$ we once again use Proposition \ref{prp:ActionOfAAndD} and Lemma \ref{lem:APrime} to get
	\begin{align*}
		\rho_\D(A)^{-1}e^0\otimes\rho_\D(B)^{-1}e^0 &= \rho_\D^{(2)}(u(A)d(J_1B))^{-1}(\rho_\D(I)^{-1}e^0\otimes\rho_\D(-J_1)^{-1}e^0) \\
		&= \rho_\D^{(2)}(u(A)d(J_1B))^{-1}\left(\frac{e(\sign(\D)/8)}{\sqrt{|\D|}}\sum_{\gamma\in \D}\rho_\D(\psi(I,-J_1))^{-1}e^\gamma\otimes e^\gamma\right) \\
		&= \frac{e(\sign(\D)/8)}{\sqrt{|\D|}}\sum_{\gamma\in \D}\rho_\D(A)^{-1}\rho_\D(\psi(I,-J_1))^{-1}e^\gamma\otimes\rho_\D(J_1B)^{-1}e^\gamma \\
		&= \frac{e(\sign(\D)/8)}{\sqrt{|\D|}}\sum_{\gamma\in \D}\rho_\D(B'J_1\psi(I,-J_1)A)^{-1}e^\gamma\otimes e^\gamma \\
		&= \frac{e(\sign(\D)/8)}{\sqrt{|\D|}}\sum_{\gamma\in \D}\rho_\D(\psi(A,B))^{-1}e^\gamma\otimes e^\gamma.
	\end{align*}
\end{proof}

The next result finally describes the relation between $E_k^{(2)}$ and the Hecke operators and is due to Stein (cf.\ \cite[Theorem 5.3]{St}). We include it together with its proof because we will generalize the argument to establish Theorem \ref{thm:diffE}.
\begin{thm}\label{thm:EIsHeckeKernel}
Let $\mathcal{D}$ be a discriminant form of even signature and let $k>3$ with $k=\sign(\mathcal{D})/2\bmod4$. Then we have
\begin{align*}
    E^{(2)}_{k}\left(\matr{z}{0}{0}{z'}\right) &= E^{(1)}_{k}(z)\otimes E^{(1)}_{k}(z') + \frac{e(\sign(\mathcal{D})/8)}{\sqrt{|\mathcal{D}|}}\sum_{\n\in\Z_{>0}}\omega_\n(z,z').
\end{align*}
\end{thm}
\begin{proof}
    Let $M=\left(\begin{smallmatrix} A & B \\ C & D \end{smallmatrix}\right)\in\Gamma^{(2)}$ and consider $C=\left(\begin{smallmatrix} c_1 & c_2 \\ c_3 & c_4 \end{smallmatrix}\right)$ and $D=\left(\begin{smallmatrix} d_1 & d_2 \\ d_3 & d_4 \end{smallmatrix}\right)$. Then for $Z = \left(\begin{smallmatrix} z & 0 \\ 0 & z' \end{smallmatrix}\right)$ 
    \begin{align*}
        \det(CZ+D) &= \det\left(\begin{pmatrix} c_1z+d_1 & c_2z'+d_2 \\ c_3z+d_3 & c_4z'+d_4 \end{pmatrix}\right) \\
        &= (c_1z+d_1)(c_4z'+d_4) - (c_2z'+d_2)(c_3z+d_3) \\
        &= (c_1c_4 - c_2c_3)zz' + (c_1d_4 - d_2c_3)z + (d_1c_4 - c_2d_3)z' + (d_1d_4-d_2d_3) \\
        &= czz'+az+dz'+b,
    \end{align*}
    where $\left(\begin{smallmatrix} a & b \\ c & d \end{smallmatrix}\right) = \varphi(M)$. Hence, $E^{(2)}_{k}\left(\left(\begin{smallmatrix} z & 0 \\ 0 & z' \end{smallmatrix}\right)\right)$ is equal to
    \begin{align*}
    	 \sumstack{M\in\Gamma_\infty^{(2)}\backslash\Gamma^{(2)} \\ \varphi(M)=\left(\begin{smallmatrix} a & b \\ c & d \end{smallmatrix}\right)}&\frac{1}{(czz'+az+dz'+b)^k}\rho_\mathcal{D}^{(2)}(M)^{-1}(e^0\otimes e^0) \\
    	&= \frac{e(\sign(\mathcal{D})/8)}{\sqrt{|\mathcal{D}|}}\sumstack{M\in\Gamma_\infty^{(2)}\backslash\Gamma^{(2)} \\ \varphi(M)=\left(\begin{smallmatrix} a & b \\ c & d \end{smallmatrix}\right)}\frac{1}{(czz'+az+dz'+b)^k}\sum_{\gamma\in \mathcal{D}}\rho_\mathcal{D}(\varphi(M))^{-1}e^{-\sgn(\nu(M))\gamma}\otimes e^\gamma \\
    	&= \frac{e(\sign(\mathcal{D})/8)}{\sqrt{|\mathcal{D}|}}\sum_{\n\in\Z_{>0}}\sumstack{M=\left(\begin{smallmatrix} a & b \\ c & d \end{smallmatrix}\right)\in M_\n}\frac{1}{(czz'+az+dz'+b)^k}\sum_{\gamma\in \mathcal{D}}\rho_\mathcal{D}(M)^{-1}e^\gamma\otimes e^{\gamma} \\
    	&\quad +\frac{e(\sign(\mathcal{D})/8)}{\sqrt{|\mathcal{D}|}}\sumstack{M=\left(\begin{smallmatrix} a & b \\ c & d \end{smallmatrix}\right)\in M_0/\{\pm1\}}\frac{1}{(czz'+az+dz'+b)^k}\sum_{\gamma\in \mathcal{D}}\rho_\mathcal{D}(M)^{-1}e^\gamma\otimes e^{\gamma},
    \end{align*}
    where we applied Propositions \ref{prp:VarPhi} and \ref{prp:ActionOne0} as well as the fact that
    \begin{align*}
        \frac{1}{(-czz'-az-dz'-b)^k}&\sum_{\gamma\in \mathcal{D}}\rho_\mathcal{D}(-\varphi(M))^{-1}e^{-\gamma}\otimes e^{\gamma} \\
        &= \frac{(-1)^k}{(czz'+az+dz'+b)^k}\sum_{\gamma\in \mathcal{D}}\rho_\mathcal{D}(\varphi(M))^{-1}\rho_\mathcal{D}(-I)e^{-\gamma}\otimes e^{\gamma} \\
        &= \frac{(-1)^ke(\sign(\mathcal{D})/4)}{(czz'+az+dz'+b)^k}\sum_{\gamma\in \mathcal{D}}\rho_\mathcal{D}(\varphi(M))^{-1}e^{\gamma}\otimes e^{\gamma}
    \end{align*}
    and $e(\sign(\mathcal{D})/4) = i^{\sign(\mathcal{D})} = i^{2k} = (-1)^k$. So it remains to show that the term for $\n=0$ is equal to $E^{(1)}_{k}(z)\otimes E^{(1)}_{k}(z')$. In fact, by Lemma \ref{lem:PsiFunction}, we have
    \begin{align*}
        &\frac{e(\sign(\mathcal{D})/8)}{\sqrt{|\mathcal{D}|}}\sumstack{M\in M_0/\{\pm1\}}\frac{1}{(czz'+az+dz'+b)^k}\sum_{\gamma\in \mathcal{D}}\rho_\mathcal{D}(M)^{-1}e^\gamma\otimes e^\gamma \\
        &= \frac{e(\sign(\mathcal{D})/8)}{\sqrt{|\mathcal{D}|}}\sumstack{A,B\in\Gamma_\infty^{(1)}\backslash\Gamma^{(1)} \\ \psi(A,B)=\left(\begin{smallmatrix} a & b \\ c & d \end{smallmatrix}\right)}\frac{1}{(czz'+az+dz'+b)^k}\sum_{\gamma\in \mathcal{D}}\rho_\mathcal{D}(\psi(A,B))^{-1}e^\gamma\otimes e^\gamma \\
        &= \sumstack{A,B\in\Gamma_\infty^{(1)}\backslash\Gamma^{(1)} \\ \psi(A,B)=\left(\begin{smallmatrix} a & b \\ c & d \end{smallmatrix}\right)}\frac{1}{(czz'+az+dz'+b)^k}\rho_\mathcal{D}(A)^{-1}e^0\otimes\rho_\mathcal{D}(B)^{-1}e^0 \\
        &= \left(\sum_{A=\left(\begin{smallmatrix} * & * \\ r & s \end{smallmatrix}\right)\in\Gamma_\infty^{(1)}\backslash\Gamma^{(1)}}\frac{1}{(rz+s)^k}\rho_\mathcal{D}(A)^{-1}e^0\right)\otimes\left(\sum_{B=\left(\begin{smallmatrix} * & * \\ t & u \end{smallmatrix}\right)\in\Gamma_\infty^{(1)}\backslash\Gamma^{(1)}}\frac{1}{(tz'+u)^k}\rho_\mathcal{D}(B)^{-1}e^0\right) \\
        &= E_k^{(1)}(z)\otimes E_k^{(1)}(z').
    \end{align*}
\end{proof}

By first applying $\partial_\h$ to the Eisenstein series we obtain

\begin{thm}\label{thm:diffE}
	Let $\mathcal{D}$ be a discriminant form of even signature, $m>6$ with $m=\sign(D)\bmod8$ and $k=m/2+\h$ with $\h\geq0$. Then
	\begin{align*}
		(\partial_{\h} E_{m/2}^{(2)})(z,z') = G_m^\h(1,1)\cdot\frac{(k-1)!}{(m/2-1)!}\cdot\frac{e(\sign(\mathcal{D})/8)}{\sqrt{|\mathcal{D}|}}\sum_{\n\in\Z_{>0}}\n^\h\omega_\n(z,z').
	\end{align*}
\end{thm}
\begin{proof}
	We have already seen that for any $M\in\Gamma_\infty^{(2)}\backslash\Gamma^{(2)}$ we can find a representative such that $M = \mathcal{A}_\n u(A)d(B)$ for suitable $A,B\in\Gamma^{(1)}$ and $\n\in\Z_{\geq0}$. Hence,
	\begin{align*}
		E_{m/2}^{(2)}(Z) = \sum_{\n\in\Z_{\geq0}}\sum_{A,B}1|_{m/2}[\mathcal{A}_\n u(A)d(B)]\rho_\mathcal{D}^{(2)}(M)^{-1}(e^{0}\otimes e^0),
	\end{align*}
	where the second sum ranges over the appropriate $A$ and $B$ (which depend on $\n$). It was shown in \cite[section 3.1.1]{I} that $\partial_{\h}$ commutes with the slash operator of $u(A)$ and $d(B)$ on $\mathbb{H}\times\mathbb{H}\subset\mathbb{H}_2$, i.e.\ for a $f:\mathbb{H}_2\rightarrow\C$ we have
	\begin{align*}
		\partial_{\h}(f|_{m/2}[u(A)]) &= (\partial_{\h}f)|_{m/2+h}^1[A]\quad\text{and} \\
		\partial_{\h}(f|_{m/2}[d(B)]) &= (\partial_{\h}f)|_{m/2+h}^2[B],
	\end{align*}
	where $|_k^1$ acts on the first and $|_k^2$ on the second variable (recall that $(\partial_\h f):\mathbb{H}\times\mathbb{H}\rightarrow\C$). Therefore, we obtain
	\begin{align*}
		(\partial_\h E_{m/2}^{(2)})(z_1,z_4) &= \sum_{\n\in\Z_{\geq0}}\sum_{A,B}\partial_\h (1|_{m/2}[\mathcal{A}_\n a(A)d(B)]\rho_\mathcal{D}^{(2)}(M)^{-1}(e^{0}\otimes e^0) \\
		&= \sum_{\n\in\Z_{\geq0}}\sum_{A,B}(\partial_\h 1|_{m/2}[\mathcal{A}_{\n}])|_{k}^1[A]|_{k}^2[B]\rho_\mathcal{D}^{(2)}(M)^{-1}(e^{0}\otimes e^0).
	\end{align*}
	Since
	\begin{align*}
		\mathcal{A}_\n &=
		\begin{pmatrix}
		\n^2+\n & -\n-1 & -1 & -\n-1 \\
		-\n-1    & 1   & 0 & 0   \\
		-\n      & 1  & 0 & 0   \\
		0      & 0   & -1 & -\n
		\end{pmatrix},
	\end{align*}
	we have
	\begin{align*}
		1|_{m/2}[\mathcal{A}_\n](Z) &= \det\left(\begin{pmatrix} -\n z_1+z_2 & -\n z_2+z_4 \\ -1 & -\n \end{pmatrix}\right)^{-m/2} = (\n^2z_1-2\n z_2+z_4)^{-m/2}.
	\end{align*}
	Then
	\begin{align*}
		\frac{\partial}{\partial z_2}(\n^2z_1-2\n z_2+z_4)^{-s} &= 2\n s(\n^2z_1-2\n z_2+z_4)^{-s-1}
	\end{align*}
	and
	\begin{align*}
		\frac{\partial^2}{\partial z_1\partial z_4}(\n^2z_1-2\n z_2+z_4)^{-s} &= \frac{\partial}{\partial z_4}\n^2(-s)(\n^2z_1-2\n z_2+z_4)^{-s-1} \\
		&= \n^2s(s+1)(\n^2z_1-2\n z_2+z_4)^{-s-2}.
	\end{align*}
	Since $G_m^\h(x,y^2)$ is homogeneous of degree $\h$, we find
	\begin{align*}
		G_m^\h\left(\frac{1}{2}\frac{\partial}{\partial z_2},\frac{\partial^2}{\partial z_1\partial z_4}\right)&(\n^2z_1-2\n z_2+z_4)^{-m/2} \\
		&= G_m^\h(\n,\n^2)\cdot m/2\cdot\hdots\cdot(k-1)\cdot(\n^2z_1-2\n z_2+z_4)^{-k} \\
		&= G_m^\h(1,1)\cdot\frac{(k-1)!}{(m/2-1)!}\cdot \n^\h\cdot 1|_{k}[\mathcal{A}_\n](Z).
	\end{align*}
	Hence
	\begin{align*}
		(\partial_\h E_{m/2}^{(2)})(z_1,z_4) &= G_m^\h(1,1)\cdot\frac{(k-1)!}{(m/2-1)!}\cdot\sum_{\n\in\Z_{\geq0}}\n^\h \\
		&\qquad\qquad\qquad\sum_{A,B}1|_{k}[\mathcal{A}_\n u(A)d(B)]\left(\left(\begin{smallmatrix}
			z_1&0\\0&z_4
		\end{smallmatrix}\right)\right)\rho_\mathcal{D}^{(2)}(M)^{-1}(e^{0}\otimes e^0) \\
		&= G_m^\h(1,1)\cdot\frac{(k-1)!}{(m/2-1)!}\cdot\frac{e(\sign(\mathcal{D})/8)}{\sqrt{|\mathcal{D}|}}\cdot\sum_{\n\in\Z_{>0}}\n^\h \\
		&\qquad\qquad\qquad\sumstack{M=\left(\begin{smallmatrix} a & b \\ c & d \end{smallmatrix}\right)\in M_\n}\frac{1}{(czz'+az+dz'+b)^k}\sum_{\gamma\in \mathcal{D}}\rho_\mathcal{D}(M)^{-1}e^\gamma\otimes e^{\gamma} \\
		&= G_m^\h(1,1)\cdot\frac{(k-1)!}{(m/2-1)!}\cdot\frac{e(\sign(\mathcal{D})/8)}{\sqrt{|\mathcal{D}|}}\sum_{\n\in\Z_{>0}}n^\h\omega_\n(z,z'),
	\end{align*}
	where the last two steps are the same as in Theorem \ref{thm:EIsHeckeKernel}.
\end{proof}

\section{The vector valued Siegel--Weil formula}\label{sec:SiegelWeil}

In this section we will see that the genus theta series defined in section \ref{sec:TheWeilRepresentation} is equal to the Eisenstein series defined in section \ref{sec:VectorValuedEisensteinSeries}. This is essentially an application of the Siegel--Weil formula shown by Weil in \cite{W2}. Some missing cases were later completed by Kudla and Rallis in \cite{KR}. The Siegel--Weil formula is formulated in an adelic setup that contains the classical case that we are interested in. We will state the Siegel--Weil formula and show how to deduce the equality of the genus theta series and the Eisenstein series from it. For this purpose we will first introduce this more general setup. We mostly follow \cite{Li}. Most of the calculations in this section are known, but there seems to be no reference giving the result in the generality that we require.

Consider $\Z^{2\g}$ together with the alternating bilinear form given by $(x,y)\mapsto x^TJ_\g y$ and let $\mathrm{Sp}_{2\g}$ be the corresponding symplectic $\Z$-group scheme. We denote by $P = A N \subset \mathcal{G} = \mathrm{Sp}_{2\g}$ the standard Siegel parabolic subgroup, so that we have
\begin{align*}
    A(R)&=\left\{a(u) = \begin{pmatrix} u & 0 \\ 0 & (u^T)^{-1} \end{pmatrix}\bigm\vert u\in\GL_\g(R)\right\}, \\
    N(R)&=\left\{n(s) = \begin{pmatrix} I_\g & s \\ 0 & I_\g \end{pmatrix}\bigm\vert s\in\mathrm{Sym}_\g(R)\right\}
\end{align*}
for any ring $R\supset\Z$. Let $V$ be a quadratic space over $\Q$ of dimension $m=2k$ with non-degenerate symmetric bilinear form $(\cdot,\cdot)$ and $V(R) = V\otimes_\Q R$ for any ring $R\supset\Q$. Let $\Ort(V)$ be the corresponding orthogonal group scheme. Then $(\mathcal{G}, \mathcal{H}) = (\mathrm{Sp}_{2\g}, \Ort(V))$ is a reductive dual pair. 

For a prime $p$ let $\Q_p$ denote the completion of $\Q$ at $p$ and $\Z_p$ the ring of integers of $\Q_p$. Let $\hat{\Z} = \prod_{p<\infty}\Z_p$ and $\A = \A_f\times\R$ be the ring of adeles of $\Q$. Let $\mathscr{S}(V(\A)^\g)$ be the space of Schwartz functions on $V(\A)^n$. For an additive character $\psi:\A\rightarrow\C^\times$ we will define a representation $\omega = \omega_{V,\psi}$ of $\mathcal{G}(\A)\times \mathcal{H}(\A)$ on this space that we will also call the Weil representation. It generalizes the representations described in section \ref{sec:TheWeilRepresentation}. Let $\mathrm{disc}(V) \in \Q^\times/(\Q^\times)^2$ be the discriminant of $V$ defined to be
\begin{align*}
	\mathrm{disc}(V) \coloneqq (-1)^{k}\det((x_i,x_j))_{i,j=1}^{2k}
\end{align*}
for any $\Q$-basis $\{x_1,\hdots, x_{2k} \}$ of $V$ and let $\chi_V:\A^\times/\Q^\times\rightarrow\C^\times$ be the quadratic character that corresponds to the quadratic extension $\Q(\sqrt{\mathrm{disc}(V)})/\Q$, i.e.\ $\chi_V(x) = (x,\mathrm{disc}(V))$, where $(\cdot,\cdot)$ is the Hilbert symbol. We denote by $|\cdot|:\A^\times\rightarrow\R_{>0}$ the normalized absolute value and define $(\underline{x},\underline{y}) \coloneqq ((x_i,y_j))_{i,j=1}^{\g}\in\mathrm{Mat}_\g(\A)$ for $\underline{x} = (x_1,\hdots,x_\g)\in V(\A)^\g$ and $\underline{y} = (y_1,\hdots,y_\g)\in V(\A)^\g$. Finally, denote by $\hat{\varphi}$ the Fourier transform of $\varphi$ using the self-dual Haar measure on $V(\A)^\g$ with respect to $\psi$, i.e.\
\begin{align*}
	\hat{\varphi}(\underline{x}) = \int_{V(\A)^\g}\varphi(\underline{y})\psi(\tr(\underline{x},\underline{y}))d\underline{y}.
\end{align*}
Let us now fix the standard additive character $\psi:\A\rightarrow\C^\times$ whose archimedean component is given by $\psi_\infty:\R\rightarrow\C^\times, \ x_\infty\mapsto e(x_\infty)$ and the finite components by $\psi_p:\Q_p\rightarrow\C^\times, \ x_p\mapsto e(-x_p')$, where $x_p'\in\Q/\Z$ is the principal part of $x_p$. The Weil representation $\omega = \omega_{V,\psi}$ is the representation of $\mathcal{G}(\A)\times \mathcal{H}(\A)$ on the space of Schwartz functions $\mathscr{S}(V(\A)^\g)$ that is uniquely determined by the following equations. Let $\varphi\in\mathscr{S}(V(\A)^\g)$ and $\underline{x}\in V(\A)^\g$, then
\begin{alignat*}{2}
    \omega(a(u))\varphi(\underline{x}) &= \chi_V(\det(u))|\det(u)|^k\varphi(\underline{x}\cdot u),\qquad &&a(u)\in A(\A), \\
    \omega(n(s))\varphi(\underline{x}) &= \psi(\textstyle{\frac{1}{2}}\tr(s(\underline{x},\underline{x})))\varphi(\underline{x}),\qquad &&n(s)\in N(\A), \\
    \omega(J_\g)\varphi(\underline{x}) &= \hat{\varphi}(\underline{x}), \\
    \omega(h)\varphi(\underline{x}) &= \varphi(h^{-1}\cdot\underline{x}),\qquad &&h\in H(\A).
\end{alignat*}
Note that for the local components $\omega_p$ we define $\omega_p(J_\g)\varphi_p(\underline{x}) = \gamma_p\hat{\varphi_p}(\underline{x})$, where $\gamma_p$ is an eighth root of unity. If $(t_+,t_-)$ is the signature of the quadratic space $V$, then $\gamma_\infty = e((t_+-t_-)/8)$ and \[ e((t_+-t_-)/8)\prod_{p<\infty}\gamma_p = 1. \] 
If $L$ is an even lattice in $V$, then $\rho_{L'/L}$ or rather its dual $\overline{\rho}_{L'/L}$ can be viewed as a subrepresentation of $\omega_{V,\psi}$. More precisely, let $L$ be a positive-definite even lattice of rank $2k$ and set $V\coloneqq L\otimes\Q$. Consider the subspace $\mathscr{S}_L$ of Schwartz functions in $\mathscr{S}(V(\A_f)^\g)$ which are supported on $(L')^\g\otimes\hat{\Z}$ and which are constant on cosets of $L^\g \otimes\hat{\Z}$. There is an isomorphism 
\[ \iota:\C[(L'/L)^\g]\to\mathscr{S}_L \]
given by mapping a basis element $e^{\underline{\mu}}$ of $\C[(L'/L)^\g]$ to $\varphi_{\underline{\mu}} = \bigotimes_{p<\infty}\varphi_p$, where $\varphi_p\in\mathscr{S}(V(\Q_p)^\g)$ is the characteristic function of $\underline{\mu} + (L\otimes\Z_p)^\g$. We obtain (cf.\ \cite[section 2.1.3]{Zh})
\begin{prp}\label{prp:WeilIsSub}
	For any $M\in\Gamma^{(\g)}$ we have 
	\[ \omega_f(M)\circ\iota = \iota\circ\overline{\rho}_{L'/L}^{(\g)}(M). \]
\end{prp}
We want to define theta series and Eisenstein series in the adelic setting. These will be generalizations of the corresponding classical objects.
\begin{defi}
For $\varphi\in\mathscr{S}(V(\A)^\g)$ define the \textit{theta series}
\begin{align*}
    \theta(g,h;\varphi) \coloneqq \sum_{\underline{x}\in V^\g}\omega(g)\varphi(h^{-1}\underline{x}),\qquad g\in \mathcal{G}(\A), \ h\in \mathcal{H}(\A).
\end{align*}
\end{defi}
Then $\theta(g,h;\varphi)$ is automorphic on both $\mathcal{G}$ and $\mathcal{H}$ (i.e.\ invariant under $\mathcal{G}(\Q)\times \mathcal{H}(\Q)$) by Poisson summation.
\begin{defi}
For $\varphi\in\mathscr{S}(V(\A)^\g)$ define the \textit{Siegel Eisenstein series}
\begin{align*}
    E(g,s;\varphi) \coloneqq \sum_{\gamma\in P(\Q)\backslash \mathcal{G}(\Q)}\Phi_\varphi(\gamma g,s),\qquad g\in G(\A), \ s\in \C,
\end{align*}
where 
\begin{align*}
    \mathscr{S}(V(\A)^\g)\rightarrow\mathrm{Ind}_{P(\A)}^{\mathcal{G}(\A)}(\chi_V|\cdot|^s),\qquad\varphi\mapsto\Phi_\varphi(g,s)\coloneqq(\omega(g)\varphi)(0)\cdot|\det u(g)|^{s-s_0}
\end{align*}
is the \textit{standard Siegel--Weil section} and
\begin{align*}
    s_0 \coloneqq k - \frac{\g+1}{2}.
\end{align*}
Here we write $g=na(u)k$ under the Iwasawa decomposition $\mathcal{G}(\A) = N(\A)A(\A)K$ for $K$ the standard maximal open compact subgroup of $\mathcal{G}(\A)$, and the quantity $|\det u(g)| \coloneqq |\det u|$ is well-defined.
\end{defi}
The Eisenstein series converges absolutely for $\Rep(s)>\frac{\g+1}{2}$ and thus defines an automorphic form for $\mathcal{G}(\Q)$. The average of the theta series over the orthogonal group is equal to the Eisenstein series at $s_0$.
\begin{thm}[Siegel--Weil formula (cf.\ \cite{W2} and \cite{KR})]\label{thm:SiegelWeil}
Let $V$ be a positive-definite quadratic space of rank $2k$ over $\Q$ and let $\varphi\in\mathscr{S}(V(\mathbb{A})^\g)$ with $k>\g+1$. Then $E(g,s;\varphi)$ is holomorphic at $s_0$ and
\begin{align*}
    \int_{\mathcal{H}(\Q)\backslash\mathcal{H}(\A)}\theta(g,h;\varphi)\mathrm{d}h = E(g,s_0;\varphi).
\end{align*}
Here the Haar measure $\mathrm{d}h$ is normalized so that $\mathrm{vol}(\mathcal{H}(\Q)\backslash \mathcal{H}(\A)) = 1$.
\end{thm}
Now we can use Theorem \ref{thm:SiegelWeil} to prove that the genus theta series defined in section \ref{subsec:ModularFormsForTheWeilRepresentation} is equal to the Eisenstein series defined in section \ref{sec:VectorValuedEisensteinSeries}.
\begin{thm}\label{thm:SiegelWeilApp}
Let $L$ be a positive-definite even lattice of rank $2k$ for $k\in\Z_{>0}$. Let $G$ be the genus of $L$. Then
\begin{align*}
    \theta_{G}^{(\g)} = E_{k,L'/L}^{(\g)}.
\end{align*}
for $k>\g+1$.
\end{thm}
\begin{proof}
    We will show that for a suitable choice of $\varphi$, the component functions of the vector valued theta series can be recovered from $\theta(g,h;\varphi)$. Averaging over $\mathcal{H}(\Q)\backslash\mathcal{H}(\A)$ then becomes a sum over the genus of $L$. Finally we show that $E(g,s_0;\varphi)$ recovers the component functions of the Eisenstein series and so the result follows from the Siegel--Weil formula. \\
    Let $V \coloneqq L\otimes_\Z\Q$, $\mathcal{D}=L'/L$ and $\underline{\mu} = (\mu_1,\hdots,\mu_\g)\in \mathcal{D}^\g$. Let $\varphi_\infty(\underline{x}) = e^{-\pi\tr(\underline{x},\underline{x})}$ be the standard Gaussian function and set $\varphi = \varphi_\infty\otimes\iota(e^{\underline{\mu}})$.
    For $Z = X+iY\in\mathbb{H}_\g$, we consider $g_Z = n(X)a(U)\in \mathcal{G}(\R)$, where $U\in\mathrm{Pos}_\g(\R)$ is symmetric such that $Y = U^2$. Then $g_Z\cdot iI_\g = Z$ and
    \begin{align*}
        \omega_\infty(g_Z)\varphi_\infty(\underline{x}) = \chi_\infty(\det U)|\det U|_\infty^k\cdot e^{\pi i\tr((\underline{x},\underline{x})Z)}.
    \end{align*}
    We show that
    \begin{align*}
        \chi_\infty(\det U)^{-1}|\det U|_\infty^{-k}\cdot\int_{\mathcal{H}(\Q)\backslash \mathcal{H}(\A)}\theta(g_Z,h;\varphi)\mathrm{d}h
    \end{align*}
    is the $\underline{\mu}$-component of the genus theta series: \\    
    Let $S\subset \mathcal{H}(\A)$ be the stabilizer of $L$. Then we have a bijection
    \begin{align*}
        \mathcal{H}(\Q)\backslash \mathcal{H}(\A)/S\xrightarrow{\sim} G,\qquad h\mapsto h(L\otimes\hat{\Z})\cap V.
    \end{align*}
    Let $\{h_j\}$ be a complete set of representatives of $\mathcal{H}(\Q)\backslash \mathcal{H}(\A)/S$ and let $\{L_j\}$ be the corresponding representatives of $G$ under this bijection. Then
    \begin{align*}
        \int_{\mathcal{H}(\Q)\backslash \mathcal{H}(\A)}\theta(g_Z,h;\varphi)\mathrm{d}h = \sum_j\int_{\mathcal{H}(\Q)\backslash \mathcal{H}(\Q)h_jS}\theta(g_Z,h;\varphi)\mathrm{d}h.
    \end{align*}
    Substituting $h$ for $hh_j$ and applying $\mathcal{H}(\Q)\cap h_jSh_j^{-1} = \Aut(L_j)$, we obtain
    \begin{align*}
    	\int_{\mathcal{H}(\Q)\backslash \mathcal{H}(\Q)h_jS}\theta\left(g_Z,h;\varphi\right)\mathrm{d}h &= \int_{\mathcal{H}(\Q)\backslash \mathcal{H}(\Q)h_jSh_j^{-1}}\theta\left(g_Z,hh_j;\varphi\right)\mathrm{d}h \\
    	&= \frac{1}{\#\Aut(L_j)}\int_{h_jSh_j^{-1}}\theta\left(g_Z,hh_j;\varphi\right)\mathrm{d}h.
    \end{align*}
    Substituting $h$ for $h_jhh_j^{-1}$, the third integral becomes
    \begin{align*}
    	\int_{S}\theta\left(g_Z,h_jh;\varphi\right)\mathrm{d}h 
    	&= \int_S\sum_{\underline{x}\in V^\g}(\omega_\infty(g_Z)\varphi_{\infty}\otimes\iota(e^{\underline{\mu}}))(h^{-1}h_j^{-1}\underline{x})\mathrm{d}h \\
    	&= \frac{\mathrm{vol}(S)}{|\Ort(\mathcal{D})|}\sum_{\sigma\in\Iso(\mathcal{D},L_j'/L_j)}\sum_{\underline{x}\in\sigma\underline{\mu} + L_j^\g}(\omega_\infty(g_Z)\varphi_\infty)(\underline{x}),
    \end{align*}
    where in the last equation we used that the canonical homomorphism $S\rightarrow\Ort(\mathcal{D})$ is surjective (see \cite[Corollary 1.9.6.]{N}). Combining these computations we find
    \begin{align*}
        \chi_\infty(\det U)^{-1}&|\det U|_\infty^{-k}\cdot\int_{\mathcal{H}(\Q)\backslash \mathcal{H}(\A)}\theta(g_Z,h;\varphi)\mathrm{d}h \\
        &= \frac{\mathrm{vol}(S)}{|\Ort(\mathcal{D})|}\sum_j\frac{1}{\#\Aut(L_j)}\sum_{\sigma\in\Iso(\mathcal{D},L_j'/L_j)}\sum_{\underline{x}\in \sigma\underline{\mu} + L_j^\g}e^{\pi i\tr((\underline{x},\underline{x})Z)} \\
        &= \frac{\mathrm{vol}(S)}{|\Ort(\mathcal{D})|}\sum_j\frac{1}{\#\Aut(L_j)}\sum_{\sigma\in\Iso(\mathcal{D},L_j'/L_j)}\langle\theta_{L_j}^{(\g)},e^{\sigma\underline{\mu}}\rangle \\
        &= \frac{\mathrm{vol}(S)}{|\Ort(\mathcal{D})|}\sum_j\frac{1}{\#\Aut(L_j)}\sum_{\sigma\in\Iso(\mathcal{D},L_j'/L_j)}\langle\sigma^*\theta_{L_j}^{(\g)},e^{\underline{\mu}}\rangle.
    \end{align*}
    Finally, note that
    \begin{align*}
        \mathrm{vol}(S)\cdot\mu(G) &= \mathrm{vol}(S)\cdot\sum_j\frac{1}{\#\Aut(L_j)} = \mathrm{vol}(\mathcal{H}(\Q)\backslash \mathcal{H}(\A)) = 1
    \end{align*}
    by our choice of normalization.\\    
    On the other hand we consider $\chi_\infty(\det U)^{-1}|\det U|_\infty^{-k}\cdot E(g_Z,s;\varphi)$ at the value $s_0$: \\
    It is not difficult to prove that $P(\Q)\backslash \mathcal{G}(\Q)\cong\Gamma_\infty^{(\g)}\backslash\Gamma^{(\g)}$. We write
    \begin{align*}
        \Phi_\varphi(g,s) = \Phi_f(g,s)\otimes\Phi_\infty(g,s)
    \end{align*}
    with
    \begin{align*}
        \Phi_f(g,s) = \bigotimes_{p<\infty}\Phi_p(g,s).
    \end{align*}
    One can show that for $M=\left(\begin{smallmatrix} A & B \\ C & D \end{smallmatrix}\right)\in\Gamma^{(\g)}$
    \begin{align*}
        \Phi_\infty(Mg_Z,s) = \det(U)^{k+s-s_0}\det(CZ+D)^{-k}|\det(CZ+D)|_\infty^{-s+s_0}.
    \end{align*}
    For details see for example \cite{Ku}. We consider the finite part. Note that for $v\in\C[\mathcal{D}^\g]$, $\underline{\beta}\in\mathcal{D}^\g$ and $\underline{x}\in\underline{\beta}+ L^\g$ we have $\iota(v)(\underline{x}) = \langle v,e^{\underline{\beta}}\rangle$. Furthermore, $\omega(g_Z)$ acts trivially on $\iota(e^{\underline{\mu}})$ since $g_Z\in \mathcal{G}(\R)$ and therefore
    \begin{align*}
        \chi_\infty(\det U)^{-1}|\det U|_\infty^{-k}\cdot E(g_Z,s_0;\varphi_\infty\otimes\iota(e^{\underline{\mu}})) &= \sum_{M\in\Gamma_\infty^{(\g)}\backslash\Gamma^{(\g)}}\det(CZ+D)^{-k}\omega_f(M)\iota(e^{\underline{\mu}})(0) \\
        &= \sum_{M\in\Gamma_\infty^{(\g)}\backslash\Gamma^{(\g)}}\det(CZ+D)^{-k}\langle\overline{\rho}_{L'/L}^{(\g)}(M)e^{\underline{\mu}},e^{\underline{0}}\rangle \\
        &= \sum_{M\in\Gamma_\infty^{(\g)}\backslash\Gamma^{(\g)}}\det(CZ+D)^{-k}\langle\rho_{L'/L}^{(\g)}(M^{-1})e^{\underline{0}},e^{\underline{\mu}}\rangle \\
        &= \langle E_{k}^{(\g)},e^{\underline{\mu}}\rangle,
    \end{align*}
    where we used Proposition \ref{prp:WeilIsSub} and the fact that $\chi_\infty(x)|x|_\infty^k = x^k$. This proves the theorem.
\end{proof}

\section{The space of theta series}\label{sec:ThetaFunctions}

In this section we will prove the main theorem of this paper. As explained in the introduction, we define a map from the space of cusp forms to the space generated by the theta series by integrating a cuspform against the genus theta series of Siegel genus $2$ evaluated on a diagonal matrix. This process is called the doubling method. By the Siegel--Weil formula we can then substitute the genus theta series for the Eisenstein series. Using the results from sections \ref{sec:VectorValuedHeckeOperators} and \ref{sec:VectorValuedEisensteinSeries} we find that the resulting map is a linear combination of Hecke operators. Finally, we will show, that this map is bijective if the conditions of the main theorem are met.

Let $\mathcal{D}$ be a discriminant form of even signature and level $N$ and let $m$ be even with $m>p\text{-rank}(\mathcal{D})$ for all primes $p$. We set $G = I\!I_{m,0}(\mathcal{D})$ and $k=m/2+\h$ with $\h\geq0$. By \cite[Corollary 1.10.2]{N} $G$ is non-empty. We define a linear map
\begin{align*}
    \Phi \coloneqq \Phi_{\mathcal{D}} \coloneqq \Phi_{\mathcal{D},m,k}:\mathrm{S}_k(\mathcal{D}) \rightarrow\Theta_{m,k}(\mathcal{D})
\end{align*}
by
\begin{align*}
        \Phi_{\mathcal{D},m,k}(f)(z') \coloneqq \int_{\Gamma\backslash\mathbb{H}}\langle f(z),\vartheta_{G,k}(z,-\overline{z'})\rangle y^k\frac{\mathrm{d}x\mathrm{d}y}{y^2}.
\end{align*}
Recall that $\vartheta_{G,k} \coloneqq \partial_\h\theta_G^{(2)}$. Note that for a lattice $L$ we have
\begin{align*}
    \theta_L^{(2)}\left(\matr{z}{0}{0}{z'}\right) = \theta_L^{(1)}(z)\otimes\theta_L^{(1)}(z').
\end{align*}
Because of Proposition \ref{prp:partialthetaislin}, for $\h>0$ also
\[ (\partial_\h\theta_L^{(2)})(z,z') = C\cdot\sum_{i=1}^r\theta_{L,P_i}^{(1)}(z)\otimes\theta_{L,\overline{P_i}}^{(1)}(z'), \]
where $(P_1,\hdots,P_r)$ is an orthonormal basis of $\operatorname{H}_m^\h$ and $C$ is some non-zero constant. We therefore find that
\begin{align}\label{eq:PhiIsFamily}
        \Phi_{\mathcal{D},m,k}(f) = \mu(G)^{-1}|\Ort(\mathcal{D})|^{-1}C\sum_{L\in G}\frac{1}{\#\Aut(L)}\sum_{\sigma\in\mathrm{Iso}(\mathcal{D},L'/L)}\sum_{i=1}^r(f,\sigma^*\theta_{L,P_i})\cdot\sigma^*\theta_{L,P_i}.
\end{align}
If $\D'$ is a discriminant form isomorphic to $\D$, then for a $\tau\in\Iso(\D',\D)$ it follows from $\langle \tau^*v,\tau^*w\rangle = \langle v,w\rangle$ for all $v,w\in\C[\D]$ that
\begin{equation}\label{eq:PhiTauCommute}
	\Phi_{\D',m,k}\circ\tau^* = \tau^*\circ\Phi_{\D,m,k}.
\end{equation}
The following lemma will be useful.
\begin{lem}\label{lem:family}
    Let $V$ be a $\C$-vector space with scalar product $(\cdot,\cdot)$ which is linear in the first variable and $(v_i)_{i=1}^n\subset V$ an arbitrary finite family. Then $f:V\rightarrow V$ defined by
    \begin{align*}
        f(v) \coloneqq \sum_{i=1}^n(v,v_i)\cdot v_i
    \end{align*}
    is surjective onto $\spann(v_i)_{i=1}^n$ and is self-adjoint. In particular $f$ is diagonalizable and
    \begin{align*}
        V = \image(f)\oplus\ker(f).
    \end{align*}
\end{lem}
\begin{proof}
    Let $v,w\in V$. Then
\begin{align*}
    (f(v),w) &= \sum_{i=1}^n(v,v_i)\cdot (v_i,w) \\
    &= (v,\sum_{i=1}^n(w,v_i)\cdot v_i) \\
    &= (v,f(w)).
\end{align*}
Let $A=(a_{ij})$ with $a_{ij} = (v_i,v_j)$ be the Gram matrix of $(v_i)_{i=1}^n$ and define $g:\C^n\twoheadrightarrow \spann(v_i)_{i=1}^n$ by
\begin{align*}
    x = (x_1,\hdots,x_n) \mapsto \sum_{i=1}^nx_i\cdot v_i.
\end{align*}
Then
\begin{align*}
    x\cdot A = ((g(x),v_1),\hdots,(g(x),v_n))
\end{align*}
and $g(x\cdot A) = f(g(x))$. From the first assertion we see that $\ker(g) = \ker(A)$ and since $A$ is self-adjoint, $\C^n = \image(A) \oplus \ker(A) = \image(A) \oplus \ker(g)$. This implies that
\begin{align*}
    \image(f\circ g) = \image(g(\cdot A)) = \image(g) = \spann(v_i)_{i=1}^n.
\end{align*}
\end{proof}
Applying Lemma \ref{lem:family} to $\Phi$ yields
\begin{prp}\label{prp:PhiSurjects}
    Let $\mathcal{D}$, $m$ and $\h$ be as above. If $\h=0$, assume $m>4$. The linear map $\Phi$ is self-adjoint and surjective onto $\Theta_{m,k}(\mathcal{D})_0$. In particular we have
    \begin{align*}
        \mathrm{S}_k(\mathcal{D}) = \Theta_{m,k}(\mathcal{D})_0\oplus\ker(\Phi)
    \end{align*}
    and $\Phi$ is diagonalizable.
\end{prp}
\begin{proof}
We begin with the case $\h=0$. By the Siegel--Weil formula we know that for $m>4$
\begin{align*}
    \mu(G)^{-1}|\Ort(\mathcal{D})|^{-1}\cdot\sum_{L\in G}\frac{1}{\#\Aut(L)}\sum_{\sigma\in\mathrm{Iso}(\mathcal{D},L'/L)}\sigma^*\theta_{L} &=  E_k.
\end{align*}
Furthermore, the forms $\sigma^*\theta_{L}-E_k\in\mathrm{S}_k(\mathcal{D})$ span $\Theta_{m,k}(\mathcal{D})_0$. Since any $f\in\mathrm{S}_k(\mathcal{D})$ is orthogonal on $E_k$, we thus get
\begin{align*}
    \Phi_{\mathcal{D},2k,k}(f) &= \mu(G)^{-1}|\Ort(\mathcal{D})|^{-1}\cdot\sum_{L\in G}\frac{1}{\#\Aut(L)}\sum_{\sigma\in\mathrm{Iso}(\mathcal{D},L'/L)}(f,\sigma^*\theta_{L})\cdot\sigma^*\theta_{L} \\
    &= \mu(G)^{-1}|\Ort(\mathcal{D})|^{-1}\cdot\sum_{L\in G}\frac{1}{\#\Aut(L)}\sum_{\sigma\in\mathrm{Iso}(\mathcal{D},L'/L)}(f,\sigma^*\theta_{L})\cdot\sigma^*\theta_{L} - (f,E_k)E_k \\
    &= \mu(G)^{-1}|\Ort(\mathcal{D})|^{-1}\cdot\sum_{L\in G}\frac{1}{\#\Aut(L)}\sum_{\sigma\in\mathrm{Iso}(\mathcal{D},L'/L)}(f,\sigma^*\theta_{L})\cdot(\sigma^*\theta_{L}-E_k) \\
    &= \mu(G)^{-1}|\Ort(\mathcal{D})|^{-1}\cdot\sum_{L\in G}\frac{1}{\#\Aut(L)}\sum_{\sigma\in\mathrm{Iso}(\mathcal{D},L'/L)}(f,\sigma^*\theta_{L}-E_k)\cdot(\sigma^*\theta_{L}-E_k)
\end{align*}
and we can apply Lemma \ref{lem:family} to the family $(\sigma^*\theta_{L}-E_k)_{(L,\sigma)}$. \\
If $\h>0$, we can immediately apply Lemma \ref{lem:family} because of equation (\ref{eq:PhiIsFamily}).
\end{proof}

It remains to determine when $\Phi_\mathcal{D}$ has trivial kernel. To do this we will need
\begin{lem}\label{lem:KernelsAreTheSame}
    Let $H\subset\mathcal{D}$ be an isotropic subgroup and let $f\in\ker(\Phi_{H^\bot/H})$. Then also $\uparrow_H(f)\in\ker(\Phi_{\mathcal{D}})$.
\end{lem}
\begin{proof}
    If $g\in\image(\Phi_\mathcal{D})$, then it is a linear combination of theta series, i.e.
\begin{align*}
    g = \sum_{L\in G}\sum_{\sigma\in\mathrm{Iso}(L'/L,\mathcal{D})}\sum_{i=1}^{r}c_{L,\sigma,i}\sigma_*\theta_{L,P_i}.
\end{align*}
Then
\begin{align*}
    \downarrow_H(g) = \sum_{L\in G}\sum_{\sigma\in\mathrm{Iso}(L'/L,\mathcal{D})}\sum_{i=1}^{r}c_{L,\sigma,i}\Tilde{\sigma}_*\downarrow_{\sigma^{-1}H}\left(\theta_{L,P_i}\right)\in\Theta_{m,k}(H^\bot/H)_0,
\end{align*}
where $\downarrow_{\sigma^{-1}H}\left(\theta_{L,P_i}\right)$ is again a theta series (cf.\ page \pageref{pg:SublatticeTheta}) and $\Tilde{\sigma}$ is as on page \pageref{pg:sigmaTilde}. Hence, $\downarrow_H(g)\in\Theta_{m,k}(H^\bot/H)_0=\image(\Phi_{H^\bot/H})$. Therefore, we have
\begin{alignat*}{2}
    f\in\ker(\Phi_{H^\bot/H}) &\Leftrightarrow (f,g) = 0 &&\forall g\in\image(\Phi_{H^\bot/H}) \\
    &\Rightarrow (f,\downarrow_H(g)) = 0 &&\forall g\in\image(\Phi_{\mathcal{D}}) \\
    &\Leftrightarrow (\uparrow_H(f),g) = 0 &&\forall g\in\image(\Phi_{\mathcal{D}}) \\
    &\Leftrightarrow \uparrow_H(f)\in\ker(\Phi_{\mathcal{D}}).
\end{alignat*}
\end{proof}
Combining the results of the previous sections we obtain
\begin{thm}
Let $m>6$. The map $\Phi$ is a linear combination of Hecke operators, namely
\begin{align*}
    \Phi_{\mathcal{D},m,k} = C(m,k)\frac{e(-\sign(\mathcal{D})/8)}{\sqrt{|\mathcal{D}|}}\sum_{\n=1}^\infty\frac{T(\n^2)}{\n^{2k-2-\h}},
\end{align*}
where $C(m,k) = G_m^{k-m/2}(1,1)\frac{(k-1)!}{(m/2-1)!}C(k)$ with $C(k)$ as in (\ref{eq:Ck}).
\end{thm}
\begin{proof}
	Let us first consider the case $\h=0$. We have
    \begin{align*}
        \Phi_{\mathcal{D},2k,k}(f)(z') &= \int_{\Gamma\backslash\mathbb{H}}\langle f(z),\theta_G^{(2)}\left(\matr{z}{0}{0}{-\overline{z'}}\right)\rangle y^k\frac{\mathrm{d}x\mathrm{d}y}{y^2} \\
        &= \int_{\Gamma\backslash\mathbb{H}}\langle f(z),E_k^{(2)}\left(\matr{z}{0}{0}{-\overline{z'}}\right)\rangle y^k\frac{\mathrm{d}x\mathrm{d}y}{y^2} \\
        &= \int_{\Gamma\backslash\mathbb{H}}\langle f(z),E_k^{(1)}(z)\rangle y^k\frac{\mathrm{d}x\mathrm{d}y}{y^2}\otimes E_k^{(1)}(z') \\
        &\qquad + \frac{e(-\sign(\mathcal{D})/8)}{\sqrt{|\mathcal{D}|}}\int_{\Gamma\backslash\mathbb{H}}\sum_{\n=1}^\infty\langle f(z),\omega_\n(z,-\overline{z'})\rangle y^k\frac{\mathrm{d}x\mathrm{d}y}{y^2},
    \end{align*}
    where we have used Theorems \ref{thm:SiegelWeilApp} and \ref{thm:EIsHeckeKernel}. The integral defines a scalar product on the finite dimensional vector space $\mathrm{S}_k(\D)$ and for a fixed $z'$ the sum is a convergent series in the space $\mathrm{S}_k(\D)$. Hence, by the continuity of the scalar product we may interchange the order of integration and summation. We use the fact that $E_k^{(1)}\in\mathrm{S}_k(\mathcal{D})^\bot$ and Proposition \ref{prp:HeckeKernels} to obtain
    \begin{align*}
    	\Phi_{\mathcal{D},2k,k}(f)(z') &= \frac{e(-\sign(\mathcal{D})/8)}{\sqrt{|\mathcal{D}|}}\sum_{\n=1}^\infty\int_{\Gamma\backslash\mathbb{H}}\langle f(z),\omega_\n(z,-\overline{z'})\rangle y^k\frac{\mathrm{d}x\mathrm{d}y}{y^2} \\
    	&= C(k)\frac{e(-\sign(\mathcal{D})/8)}{\sqrt{|\mathcal{D}|}}\left(\sum_{\n=1}^\infty\frac{T(\n^2)}{\n^{2k-2}}\right)f.
    \end{align*} 
    If $\h>0$, we first apply $\partial_{\h}$ and use Theorem \ref{thm:diffE} to obtain
    \begin{align*}
        \Phi_{\mathcal{D},m,k}(f)(z') &= \int_{\Gamma\backslash\mathbb{H}}\langle f(z),(\partial_\h\theta_G^{(2)})(z,-\overline{z'})\rangle y^k\frac{\mathrm{d}x\mathrm{d}y}{y^2} \\
        &= \int_{\Gamma\backslash\mathbb{H}}\langle f(z),(\partial_\h E_k^{(2)})(z,-\overline{z'})\rangle y^k\frac{\mathrm{d}x\mathrm{d}y}{y^2} \\
        &= G_m^\h(1,1)\frac{(k-1)!}{(m/2-1)!}\frac{e(-\sign(\mathcal{D})/8)}{\sqrt{|\mathcal{D}|}}\sum_{\n=1}^\infty\int_{\Gamma\backslash\mathbb{H}}\langle f(z),\n^\h\omega_\n(z,-\overline{z'})\rangle y^k\frac{\mathrm{d}x\mathrm{d}y}{y^2} \\
        &= G_m^\h(1,1)\frac{(k-1)!}{(m/2-1)!}C(k)\frac{e(-\sign(\mathcal{D})/8)}{\sqrt{|\mathcal{D}|}} \left(\sum_{\n=1}^\infty \frac{T(\n^2)}{\n^{2k-2-\h}}\right)f.
    \end{align*}
\end{proof}

We remark that in \cite{St} the formula
\begin{align*}
    \int_{\Gamma\backslash\mathbb{H}}\langle f,E_k^{(2)}\left(\matr{z}{0}{0}{-\overline{z'}},\overline{s}\right)\rangle y^k\frac{\mathrm{d}x\mathrm{d}y}{y^2} &= \frac{e(\sign(\mathcal{D})/8)}{\sqrt{|\mathcal{D}|}}C(k,s)\sum_{d\in\Z_{>0}}d^{-k-s}T(d^2)f
\end{align*}
was derived, where 
\begin{align*}
    E_k^{(2)}(Z,s) &= \sum_{M\in\Gamma_\infty^{(2)}\backslash\Gamma^{(2)}}\det(\Imp(Z))^s|_k[M]\rho_\mathcal{D}^{(2)}(M^{-1})e^0\otimes e^0
\end{align*}
and  $C(k,s)$ only depends on the weight $k$ and the variable $s$. However, a factor got lost: In formula (4.8) in \cite{St} a factor of $d^{k-2}$ must be added. Then instead of $d^{-k-s}$ one gets $d^{-2k+2-s}$. Furthermore, the factor $e(\sign(\mathcal{D})/8)$ should be $e(-\sign(\mathcal{D})/8)$, as it was pulled out of the antilinear part of the scalar product. Then Stein's formula coincides with ours.

We further study when $\Phi_\mathcal{D}(f)$ vanishes.

\begin{lem}\label{lem:ZeroForSomeP}
	Let $m>6$. If $\Phi_\D(f)=0$, then 
	\begin{align*}
		\prod_{p\mid N}\left(\sum_{r=0}^\infty\frac{T(p^{2r})}{p^{2rk-2r - \h r}}\right)f = 0.
	\end{align*}
\end{lem}
\begin{proof}
	By the previous theorem and Theorem \ref{thm:BruinierStein} we can write
	\begin{align*}
		\Phi_{\D,m,k}(f) &= C(m,k)\frac{e(-\sign(\mathcal{D})/8)}{\sqrt{|\mathcal{D}|}}\sum_{\n=1}^\infty\frac{T(\n^2)}{\n^{2k-2-\h}}f \\
		&= C(m,k)\frac{e(-\sign(\mathcal{D})/8)}{\sqrt{|\mathcal{D}|}}\left(\sumstack{\n=1 \\ (\n,N) = 1}^\infty\frac{T(\n^2)}{\n^{2k-2-\h}}\right)\prod_{p\mid N}\left(\sum_{r=0}^\infty\frac{T(p^{2r})}{p^{2rk-2r-\h r}}\right)f.
	\end{align*}
	For the primes coprime to $N$ we have seen in Proposition \ref{prp:EulerProduct} that $\mathrm{S}_k(\mathcal{D})$ has a basis consisting of simultaneous Hecke eigenforms and for such an eigenform $f$ and $\Rep(s)>k$ we have
	\[ \sumstack{\n=1 \\ (\n,N) = 1}^\infty\frac{T(\n^2)}{\n^{s}}f = L(f,s)\cdot f \]
	with
	\[ L(f,s) = \prod_{p\nmid N}\frac{(1-\chi_\mathcal{D}(p)p^{k-2-s})(1 +\chi_\mathcal{D}(p)p^{k-1-s})}{1-(\lambda(p^2)+\chi_\mathcal{D}(p)(1-p)p^{k-2})p^{-s}+p^{2k-2-2s}}. \]
	Since $L(f,2k-2-\h)\not=0$ for $m>4$, the operator
	\begin{align*}
		\sumstack{\n=1 \\ (\n,N) = 1}^\infty\frac{T(\n^2)}{\n^{2k-2-\h}}
	\end{align*}
	is bijective. This proves the lemma.
\end{proof}

As explained in the introduction, we now show that if the conditions of the main theorem are satisfied, then for any lattice $L\in G$ there exists a sublattice $M\subset L$ with certain properties that will be useful in the proof of the main theorem.
\begin{lem}\label{lem:SublatticeExists}
	Let $\D$ be a discriminant form of even signature $\sign(\mathcal{D})$ and $m$ a positive integer such that $m = \sign(\D)\bmod 8$ and $m>\text{$p$-rank}(\D)$ for all primes $p$. Then the genus $I\!I_{m,0}(\D)$ is non-empty. Suppose for any $L$ of genus $I\!I_{m,0}(\D)$ the $\Z_p$-lattice $L_p = L\otimes_\Z\Z_p$ splits a hyperbolic plane over $\Z_p$. Then for any $L$ of genus $I\!I_{m,0}(\mathcal{D})$, there exists a sublattice $M\subset L$ such that 
\begin{align*}
        M'/M \cong \mathcal{D}\oplus \langle\gamma,\mu\rangle,
    \end{align*}
    where $n\gamma=n\mu = 0$, $\q(\gamma)=\q(\mu) = 0\bmod 1$ and $(\gamma,\mu) = 1/n\bmod 1$. 
\end{lem}
\begin{proof}
    By assumption we can write
    \[ L_p  = \Tilde{L_p}\oplus U_p, \]
    where $U_p$ is the lattice $\alpha_1\Z_p+\alpha_2\Z_p$ with Gram matrix $\left(\begin{smallmatrix}0&1\\1&0\end{smallmatrix}\right)$. We define
    \[ M_p \coloneqq \begin{cases}
    	\Tilde{L_p}\oplus U_p(p) & \text{if $p\mid N$} \\
    	L_p & \text{if $p\nmid N$}
    \end{cases}\]
    and $M \coloneqq \bigcap_{p<\infty}(M_p\cap (L\otimes_\Z\Q))$, where $U_p(p)$ is the lattice $p\alpha_1\Z_p+\alpha_2\Z_p$. By \cite[Satz 21.5]{Kn} we have $M\otimes_\Z\Z_p = M_p$ for all $p$ and for $p\mid N$ we have
    \[ M_p'/M_p = \Tilde{L_p}'/\Tilde{L_p} \oplus U_p(p)'/U_p(p) \cong L_p'/L_p \oplus \langle\gamma_p,\mu_p\rangle \]
    with $p\gamma_p=p\mu_p = 0$, $\q(\gamma_p)=\q(\mu_p) = 0\bmod 1$ and $(\gamma_p,\mu_p) = 1/p\bmod 1$. This proves the result.
\end{proof}

Finally, we can prove the main theorem.
\begin{thm}\label{thm:MainTheorem}
	\mainTheorem
\end{thm}
\begin{proof}%[Proof of Theorem \ref{thm:MainTheorem}]
    By Proposition \ref{prp:PhiSurjects} it suffices to show that $\Phi_\mathcal{D}$ is injective, so assume that $\Phi_\mathcal{D}(f) = 0$ for some $f\in\mathrm{S}_k(\mathcal{D})$. Let $n=\prod_{p\mid N}p$ be the radical of $N$. The genus $I\!I_{m,0}(\mathcal{D})$ is non-empty (see \cite[Corollary 1.10.2]{N}). We choose some $L\in I\!I_{m,0}(\mathcal{D})$. By Lemma \ref{lem:SublatticeExists} there exists a lattice $M\subset L$ such that
    \begin{align*}
        \widetilde{\D}\coloneqq M'/M \cong \mathcal{D}\oplus \langle\gamma,\mu\rangle,
    \end{align*}
    where $n\gamma=n\mu = 0$, $\q(\gamma)=\q(\mu) = 0\bmod 1$ and $(\gamma,\mu) = 1/n\bmod 1$. We set $H\coloneqq\langle\gamma\rangle$ so that $\mathcal{D}\cong H^\bot/H$. Because of (\ref{eq:PhiTauCommute}), we can assume that $\D=H^\bot/H$. According to Lemma \ref{lem:KernelsAreTheSame} also $\Phi_{\widetilde{\mathcal{D}}}(\uparrow_H(f)) = 0$. Then
    \begin{align*}
        \prod_{p\mid N}\left(\sum_{r=0}^\infty\frac{T(p^{2r})}{p^{2rk-2r-\h r}}\right)\uparrow_H(f) = 0
    \end{align*}
    by Lemma \ref{lem:ZeroForSomeP}. For any $\beta\in \mathcal{D}$ clearly $\gamma+\beta$ and $\mu+\beta$ are not in $\widetilde{\mathcal{D}}^{p}$ for $p\mid N$. If $2\mid N$, then $(n/2)\gamma,(n/2)\mu\in\mathcal{D}_2$ and
    \begin{align*}
    	2\q((n/2)\mu) + ((n/2)\mu,\gamma+\beta) &= 2\q((n/2)\gamma) + ((n/2)\gamma,\mu+\beta) \\
    	&= (n/2)(\mu,\gamma) = 1/2\mod1,
    \end{align*}
    so that $\gamma+\beta\not\in\widetilde{\mathcal{D}}^{2*}$ and $\mu+\beta\not\in\widetilde{\mathcal{D}}^{2*}$. Since $n\cdot(\gamma+\beta) = n\beta = n\cdot(\mu+\beta)$ and $\q(\gamma+\beta) = \q(\beta) = \q(\mu+\beta)$, we can apply Corollary \ref{cor:Symmetry} with set of primes $P = \{p\text{ prime}\mid p\mid N\}$. Then the element $v\coloneqq v_{\gamma+\beta,\mu+\beta,P}$ from Corollary \ref{cor:Symmetry} is given by
    \[ v = \sum_{S\subset P}(-1)^{|S|}e^{\gamma_S^\mu + \beta}. \]
    Note that $(\gamma_S^\mu + \beta,\gamma) = \prod_{p\in S}\frac{1}{p}\bmod1$, so that $\downarrow_H(v) = e^\beta$ by the definition of $\downarrow_H$ (see page \pageref{eq:IsotropicDescent}). We obtain
    \[ \langle f,e^\beta\rangle = \langle \uparrow_H(f),v\rangle = \langle\prod_{p\mid N}\left(\sum_{r=0}^\infty\frac{T(p^{2r})}{p^{2rk-2r-\h r}}\right)\uparrow_H(f),v\rangle = 0. \]
    Because $\beta$ was chosen arbitrarily, it follows that $f=0$.
\end{proof}
We remark that a $p$-adic lattice $L_p$ of rank $m$ splits a hyperbolic plane over $\Z_p$ if and only if $\text{$p$-rank}(\mathcal{D}) < m-2$ or $\text{$p$-rank}(\mathcal{D}) = m-2$ and $\prod_{q}\epsilon_q = \leg{-a}{p}$, where the $p$-adic component of $\D$ is equal to
\[ \bigoplus_qq^{\epsilon_qn_q} \]
in the notation of Conway and Sloane (cf.\ \cite{CS}, chapter 15) and $|\mathcal{D}| = p^\alpha a$ with $(a,p) = 1$. This is a property of the genus of $L$, rather than of $L$ itself.

We extend the result of the main theorem to all discriminant forms of even signature and lattices of rank at least $10$ by considering the space
\begin{align*}
	\Theta_{m,k}^\uparrow(\mathcal{D}) \coloneqq &\spann\{\uparrow_H^\mathcal{D}(\sigma^*\theta_{L,P})\mid L\in I\!I_{m,0}(H^\bot/H) \text{ for some isotropic subgroup } H\subset\mathcal{D}, \\
	&\qquad P \in\operatorname{H}_m^{k-m/2}, \sigma\in\Iso(H^\bot/H,L'/L)\}.
\end{align*}
Then we obtain
\begin{cor}\label{cor:MainTheorem}
	\mainTheoremCorollary
\end{cor}
\begin{proof}
	Let $H\subset \mathcal{D}$ be any isotropic subgroup such that all $p$-ranks of $H^\bot/H$ are less than or equal to $6$. Then any lattice in $I\!I_{m,0}(H^\bot/H)$ locally splits a hyperbolic plane (cf.\ \cite[Corollary 1.9.3]{N}). By Theorem \ref{thm:MainTheorem} we know that $\mathrm{S}_k(H^\bot/H)\subset\Theta_{m,k}(H^\bot/H)$. It was shown in \cite{M2} that
	\begin{multline*}
		\mathrm{M}_k(\mathcal{D}) = \spann\{\uparrow_H(f) \mid f\in\mathrm{M}_k(H^\bot/H), \\
		H\subset \mathcal{D}\text{ isotropic subgroup such that $p$-rank$(H^\bot/H)\leq6$ for all primes $p$}\},
	\end{multline*}
	which proves the corollary.
\end{proof}

%We want to make a remark about the sharpness of the bound given in Theorem \ref{thm:MainTheorem}. Indeed when the $\text{$p$-rank}$ of $L'/L$ is $m-2$ and the condition for $\text{$p$-sign}$ is not satisfied or $\text{$p$-rank}(L'/L)>m-2$, then we can usually find some cusp form $F\in\mathrm{S}_k(L'/L)$ that is not in $\Theta_{m,k}(L'/L)$, because of the following argument: \\
%For an even lattice $L$ of level $N$ the lattice $N\cdot L'$ is also even. And if two lattices $L$ and $M$ are in the same genus, then their rescaled dual lattices are also in the same genus. Therefore, the lattices in $G(L)$ and the lattices ind $G(N\cdot L')$ are in a one-to-one correspondence with each other. This 

\section{Applications}

In this section we describe two applications of our main result.

\subsection*{Waldspurger's result on the scalar valued basis problem}

Waldspurger's result on the scalar valued basis problem can be derived from Theorem \ref{thm:MainTheorem} by showing that any newform of level $N$ is the $0$-component of a suitable vector valued cusp form. For a positive integer $N$ let $\mathrm{S}_k^{\text{new}}(N)$ denote the space of scalar valued newforms of level $N$. %For a square-free integer $D>1$ with $D=1\bmod 4$ let $\chi_D(n) = \leg{n}{D}$ and denote by $\mathrm{S}_k(D,\chi_D)$ the space of cuspforms of level $D$ and character $\chi_D$. This space decomposes into the direct sum
%\[ \mathrm{S}_k(D,\chi_D) = \mathrm{S}_k^{+1}(D,\chi_D)\oplus\mathrm{S}_k^{-1}(D,\chi_D), \]
%where for $\epsilon=\pm1$
%\[ \mathrm{S}_k^\epsilon(D,\chi_D) = \{f\in\mathrm{S}_k(D,\chi_D)\mid a_n(f) = 0\text{ if }\chi_D(n) = -\epsilon \}, \]
%where $a_n(f)$ denotes the $n$'th Fourier coefficient of $f$.
Following \cite{Wal} we let $\Theta(m,k,N,D)$ denote the space generated by scalar valued theta series of positive-definite even lattices of rank $m$, level $N$ and discriminant $D$ weighted with harmonic polynomials of degree $k-m/2$, i.e.\ the $0$-components of elements in $\Theta_{m,k}(\D)$, where $\D$ has level $N$ and $|\D| = D$.

\begin{cor}\label{cor:Waldspurger}
	Let $m,N$ be positive integers.
	\begin{enumerate}[(i)]
		\item For $m=0\bmod 8$ and $m\geq8$ we have $\mathrm{S}_k^{\text{new}}(N) \subset \Theta(m,k,N,N^2)$	(cf.\ \cite[Theorem 1]{Wal}).
		\item For $m=4\bmod 8$ and $m\geq12$ and for any prime $q\mid N$ we have $\mathrm{S}_k^{\text{new}}(N) \subset \Theta(m,k,N,N^2q^2)$ (cf.\ \cite[Theorem 2]{Wal}).
		%\item For $m$ even and any square-free integer $D>1$ with $D=1\bmod 4$ and $\epsilon = (-1)^{m/2}$ one has $\mathrm{S}_k^\epsilon(D,\chi_D) = \Theta(m,k,D,D^{m-1})$	(cf.\ \cite[Theorem 3 (i)]{Wal}).
	\end{enumerate}
\end{cor} 
\begin{proof}
	By Theorem \ref{thm:MainTheorem} it suffices to show that any $f\in\mathrm{S}_k^{\text{new}}(N)$ is the $0$-component of some vector valued cusp form for an appropriate Weil representation. For a discriminant form $\D$ of level $N$ we define a lift $\mathcal{L}_\D:\mathrm{S}_k^{\text{new}}(N)\to\mathrm{S}_k(\D)$ by
	\[ \mathcal{L}_\D(f) \coloneqq \sum_{M\in\Gamma_0(N)\backslash\Gamma}f|_k[M]\rho_{\D}(M^{-1})e^0 \]
	and consider the map $\Psi_\D:\mathrm{S}_k^{\text{new}}(N)\to\mathrm{S}_k(N)$ defined by $\Psi_\D(f) = \langle \mathcal{L}_\D(f),e^0\rangle$. In \cite[Theorem 1.1]{SV} it was shown that for certain discriminant forms, $\Psi_\D$ is a non-zero multiple of the identity map, say $c_\D\cdot\mathrm{id}$. \\
	For part (i) let $m=0\bmod8$ with $m\geq8$ and $\D = (\Z/N\Z)^2$ with quadratic form given by
	\[ (a,b)\mapsto \frac{ab}{N}\mod1. \]
	Then $I\!I_{m,0}(\D)$ is non-empty and satisfies the conditions of Theorem \ref{thm:MainTheorem} as well as those of \cite[Theorem 1.1]{SV}, so that for any $f\in\mathrm{S}_k^{\text{new}}(N)$ we find $\mathcal{L}_\D(f)\in\Theta_{m,k}(\D)$ and hence, $f = \Psi_\D(c_\D^{-1} f)\in\Theta(m,k,N,N^2)$. \\
	For part (ii) let $m=4\bmod8$ with $m\geq12$ and $q\mid N$ a prime and consider the discriminant form $\D = (\Z/N\Z)^2\oplus(\Z/q\Z)^2$ with quadratic form given by
	\[ (a,b)+(c,d)\mapsto \frac{ab}{N} + \frac{c^2-ud^2}{q}\mod1, \]
	where $u$ is a non-square modulo $q$. Then $I\!I_{m,0}(\D)$ is non-empty and again, the conditions of Theorem \ref{thm:MainTheorem} and those of \cite[Theorem 1.1]{SV} are satisfied, which implies $\mathrm{S}_k^{\text{new}}(N) \subset \Theta(m,k,N,N^2q^2)$.
\end{proof}
Theorem 3 in \cite{Wal} on cuspforms with character can probably also be proved using Theorem \ref{thm:MainTheorem}, but we have not checked all the details.

\subsection*{Local Borcherds products}

Now we want to apply the main result of the present work to show that the space of local obstructions for constructing Borcherds products generates the space of global obstructions. This extends Theorem 5.4 in \cite{BF}. For details on Borcherds products see for example \cite{Bo1}, \cite{Bo2}, \cite{Br} and \cite{BF}.

Let $L$ be an even lattice of signature $(l,2)$ with bilinear form $(\cdot,\cdot)$ and quadratic form $\q(x) = \frac{1}{2}(x,x)$. Note that in \cite{BF} lattices of signature $(2,l)$ were considered. Throughout we will assume that $l>8$ is even and that $L$ splits two hyperbolic planes $I\!I_{1,1}\oplus I\!I_{1,1}$. The complex manifold
\[ \mathcal{K}=\{[Z_L]\in \mathbb{P}(L\otimes_\Z\C) \mid (Z_L,Z_L)=0, (Z_L,\overline{Z_L})<0\} \]
has two connected components, which are exchanged by complex conjugation of $Z_L$. We choose one of the connected components and denote it by $\mathcal{H}_l$. The group $\Ort(L\otimes_\Z\R)$ acts on $\mathcal{K}$ and its index-$2$ subgroup $\Ort(L\otimes_\Z\R)^+$ preserving $\mathcal{H}_l$ consists of the elements of positive spinor norm. Let $\Gamma$ be a subgroup of finite index of the orthogonal group of $L$, denote $\Gamma^+ = \Gamma\cap\Ort(L\otimes_\Z\R)^+$ and let $X_\Gamma$ be the Baily-Borel compactification of $\Gamma^+\backslash\mathcal{H}_l$. The boundary of this compactification is a curve with usually many irreducible components, which are determined by the isotropic subspaces of $L\otimes_\Z\Q$. 

A divisor $D$ on $X_\Gamma$ is a formal linear combination $D = \sum n_Y Y$
$(n_Y \in\Z)$ of irreducible closed analytic subsets $Y$ of codimension $1$ such that the support
$\bigcup_{n_Y\not=0} Y$ is a closed analytic subset of everywhere pure codimension $1$. For any vector $\lambda\in L'$ of positive norm the orthogonal complement of $\lambda$ in $\mathcal{H}_l$ defines
a divisor $\lambda^\bot$ on $\mathcal{H}_l$. Let $\beta\in L'/L$ and $n \in\Z + \q(\beta)$ with $n > 0$. Then
\[ H(\beta,n) = \sumstack{\lambda\in \beta+L \\ \q(\lambda) = n}\lambda^\bot \]
is a $\Gamma$-invariant divisor on $\mathcal{H}_l$. It determines a divisor on $\Gamma^+\backslash\mathcal{H}_l$ and by taking the closure also on $X_\Gamma$. Following Borcherds we call this divisor \textit{Heegner divisor} of discriminant $(\beta, n)$. Note that $H(\beta,n) = H(-\beta,n)$. A Heegner divisor defines an element in the Picard group $\mathrm{Pic}(X_\Gamma)$ of $X_\Gamma$, which is the group of divisors modulo linear equivalence. Two divisors are called linearly equivalent if their difference is principal, i.e.\ the divisor of a meromorphic function on $X_\Gamma$.

Let $F\subset L\otimes_\Z\Q$ be any 2-dimensional isotropic subspace and let $\Tilde{F}\subset L\otimes_\Z\Q$ be a complementary subspace such that $F+\Tilde{F}$ is the sum of two hyperbolic planes.
We set $M = L\cap F^\bot\cap\Tilde{F}^\bot$. Then $M$ is positive-definite of rank $l-2>6$. Note that $\mathcal{D} \coloneqq L'/L \cong M'/M$. There exists a natural isomorphism $\pi:\mathcal{D}\rightarrow M'/M$ induced from the projection from $L'$ to $M'$. Let $s$ be a generic boundary point, i.e.\ a point on the $1$-dimensional boundary component corresponding to $F$ that does not correspond to a $0$-dimensional boundary component. The local divisor class group of $X_\Gamma$ in $s$ is the local Picard group
\[ \mathrm{Pic}(X_\Gamma,s) = \lim_{\longrightarrow}\mathrm{Pic}(U_\mathrm{reg}), \]
where $U$ ranges through all open neighbourhoods of $s$ and $U_\mathrm{reg} = U\cap(\Gamma\backslash\mathcal{H})$. For $\Gamma_\infty = P\cap\Gamma$, where $P\subset\Ort(L\otimes_\Z\R)$ is the parabolic subgroup stabilizing $F\otimes\R$, we can consider the pullback of divisors on $\Gamma\backslash\mathcal{H}$ to divisors on $\Gamma_\infty\backslash U_\epsilon$ for certain neighbourhoods $U_\epsilon$ (cf.\ \cite[Section 4]{BF}). The pullback of $H(\beta,l)$ is denoted by $H_F(\beta,l)$ and will be called a \textit{local Heegner divisor}. One can show that
\[ H_F(\beta,l) = \sumstack{\lambda\in (\beta+L)\cap F^\bot \\ \q(\lambda) = l}\lambda^\bot \]
and $H_F(\beta,l)$ defines an element of $\mathrm{Pic}(X_\Gamma,s)$ (for details see \cite{BF}). The following proposition was shown in \cite{BF}.
\begin{prp}[Proposition 5.1, \cite{BF}]\label{prp:BruinierHeegner}
	A finite linear combination of divisors
	\[ \frac{1}{2}\sum_{\beta\in \mathcal{D}}\sumstack{n\in\Z+\q(\beta) \\ n > 0}c(\beta,n)H_F(\beta,n) \]
	(with $c(\beta,n)\in\Z$ and $c(\beta,n) = c(-\beta,n)$) is a torsion element of $\mathrm{Pic}(X_\Gamma,s)$ if and only if
	\[ \sum_{\beta\in \mathcal{D}}\sumstack{n\in\Z+\q(\beta) \\ n > 0}c(\beta,n)a(\pi(\beta),n) = 0 \]
	for all theta series $\theta_{M,P}$ with spherical polynomial $P\in\operatorname{H}_{l-2}^2$ and Fourier coefficients $a(\beta,n)$ ($\beta\in M'/M$ and $n\in\Z+\q(\beta)$).
\end{prp}
Let from now on $\Gamma = \ker(\Ort(L)\to\Ort(L'/L))$ be the discriminant kernel of $L$. In \cite{Bo1} Borcherds constructed lifts of certain vector valued modular forms of weight $1-l/2$ for $\Mp$ to meromorphic modular forms for the group $\Gamma^+$ (cf.\ Theorem 13.3 \cite{Bo1}). These lifts have product expansions at $0$-dimensional cusps and are therefore called automorphic products or Borcherds products. Their divisors are linear combinations of Heegner divisors. The following characterization can be found in \cite[Theorem 5.2]{BF} and goes back to Borcherds (cf.\ Theorem 3.1 \cite{Bo2}).

\begin{thm}[Borcherds]\label{thm:BorcherdsHeegner}
	A finite linear combination of Heegner divisors
	\[ \frac{1}{2}\sum_{\beta\in \mathcal{D}}\sumstack{n\in\Z+\q(\beta) \\ n > 0}c(\beta,n)H(\beta,n) \]
	(with $c(\beta,n)\in\Z$ and $c(\beta,n) = c(-\beta,n)$) is the divisor of a Borcherds product for the group $\Gamma^+$ (as in \cite[Theorem 13.3]{Bo1}) if and only if for any cusp form $f\in\mathrm{S}_{1+l/2}(\mathcal{D})$ with Fourier coefficients $a(\beta,n)$ the equality
	\[ \sum_{\beta\in \mathcal{D}}\sumstack{n\in\Z+\q(\beta) \\ n > 0}c(\beta,n)a(\beta,n) = 0 \]
	holds.
\end{thm}

According to Theorem \ref{thm:BorcherdsHeegner} the space $\mathrm{S}_{1+l/2}(\mathcal{D})$ carries some information on the subgroup of $\mathrm{Pic}(X_\Gamma)$ generated by the divisors of Borcherds products, while the space generated by the theta series $\theta_{M,P}$ carries information on the local Picard group $\mathrm{Pic}(X_\Gamma,s)$. So when the theta series span the space of cusp forms, then we can infer that a linear combination of Heegner divisors is the divisor of a Borcherds product if and only if it is locally trivial. 

\begin{defi}
	A divisor $H$ on $X_\Gamma$ is called trivial at generic boundary points if for every one-dimensional irreducible component $B$ of the boundary of $X_\Gamma$ there exists a generic point $s\in B$ such that $H$ is a torsion element of $\mathrm{Pic}(X_\Gamma,s)$.
\end{defi}

The following result was suggested by Jan Bruinier and generalizes \cite[Theorem 5.4]{BF} to non-unimodular lattices.

\begin{thm}\label{thm:BorcherdsProductLocallyTrivial}
	Let $L$ be an even lattice of signature $(l,2)$ with even $l>8$ splitting two hyperbolic planes $I\!I_{1,1}\oplus I\!I_{1,1}$. Assume that the discriminant form $\mathcal{D} = L'/L$ satisfies the conditions of Theorem \ref{thm:MainTheorem} for $m=l-2$. Let
	\[ H = \frac{1}{2}\sum_{\beta\in \mathcal{D}}\sumstack{n\in\Z+\q(\beta) \\ n > 0}c(\beta,n)H(\beta,n) \]
	be a finite linear combination of Heegner divisors $H(\beta,n)$ (with coefficients $c(\beta,n)\in\Z$). Then the following statements are equivalent:
	\begin{enumerate}[i)]
		\item $H$ is the divisor of a Borcherds product for the group $\Gamma^+$ as in \cite[Theorem 13.3]{Bo1}.
		\item $H$ is the divisor of a meromorphic automorphic form for $\Gamma^+$.
		\item $H$ is trivial at generic boundary points.
	\end{enumerate}
\end{thm}
\begin{proof}
	The modularity of a meromorphic modular form $\psi$ for the group $\Gamma^+$ immediately implies that the divisor $(\psi)$ attached to $\psi$ is trivial at generic boundary points. Therefore, we only need to prove (iii) implies (i). So assume that $H$ is trivial at generic boundary points. Note that, since $L$ splits two hyperbolic planes, the genus of $L$ contains only one class and the natural projection $\Ort(L)\rightarrow\Ort(L'/L)$ is surjective (cf.\ \cite[Theorem 1.14.2]{N}). So by proposition \ref{prp:BruinierHeegner}
	\[ \sum_{\beta\in \mathcal{D}}\sumstack{n\in\Z+\q(\beta) \\ n > 0}c(\beta,n)a(\beta,n) = 0 \]
	for any cusp form $f\in\Theta_{m,1+l/2}(\mathcal{D})$ with Fourier coefficients $a(\beta,n)$ ($\beta\in\mathcal{D}$ and $n\in\Z+\q(\beta)$). Now according to Theorem \ref{thm:MainTheorem} $\Theta_{m,1+l/2}(\mathcal{D})=\mathrm{S}_{1+l/2}(\mathcal{D})$ and so Theorem \ref{thm:BorcherdsHeegner} implies that $H$ is the divisor of a Borcherds product.
\end{proof}

%Note that the condition on the genus of $L$ containing only one class is not a strong condition as this is usually the case for indefinite lattices. Indeed if the genus of an indefinite lattice $L$ contains more than one class, then we know that $|\det(L)|\geq128$ and $k^{\binom{n}{2}}$ divides $4^{[\frac{n}{2}]}\det(L)$ for some non-square $k\equiv0\text{ or }1\bmod4$ (cf.\ \cite[chapter 15, \S 9]{CS} and \cite[chapter 7, \S 3]{Wat}).

As a corollary we find for lattices $L$ that satisfy the conditions of Theorem \ref{thm:BorcherdsProductLocallyTrivial} that any meromorphic modular form for the discriminant kernel of $L$, whose divisor is a linear combination of Heegner divisors, is a Borcherds product. This was already proved in greater generality in \cite{Br}, however, using an entirely different argument, which says nothing about the local Picard groups.

\bibliographystyle{abbrv}
\bibliography{references}

\begin{thebibliography}{10}

\bibitem{AGM}
D.~Allcock, I.~Gal, and A.~Mark.
\newblock The {C}onway-{S}loane calculus for 2-adic lattices.
\newblock {\em Enseign. Math.}, 66(1-2):5--31, 2020.

\bibitem{Boe}
S.~B\"ocherer.
\newblock \"uber die {F}unktionalgleichung automorpher {$L$}-{F}unktionen zur
  {S}iegelschen {M}odulgruppe.
\newblock {\em J. Reine Angew. Math.}, 362:146--168, 1985.

\bibitem{BKS}
S.~B\"ocherer, H.~Katsurada, and R.~Schulze-Pillot.
\newblock On the basis problem for {S}iegel modular forms with level.
\newblock In {\em Modular forms on {S}chiermonnikoog}, pages 13--28. Cambridge
  Univ. Press, Cambridge, 2008.

\bibitem{Bo1}
R.~E. Borcherds.
\newblock Automorphic forms with singularities on {G}rassmannians.
\newblock {\em Invent. Math.}, 132(3):491--562, 1998.

\bibitem{Bo2}
R.~E. Borcherds.
\newblock The {G}ross-{K}ohnen-{Z}agier theorem in higher dimensions.
\newblock {\em Duke Math. J.}, 97(2):219--233, 1999.

\bibitem{Bo3}
R.~E. Borcherds.
\newblock Reflection groups of {L}orentzian lattices.
\newblock {\em Duke Math. J.}, 104(2):319--366, 2000.

\bibitem{Br}
J.~H. Bruinier.
\newblock {\em Borcherds products on {O}(2, {$l$}) and {C}hern classes of
  {H}eegner divisors}, volume 1780 of {\em Lecture Notes in Mathematics}.
\newblock Springer-Verlag, Berlin, 2002.

\bibitem{BF}
J.~H. Bruinier and E.~Freitag.
\newblock Local {B}orcherds products.
\newblock {\em Ann. Inst. Fourier (Grenoble)}, 51(1):1--26, 2001.

\bibitem{BS}
J.~H. Bruinier and O.~Stein.
\newblock The {W}eil representation and {H}ecke operators for vector valued
  modular forms.
\newblock {\em Math. Z.}, 264(2):249--270, 2010.

\bibitem{CS}
J.~H. Conway and N.~J.~A. Sloane.
\newblock {\em Sphere packings, lattices and groups}, volume 290 of {\em
  Grundlehren der mathematischen Wissenschaften [Fundamental Principles of
  Mathematical Sciences]}.
\newblock Springer-Verlag, New York, third edition, 1999.
\newblock With additional contributions by E. Bannai, R. E. Borcherds, J.
  Leech, S. P. Norton, A. M. Odlyzko, R. A. Parker, L. Queen and B. B. Venkov.

\bibitem{E1}
M.~Eichler.
\newblock Quadratische {F}ormen und {M}odulfunktionen.
\newblock {\em Acta Arith.}, 4:217--239, 1958.

\bibitem{E2}
M.~Eichler.
\newblock The basis problem for modular forms and the traces of the {H}ecke
  operators.
\newblock In {\em Modular functions of one variable, {I} ({P}roc. {I}nternat.
  {S}ummer {S}chool, {U}niv. {A}ntwerp, {A}ntwerp, 1972)}, volume Vol. 320 of
  {\em Lecture Notes in Math.}, pages 75--151. Springer, Berlin-New York, 1973.

\bibitem{EZ}
M.~Eichler and D.~Zagier.
\newblock {\em The theory of {J}acobi forms}, volume~55 of {\em Progress in
  Mathematics}.
\newblock Birkh\"auser Boston, Inc., Boston, MA, 1985.

\bibitem{F2}
E.~Freitag.
\newblock {\em Siegelsche {M}odulfunktionen}, volume 254 of {\em Grundlehren
  der mathematischen Wissenschaften [Fundamental Principles of Mathematical
  Sciences]}.
\newblock Springer-Verlag, Berlin, 1983.

\bibitem{Ga}
P.~B. Garrett.
\newblock Pullbacks of {E}isenstein series; applications.
\newblock In {\em Automorphic forms of several variables ({K}atata, 1983)},
  volume~46 of {\em Progr. Math.}, pages 114--137. Birkh\"auser Boston, Boston,
  MA, 1984.

\bibitem{I}
T.~Ibukiyama.
\newblock On differential operators on automorphic forms and invariant
  pluri-harmonic polynomials.
\newblock {\em Comment. Math. Univ. St. Paul.}, 48(1):103--118, 1999.

\bibitem{KV}
M.~Kashiwara and M.~Vergne.
\newblock On the {S}egal-{S}hale-{W}eil representations and harmonic
  polynomials.
\newblock {\em Invent. Math.}, 44(1):1--47, 1978.

\bibitem{Kn}
M.~Kneser.
\newblock {\em Quadratische {F}ormen}.
\newblock Springer-Verlag, Berlin, 2002.
\newblock Revised and edited in collaboration with Rudolf Scharlau.

\bibitem{Ku}
S.~S. Kudla.
\newblock Some extensions of the {S}iegel-{W}eil formula.
\newblock In {\em Eisenstein series and applications}, volume 258 of {\em
  Progr. Math.}, pages 205--237. Birkh\"auser Boston, Boston, MA, 2008.

\bibitem{KR}
S.~S. Kudla and S.~Rallis.
\newblock On the {W}eil-{S}iegel formula.
\newblock {\em J. Reine Angew. Math.}, 387:1--68, 1988.

\bibitem{Li}
C.~Li.
\newblock From sum of two squares to arithmetic {S}iegel-{W}eil formulas.
\newblock {\em Bull. Amer. Math. Soc. (N.S.)}, 60(3):327--370, 2023.

\bibitem{Me}
I.~Metzler.
\newblock PhD thesis, Technische {U}niversit\"at {D}armstadt.
\newblock In preperation.

\bibitem{Mi}
T.~Miyake.
\newblock {\em Modular forms}.
\newblock Springer Monographs in Mathematics. Springer-Verlag, Berlin, english
  edition, 2006.
\newblock Translated from the 1976 Japanese original by Yoshitaka Maeda.

\bibitem{M2}
M.~K.-H. M\"uller.
\newblock Modular forms for the {W}eil representation induced from isotropic
  subgroups.
\newblock {\em J. Number Theory}, 263:206--233, 2024.

\bibitem{N}
V.~V. Nikulin.
\newblock Integer symmetric bilinear forms and some of their geometric
  applications.
\newblock {\em Math. USSR Izv.}, 14(1):103--167, 1980.

\bibitem{R}
R.~A. Rankin.
\newblock Contributions to the theory of {R}amanujan's function {$\tau(n)$} and
  similar arithmetical functions. {II}. {T}he order of the {F}ourier
  coefficients of integral modular forms.
\newblock {\em Proc. Cambridge Philos. Soc.}, 35:351--372, 1939.

\bibitem{S2}
N.~R. Scheithauer.
\newblock The {W}eil representation of {${\rm SL}_2(\mathbb Z)$} and some
  applications.
\newblock {\em Int. Math. Res. Not. IMRN}, (8):1488--1545, 2009.

\bibitem{S3}
N.~R. Scheithauer.
\newblock Some constructions of modular forms for the {W}eil representation of
  {${\rm SL}_2(\mathbb{Z})$}.
\newblock {\em Nagoya Math. J.}, 220:1--43, 2015.

\bibitem{S4}
N.~R. Scheithauer.
\newblock Automorphic products of singular weight.
\newblock {\em Compos. Math.}, 153(9):1855--1892, 2017.

\bibitem{SV}
M.~Schwagenscheidt and F.~V\"olz.
\newblock Lifting newforms to vector-valued modular forms for the {W}eil
  representation.
\newblock {\em Int. J. Number Theory}, 11(7):2199--2219, 2015.

\bibitem{Sk2}
N.-P. Skoruppa.
\newblock Jacobi forms of critical weight and {W}eil representations.
\newblock In {\em Modular forms on {S}chiermonnikoog}, pages 239--266.
  Cambridge Univ. Press, Cambridge, 2008.

\bibitem{St}
O.~Stein.
\newblock On analytic properties of the standard zeta function attached to a
  vector-valued modular form.
\newblock {\em Res. Number Theory}, 8(4):Paper No. 86, 27, 2022.

\bibitem{Wal}
J.-L. Waldspurger.
\newblock Engendrement par des s\'eries th\^eta de certains espaces de formes
  modulaires.
\newblock {\em Invent. Math.}, 50(2):135--168, 1978/79.

\bibitem{W1}
A.~Weil.
\newblock Sur certains groupes d'op\'erateurs unitaires.
\newblock {\em Acta Math.}, 111:143--211, 1964.

\bibitem{W2}
A.~Weil.
\newblock Sur la formule de {S}iegel dans la th\'eorie des groupes classiques.
\newblock {\em Acta Math.}, 113:1--87, 1965.

\bibitem{Wr}
F.~Werner.
\newblock {\em Vector valued Hecke theory}.
\newblock PhD thesis, Technische {U}niversit\"at {D}armstadt, 2014.
\newblock Available at
  \url{https://tuprints.ulb.tu-darmstadt.de/4238/1/phd.pdf}.

\bibitem{Z}
D.~Zagier.
\newblock Modular forms whose {F}ourier coefficients involve zeta-functions of
  quadratic fields.
\newblock In {\em Modular functions of one variable, {VI} ({P}roc. {S}econd
  {I}nternat. {C}onf., {U}niv. {B}onn, {B}onn, 1976)}, volume Vol. 627 of {\em
  Lecture Notes in Math.}, pages 105--169. Springer, Berlin-New York, 1977.

\bibitem{Zh}
W.~Zhang.
\newblock {\em Modularity of {G}enerating {F}unctions of {S}pecial {C}ycles on
  {S}himura {V}arieties}.
\newblock PhD thesis, Columbia {U}niversity, 2004.
\newblock Available at
  \url{https://www.math.columbia.edu/~thaddeus/theses/2009/zhang.pdf}.

\end{thebibliography}

\end{document}